\documentclass[11pt, preprint]{imsart}
\RequirePackage[OT1]{fontenc}
\RequirePackage{amsthm,amsmath}
\RequirePackage[colorlinks,citecolor=blue,urlcolor=blue]{hyperref}
\usepackage{graphicx}
\usepackage{enumerate}
\usepackage{latexsym}
\usepackage{amssymb}
\usepackage{multirow}
\usepackage{float}
\usepackage{threeparttable}
\usepackage{epstopdf}
\usepackage{natbib}
\usepackage{enumitem}

{\end{quotation}}
\usepackage[normalem]{ulem}

\numberwithin{equation}{section}

\def\bbR{\mathbb{R}}

\def \kxi{K_{\mathbb{X}_i}}
\def \f1t{\mathcal{F}_1^T}
\def \d1t{\mathcal{D}_1^T}
\newcommand{\vo}{\vec{o}\@ifnextchar{^}{\,}{}}
\newcommand{\ip}[1]{\langle #1 \rangle} 
\startlocaldefs
\theoremstyle{plain}

\newtheorem{Theorem}{Theorem}[section]

\newtheorem{Corollary}[Theorem]{Corollary}
\newtheorem{Proposition}[Theorem]{Proposition}

\newtheorem{Assumption}{Assumption}

\newtheorem{lemma}{\indent Lemma}
\newtheorem{proposition}{\indent Proposition}
\newtheorem{Remark}{Remark}[section]
\theoremstyle{definition}

\makeatletter\def\theequation{\arabic{section}.\arabic{equation}}
\@addtoreset{equation}{section}\makeatother

\endlocaldefs

\begin{document}
\begin{frontmatter}

\title{Statistical Inference on Panel Data Models:
A Kernel Ridge Regression Method}

\runtitle{Panel Data Models with KRR}

\runauthor{Zhao et al.}

\author[add1]{Shunan Zhao}
\ead{szhao15@binghamton.edu}
\author[add2]{Ruiqi Liu}
\ead{rliu14@binghamton.edu}
\author[add2]{Zuofeng Shang\thanks{Corresponding author. Assistant Professor at Department of Mathematical Sciences, Binghamton University. Address: PO Box 6000, Binghamton, New York 13902-6000, USA; 
Email: zshang@binghamton.edu; Tel: (607)777-4263;
Fax: (607)777-2450. Research Sponsored by a start-up grant from Binghamton University.}}

\address[add1]{Department of Economics, Binghamton University}
\address[add2]{Department of Mathematical Sciences, Binghamton University}



\begin{abstract}
We propose statistical inferential procedures for
panel data models with interactive fixed effects in a 
kernel ridge regression framework. 
Compared with traditional sieve methods, 
our method is automatic in the sense that it does not require the choice of basis functions
and truncation parameters. Model complexity is controlled by a continuous regularization parameter
which can be automatically selected by generalized cross validation.
Based on empirical processes theory and functional analysis tools,
we derive joint asymptotic distributions for the estimators
in the heterogeneous setting. These joint asymptotic results are then used to construct
confidence intervals for the regression means and prediction intervals for the future observations,
both being the first provably valid intervals in literature.
Marginal asymptotic normality of the 
functional estimators in homogeneous setting is also obtained.
Simulation and real data analysis demonstrate 
the advantages of our method.

\vspace{1pt}
\vspace{1pt}


\end{abstract}


\end{frontmatter}

\section{Introduction}
Panel data models with interactive fixed effects (IFE) have many applications in econometrics and statistics,
e.g., individuals' education decision \citep{carneiroetal2003, cunhaetal2005}, house price analysis \citep{hollyetal2010}, prediction of investment returns \citep{eberhardtetal2013}, risk-return relation \citep{ludvigsonNg2009,ludvigsonNg2016}, etc.
The IFE can capture both cross section dependence and heterogeneity in the data,
which makes these models more flexible than the classic fixed or random effect models.
Cross section dependence is usually characterized by 
time-varying common factors, and heterogeneity is
captured by individual-specific factor loadings.
There has been an increasing literature in this area addressing
various statistical inferential problems.
Earlier studies focused on parametric settings;
see \cite{pesaran2006}, \cite{baiNg2006}, \cite{bai2009}, \cite{moonWeidner2010}, \cite{moonWeidner2015},
among others. A common crucial assumption in these papers
is that the response and predictor variables are
linearly related and the parameters of interest are finite-dimensional.

Recently, efforts were devoted to nonparametric/semiparametric panel data models with IFE.
For instance, \cite{freberger2012}, \cite{suJin2012}, \cite{jinandsu2013} and \citet{suZhang2013},
\citet{suJinzhang2015}, among others, 
proposed sieve methods for estimating or testing the infinite-dimensional regression functions;
\citet{huang2013} and \citet{caietal2016} proposed
local polynomial methods, differing from sieves by their local feature. 
The success of the sieve methods hinges on
a good choice of basis functions; see \cite{chen2007} 
for a comprehensive introduction. 
The truncation parameter, i.e., number of basis functions
used for model fitting, changes in a discrete fashion, which
may yield an imprecise control on the model complexity, as pointed out in \cite{RamseySilverman2005}.
For these reasons, it is worthwhile to explore a
method that relies less on the choice of basis
or the choice of discrete truncation parameters,
which will possess computational, theoretical
and conceptual advantages.

In this paper, we propose a new kernel-based nonparametric method,
called as kernel ridge regression (KRR), for handling
panel data models. Our method relies on the assumption
that the regression function belongs to a reproducing kernel Hilbert space (RKHS)
driven by a kernel function called as reproducing kernel.
The KRR estimator is ``basis-free" in the sense that it
can be explicitly expressed by the kernels rather than the basis functions.
Our method is applicable to a broad class of RKHS, 
e.g., Euclidean space, Sobolev space, Gaussian kernel space, or spaces with
advanced structures such as semiparametric/additive structures;
see \cite{W90}, \cite{SC13}, \cite{ZCL16} about more descriptions of these spaces. 
Reproducing kernels corresponding to the above mentioned
RKHS have (approximately) explicit expressions which can be directly used
in our inferential procedures. 
For applications of RKHS in other fields such as statistical machine learning,
see \citet{hofmann2008}.

In contrast to the method of sieves,
our KRR method does not involve the discrete truncation parameter.
Instead, it directly searches the estimator in the entire (possibly infinite-dimensional)
function spaces, though the process of searching requires the use
of a continuous regularization parameter which controls the smoothness
of the estimators. In smoothing splines or KRR,
regularization parameters are usually selected by generalized cross validation
(GCV); see \citet{cravenWahba1978}, \cite{W90}, \cite{Gu11}, \cite{wang2011}.  The selection procedure
proceeds by searching a global minimizer of the GCV criteria function
which provides a more accurate management on model complexity. 
As a result, the estimator may yield better performance
as observed by \cite{SC15} in functional data analysis.
In this paper, we adapt the traditional GCV criteria to 
semiparametric panel data models in both heterogeneous and homogeneous
settings.
The selection algorithms are easy-to-use with
satisfactory performance as demonstrated in our simulation study
and real data analysis (Section \ref{sec:numerical:study}).

Besides numerical advantages, the proposed method is theoretically
powerful. For instance, based on our RKHS framework, it is theoretically more convenient to
derive the joint asymptotic distributions for the estimators of the
linear and nonlinear components; see Theorems \ref{theorem: joint distr 1.}
 and \ref{theorem: joint distr 3.}. 
Nonetheless,
joint asymptotic distributions are more difficult to prove
in the sieve or local polynomial framework.
Our joint asymptotic results can be used to 
design novel statistical procedures
such as confidence intervals for the regression means 
and prediction intervals for the future observations,
though they can naturally imply the marginal asymptotic results
obtained by \cite{suJin2012} as a corollary.
The theory developed in this paper relies on nonstandard
technical tools such as empirical processes theory
and functional analysis.
Specifically, functional Bahadur representations (FBR)
in heterogeneous and homogeneous settings,
i.e., Theorems \ref{lemma: hetero -- leading term} and \ref{theorem: rate:ghat:minus:g0:plus:SMetag0}, 
are proved based on the aforementioned technical tools, 
which can extract the leading terms from the estimators.
Our FBR theory is a nontrivial extension of 
\cite{S10} and \cite{SC13} to panel data models,
which plays a central role in our theoretical study.

The rest of this paper is structured as follows. 
Section \ref{sec:prelim} contains some technical preliminaries
including an introduction to panel data models with IFE and an
RKHS framework.
Sections \ref{sec:model:preliminary} and \ref{sec: homo}
contain the main results.
Specifically, in heterogeneous setting, Section \ref{sec:model:preliminary}
includes estimation procedures for each individual parameters,
and derives joint asymptotic normality for the estimators.
Constructions of confidence interval and prediction interval
are also mentioned.
In homogeneous setting,
Section \ref{sec: homo} includes an estimation procedure
for the common regression function and derives its marginal asymptotic normality.
Section \ref{sec:numerical:study} examines the proposed methods
based on a simulation study and real data analysis.
Proofs of the main theorems are deferred to Section \ref{sec:appendix},
and proofs of other results are separated as a supplement document.
 
\section{Preliminary}\label{sec:prelim}
\subsection{Panel Data Models with Interactive Fixed Effects}\label{sec:PDM}
Let $Y_{it}$ be a real-valued observation and 
$X_{it}\in\mathcal{X}_i\subseteq\bbR^d$ be a real vector of
observed covariates, both collected on the $i$th unit at time $t$,
for $i\in[N]:=\{1,2,\ldots,N\}, t\in[T]:=\{1,\ldots,T\}$.
Suppose that the observations  
follow a semiparametric regression model
\begin{equation}\label{basic:model}
Y_{it}=g_i(X_{it})+\boldsymbol{\gamma}_{1i}'{f}_{1t}+\boldsymbol{\gamma}_{2i}'{f}_{2t}+\epsilon_{it},\,\,\,\,
i\in[N], t\in[T],
\end{equation}
where $g_i$ is an unknown regression function,
${f}_{1t}\in\bbR^{q_1}$ is a vector of observed common factors including intercept, 
${f}_{2t}\in\bbR^{q_2}$ is a vector of unobserved common factors,
$\boldsymbol{\gamma}_{1i}$ and $\boldsymbol{\gamma}_{2i}$ are unobserved fixed vectors called as
factor loadings, and $\epsilon_{it}$ are unobserved random noise. 
The term $\boldsymbol{\gamma}_{2i}'{f}_{2t}$ is called as interactive fixed effect.
In general, we allow $g_i$ to be varying across the units, i.e., the panels demonstrate a
heterogeneous structure. The special case $g_i=g$ for all $i\in[N]$ implies that the panels
are homogeneous. In this paper, we will consider both cases
from theoretical and methodological aspects.
Since the factors in model (\ref{basic:model}) are not identifiable in the sense that they cannot
be consistently estimated, further constraints are needed. 
Similar to \cite{pesaran2006}, suppose that $X_{it}$ are related to the factors
through the following data generating equation:
\begin{equation} \label{DGP of X}
X_{it}=\Gamma_{1i}'{f}_{1t} + \Gamma_{2i}'{f}_{2t} + {v}_{it},\,\,\,\, i\in[N], t\in[T],
\end{equation}
where $\Gamma_{1i}\in\bbR^{q_1 \times d}$ and $\Gamma_{2i}
\in\bbR^{q_2 \times d}$ are unobserved but fixed matrices,
${v}_{it}\in\bbR^{d}$ is a vector of random noise.
Other constraints to guarantee identifiability were proposed by 
\cite{bai2009} based on principle component analysis.
Equation (\ref{DGP of X}) provides a convenient way to remove the unobserved factor $f_{2t}$
from model (\ref{basic:model}).
Specifically, averaging (\ref{DGP of X}) as done by \cite{pesaran2006}, we get
\begin{align}\label{DGP of X 2}
\bar{X}_t  = \bar{\Gamma}_1' {f}_{1t} + \bar{\Gamma}_2' {f}_{2t} + \bar{{v}}_t,
\,\,\,\,t\in[T],
\end{align}
where $\bar{X}_t=\frac{1}{N}\sum_{i=1}^{N}X_{it}$, $\bar{\Gamma}_1=\frac{1}{N}\sum_{i=1}^{N}\Gamma_{1i}$, $\bar{\Gamma}_2=\frac{1}{N}\sum_{i=1}^{N}\Gamma_{2i}$, and $\bar{v}_t=\frac{1}{N}\sum_{i=1}^{N}v_{it}$. Throughout we assume
that $\bar{\Gamma}_2 \bar{\Gamma}_2'$ is invertible which may hold true if $q_2 \leq d$.
It follows from (\ref{DGP of X 2}) that 
\begin{equation}\label{DGP X 3}
{f}_{2t}=(\bar{\Gamma}_2 \bar{\Gamma}_2^{\prime})^{-1} \bar{\Gamma}_2 \left(
\bar{X}_t-\bar{\Gamma}_1' {f}_{1t}  - \bar{{v}}_t \right).
\end{equation}
Replacing ${f}_{2t}$ in (\ref{basic:model}) by (\ref{DGP X 3}), 
we get the following
\begin{align} \label{basic:model 2}
Y_{it} & = g_i(X_{it}) + \gamma_{1i}'{f}_{1t}+\gamma_{2i}' (\bar{\Gamma}_2\bar{\Gamma}_2')^{-1}\bar{\Gamma}_2\left(
\bar{X}_t-\bar{\Gamma}_1'{f}_{1t}-\bar{v}_t
\right)+\epsilon_{it} \nonumber \\
&=g_i(X_{it})+Z_t'\beta_i+e_{it},\,\,\,\,i\in[N], t\in[T],
\end{align}
where $Z_t=({f}_{1t}', \bar{X}_t')'$, 
$e_{it}=\epsilon_{it}-\gamma_{2i}'(\bar{\Gamma}_2\bar{\Gamma}_2')^{-1}\bar{\Gamma}_2\bar{v}_t$, and 
\[
\beta_i={
\gamma_{1i}-\bar{\Gamma}_{1}\bar{\Gamma}_2'(\bar{\Gamma}_2\bar{\Gamma}_2')^{-1}\gamma_{2i}
\choose \bar{\Gamma}_2'(\bar{\Gamma}_2\bar{\Gamma}_2')^{-1}\gamma_{2i}.
}
\]
The equations (\ref{basic:model}),
(\ref{DGP of X}) and (\ref{basic:model 2}) play an important role in the proof of our main results. 

\subsection{Kernel Ridge Regression}\label{sec:prelim:RKHS}
Suppose $g_i\in\mathcal{H}_i$, where $\mathcal{H}_i$ is
a Reproducing Kernel Hilbert Space (RKHS).
Specifically, $\mathcal{H}_i$ is a Hilbert space of real-valued functions
on $\mathcal{X}_i$, endowed with an inner product $\langle\cdot,\cdot\rangle_{\mathcal{H}_i}$,
satisfying
the property: for any $x\in\mathcal{X}_i$, there exists a unique element 
$\bar{K}^{(i)}_x\in\mathcal{H}_i$ such that for every $g\in\mathcal{H}_i$, 
$\langle \bar{K}^{(i)}_x,g\rangle_{\mathcal{H}_i}=g(x)$.
The reproducing kernel function 
is defined by $\bar{K}^{(i)}(x_1,x_2)=\bar{K}^{(i)}_{x_1}(x_2)$, for any $x_1,x_2\in\mathcal{X}_i$.
The kernel $\bar{K}^{(i)}$ is symmetric, i.e., $\bar{K}^{(i)}(x_1,x_2)=\bar{K}^{(i)}(x_2,x_1)$,
and the matrix $\bar{\mathcal{K}}^{(i)}:=[\bar{K}(x_i,x_i)]_{i,j=1}^n$ is semi-positive definite
for any $x_1,\ldots,x_n\in\mathcal{X}_i$. \cite{BT04}
provides a nice introduction to RKHS.

By Mercer's theorem, $\bar{K}^{(i)}$ admits a spectral expansion:
\begin{equation}\label{kernel:expansion}
\bar{K}^{(i)}(x_1,x_2)=\sum_{\nu=1}^\infty\varphi^{(i)}_\nu(x_1)\varphi^{(i)}_\nu(x_2)/\rho^{(i)}_\nu,
\,\,x_1,x_2\in\mathcal{X}_i,
\end{equation}
where $0<\rho^{(i)}_1\le\rho^{(i)}_2\le\cdots$ are eigenvalues and $\varphi^{(i)}_\nu$ are eigenfunctions
which form an $L^2(P_{X_i})$ orthonormal basis with
$X_i\equiv X_{i1}$. This paper focuses on RKHS generated by the following kernels.
For simplicity, we use $a_n \asymp b_n$ to represent $a_n=O(b_n)$ and $b_n=O(a_n)$.

\textit{Finite Rank Kernel} (FRK): The kernel $\bar{K}^{(i)}$ is said to have rank $k>0$ if
$\rho^{(i)}_\nu=\infty$ for $\nu>k$. For instance, 
the $(k-1)$-order polynomial kernel $\bar{K}^{(i)}(x_1,x_2)=(1+x_1'x_2)^{k-1}$ for $x_1,x_2\in\bbR^d$
has rank $k$. Clearly, an FRK of rank $k$ corresponds to a parametric space of dimension $k$.

\textit{Polynomially Diverging Kernel} (PDK): The kernel $\bar{K}^{(i)}$ is said to be polynomially diverging
of order $k>0$ if it
has eigenvalues satisfying $\rho^{(i)}_\nu\asymp\nu^{2k}$ for $\nu\ge1$.
For instance, the $k$-order Sobolev space
is an RKHS with a kernel polynomially diverging of order $k$; see \cite{W90}.

\textit{Exponentially Diverging Kernel} (EDK): 
The kernel $\bar{K}^{(i)}$ is said to be exponentially diverging 
of order $k>0$ if its eigenvalues satisfy $\rho^{(i)}_\nu\asymp\exp(b\nu^k)$ for $\nu\ge1$,
for a constant $b>0$. For instance, Gaussian kernel $\bar{K}^{(i)}(x_1,x_2)=\exp(-|x_1-x_2|^2)$
corresponds to $k=1$; see \cite{LCL17}.

The results of this paper can be applied to more complicated RKHS such as \textit{Additive RKHS},
as described in Remark \ref{rem:additive:rkhs}. 

Let $\Theta_i:=\bbR^{q_1+d}\times\mathcal{H}_i$.
We estimate $\theta=(\beta,g)\in\Theta_i$ via the following \textit{Kernel Ridge Regression}:
\begin{eqnarray} \label{eq: hetero-likelihood fn}
\widehat{\theta}_i =(\widehat{\beta}_i, \widehat{g}_i)
&=&\arg\min_{\theta\in {\Theta}_i}\ell_{i,M,\eta_i}(\theta)
\nonumber\\
&\equiv&\arg\min_{\theta\in {\Theta}_i}\left\{\frac{1}{2T}\sum_{t=1}^T(Y_{it}-g(X_{it})-Z_t'\beta)^2+\frac{\eta_i}{2}
\|g\|_{\mathcal{H}_i}^2\right\},
\end{eqnarray}
where $M=(N,T)$ and $\eta_i>0$ is called as a regularization parameter. 

Our results will rely on an RKHS structure on 
$\Theta_i$.
Specifically, we will follow \cite{CS15} to construct two operators
$R_i:\mathcal{U}_i\to\Theta_i$ and $P_i:\Theta_i\to\Theta_i$, where
$\mathcal{U}_i\equiv \{u=(x,z): x\in\mathcal{X}_i, z\in\bbR^{q_1+d}\}$,
such that for any 
$u=(x,z)\in\mathcal{U}_i$ and 
$\theta=(\beta,g)\in\Theta_i$, the following holds:
\begin{eqnarray}\label{semi:reproducing:prop}
\langle R_iu,\widetilde{\theta}\rangle_i=\widetilde{g}(x)+z'\widetilde{\beta},\,\,\,\,
\langle P_i\theta,\widetilde{\theta}\rangle_i=\eta_i\langle g,\widetilde{g}\rangle_{\mathcal{H}_i},\,\,
\,\,\textrm{for any $\widetilde{\theta}=(\widetilde{\beta},\widetilde{g})\in\Theta_i$},
\end{eqnarray}
where $\langle\cdot,\cdot\rangle_i$ is an inner product on $\Theta_i$
to be defined later in (\ref{inner:product:Theta:i}). 

For any $x\in\mathcal{X}_i$, define $G_i(x)=E\{Z|X_i=x\}$ and $\Omega_i=E\{(Z-G_i(X_i))(Z-G_i(X_i))'\}$,
where $Z=Z_1$.
Clearly, $\Omega_i$ is a square matrix of dimension $q_1+d$.
In the below we require the eigenvalues of $\Omega_i$ to be bounded away from zero and infinity,
a standard condition to guarantee semiparametric efficiency;
see, e.g., \cite{MVG97,CS15}. Besides, $G_i$ are assumed to be $L^2$ integrable.
\begin{Assumption}\label{A3} For $i\in[N]$, $G_i\in L^2(P_{X_i})$.
Furthermore, 
$c_1\le \lambda_{\min}(\Omega_i)\le\lambda_{\max}(\Omega_i)\le c_2$ 
for positive constants $c_1,c_2$,
where $\lambda_{\min}(\cdot)$ and $\lambda_{\max}(\cdot)$ are minimal and maximal eigenvalues.
\end{Assumption}

For any $\theta_k=(\beta_k,g_k)\in\Theta_i$, $k=1,2$,
define
\begin{equation}\label{inner:product:Theta:i}
\langle\theta_1,\theta_2\rangle_i=E\{(g_1(X_i)+Z'\beta_1)(g_2(X_i)+Z'\beta_2)\}+\eta_i\langle
g_1,g_2\rangle_{\mathcal{H}_i}.
\end{equation}
Define $\langle g_1,g_2\rangle_{\star,i}=\langle(0,g_1),(0,g_2)\rangle_i$.
Following \citet{CS15}, Assumption \ref{A3} implies that
$\langle\cdot,\cdot\rangle_i$ and $\langle\cdot,\cdot\rangle_{\star,i}$ 
are valid inner products on $\Theta_i$ and $\mathcal{H}_i$, respectively.
Meanwhile, $(\mathcal{H}_i,\langle\cdot,\cdot\rangle_{\star,i})$ is an RKHS 
with kernel $K^{(i)}(x,y)\equiv\sum_{\nu=1}^\infty\varphi_\nu^{(i)}(x)
\varphi_\nu^{(i)}(y)/(1+\eta_i\rho_\nu^{(i)}),
\,\,x,y\in\mathcal{X}_i$.
We can further find a positive definite self-adjoint operator $W_i: \mathcal{H}_i\to\mathcal{H}_i$
and an element $A_i\in\mathcal{H}_i^{q_1+d}$ such that
$\langle W_i g_1, g_2\rangle_{\star,i}=\eta_i\langle g_1,g_2\rangle_{\mathcal{H}_i}$,
$\langle A_i,g\rangle_{\star,i}=V_i(G_i,g)$, for any $g,g_1,g_2\in\mathcal{H}_i$,
where $V_i(g_1,g_2)=E\{g_1(X_i)g_2(X_i)\}$.
Define $\Sigma_i=E\{G_i(X_i)(G_i(X_i)-A_i(X_i))'\}$, a symmetric matrix of dimension
$q_1+d$. Following \cite{CS15}, Proposition \ref{prop:Ri:Pi} below guarantees (\ref{semi:reproducing:prop}).
\begin{Proposition}\label{prop:Ri:Pi}
For any $u=(x,z)\in\mathcal{U}_i$ and for any
$\theta=(\beta,g)\in\Theta_i$, (\ref{semi:reproducing:prop}) holds for $R_iu=(H^{(i)}_u,T^{(i)}_u)$
and $P_i\theta=(H^{(\star i)}_g,T^{(\star i)}_g)$, where
\begin{eqnarray*}
H^{(i)}_u&=&(\Omega_i+\Sigma_i)^{-1}(z-V_i(G_i,K^{(i)}_x)),\\
T^{(i)}_u&=&K^{(i)}_x-A_i'(\Omega_i+\Sigma_i)^{-1}(z-V_i(G_i,K^{(i)}_x)),\\
H^{(\star i)}_g&=&-(\Omega_i+\Sigma_i)^{-1}V_i(G_i,W_ig),\\
T^{(\star i)}_g&=&W_ig+A_i'(\Omega_i+\Sigma_i)^{-1}V_i(G_i,W_ig).
\end{eqnarray*}
\end{Proposition} 

A direct application of Proposition \ref{prop:Ri:Pi} is to exactly calculate
the Fr\'{e}chet derivatives of $\ell_{i,M,\eta_i}$. Define $U_{it}=(X_{it},Z_t)$ for $i\in[N]$, $t\in[T]$.
For $\theta=(\beta,g),\Delta\theta=(\Delta\beta,\Delta g)\in\Theta_i$, we have
\begin{eqnarray*}
D\ell_{i,M,\eta_i}(\theta)\Delta\theta&=&\langle -\frac{1}{T}\sum_{t=1}^T(Y_{it}-\langle R_iU_{it},\theta\rangle_i)R_iU_{it}+P_i\theta,
\Delta\theta\rangle_i
\equiv\langle S_{i,M,\eta_i}(\theta),\Delta\theta\rangle_i,\\
DS_{i,M,\eta_i}(\theta)\Delta\theta&=&\frac{1}{T}\sum_{t=1}^T\langle R_iU_{it},\Delta\theta\rangle_iR_iU_{it}+P_i\Delta\theta,\\
D^2S_{i,M,\eta_i}(\theta)&=&0.
\end{eqnarray*}

\section{Heterogeneous Model}\label{sec:model:preliminary}
In this section, we consider heterogeneous model (\ref{basic:model})
where the $g_i$'s are assumed to be different across the units.
We will estimate each $g_i$ through penalized estimations,
and develop a joint asymptotic theory for the estimators.
Our joint asymptotic results lead to novel
statistical procedures as well as rediscover the existing 
marginal asymptotic results.

\subsection{Estimation Procedure}

By representer theorem \citep{W90}, the minimizer $\widehat{g}_i$ of (\ref{eq: hetero-likelihood fn})
has the expression
\begin{equation}\label{represent:thm:heter}
g(x) = \sum_{t=1}^{T} a_{t} \bar{K}^{(i)}(X_{it},x) = a' \bar{K}_x^{(i)},
\,\,\,\,x\in\mathcal{X}_i,
\end{equation}
where $a=(a_{1}, \cdots, a_{T})'$ and $\bar{K}^{(i)}_x=(\bar{K}^{(i)}(X_{i1},x), \cdots, \bar{K}^{(i)}(X_{iT},x))'$. Based on (\ref{represent:thm:heter}) we have
\begin{equation}\label{hetero:g:H:i}
\|g\|_{\mathcal{H}_i}^2 = \langle \sum_{t=1}^{T} a_{t} \bar{K}^{(i)}(X_{it},\cdot), \sum_{t=1}^{T} a_{t} \bar{K}^{(i)}(X_{it},\cdot) \rangle_{\mathcal{H}_i} = a' \ \bar{\mathcal{K}}^{(i)} \ a,
\end{equation}
where $\bar{\mathcal{K}}^{(i)}= ( \bar{K}^{(i)}_{X_{i1}}, \cdots, \bar{K}^{(i)}_{X_{iT}} )\in\bbR^{T\times T}$
is semi-positive definite. 
So, \eqref{eq: hetero-likelihood fn} can be equally transformed to the following:
\begin{align}\label{krr:optimization:function}
(\widehat{a}_i,\widehat{\beta}_i)= & 
\arg\min_{a\in\bbR^T,\beta\in\bbR^{q_1+d}}
\frac{1}{2T} \left(Y_i- \bar{\mathcal{K}}^{(i)} a - Z \beta \right)'\left(Y_i- \bar{\mathcal{K}}^{(i)} a - Z \beta \right) +\frac{\eta_i}{2} 
 a' \bar{\mathcal{K}}^{(i)} a,
\end{align}
where $Y_i=(Y_{i1}, \cdots, Y_{iT})'$ and $Z=(Z_1, \cdots, Z_T)'$. 
The solution to (\ref{krr:optimization:function}) has expression
\begin{align}
&\widehat{a}_i = \left( \left( I_T - Z(Z'Z)^{-1}Z' \right) \bar{\mathcal{K}}^{(i)} + T\eta_i I_T \right)^{-1} \left( I_T - Z(Z'Z)^{-1}Z' \right) Y_i, \nonumber \\
&\widehat{\beta}_i = \left( Z'Z \right)^{-1}Z' \left( Y_i - \bar{\mathcal{K}}^{(i)} \widehat{a}_i \right).
\end{align}
Then we estimate $g_i$ by $\widehat{g}_i(x) = \widehat{a}_i' \bar{K}_x^{(i)}$
for any $x\in\mathcal{X}_i$.

\begin{Remark}\label{gcv:1}
In practice, we choose $\eta_i$ by minimizing the following GCV function:
\begin{equation*}
\widehat{\eta}_i=\arg\min_{\eta_i>0}\textrm{GCV}_i(\eta_i)\equiv
\arg\min_{\eta_i>0}\frac{\|(I_{T}-B_{\eta_i})Y_i\|_2^2}{T[1-tr(B_{\eta_i})/(T)]^2},
\end{equation*}
where $B_{\eta_i}$ is the so-called smoothing matrix,
i.e., a $T\times T$ matrix satisfying $\widehat{Y}_i=B_{\eta_i} Y_i$,
where $\widehat{Y}_i=\bar{\mathcal{K}}^{(i)}\widehat{a}_i+Z\widehat{\beta}_i$
is the fitted response vector.
\end{Remark}


\subsection{Rate of Convergence} \label{subsec:consist-Hetero}
We will derive the rate of convergence for $\widehat{\theta}_i$.
Before that, let us assume some technical conditions.
For $t\ge j\ge 1$, define 
$\mathcal{F}_j^t=\sigma\left({f}_{1l},{f}_{2l}: j\le l\le t\right)$.
Define $\phi$-mixing coefficients
\[
\phi(t)=\sup_{t_1\ge1}\sup_{\substack{A\in \mathcal{F}_1^{t_1}
B\in\mathcal{F}_{t_1+t}^\infty\\
P(B)>0}}
\big|P(A|B)-P(A)\big|,\,\,\,\,t\ge0.
\]
 
\begin{Assumption}\label{A1}
\begin{enumerate}[label=(\alph*)]
\item\label{A1:a} $\{v_{it}: i\in[N], t\in[T]\}$ are i.i.d.,
and $\{\epsilon_{it}: i\in[N], t\in[T]\}$ are i.i.d., both of zero means.
Furthermore, $v_{it}$'s and $\epsilon_{it}$'s are independent.  

\item\label{A1:b} Both $\{{f}_{1t}:t\in[T]\}$ and $\{{f}_{2t}:t\in[T]\}$
are strictly stationary process satisfying the following $\phi$-mixing condition:
$\sum_{t=0}^\infty\phi(t)^{1-4/\alpha}<\infty$, 
where $\alpha>4$ is a constant. $(f_{1t}, f_{2t})$ is distributed independently of $v_{it}$'s and $\epsilon_{it}$'s.

\item\label{A1:c} 
$E\{\|{f}_{1t}\|_2^\alpha\}<\infty$, $E\{\|{f}_{1t}\|_2^\alpha\}<\infty$,
$E\{|v_{it}|^\alpha\}<\infty$, $E\{|\epsilon_{it}|^\alpha\}<\infty$,
where $\|\cdot\|_2$ denotes the Euclidean norm.

\item\label{A1:d} $\sup_{i\ge1}\|\Delta_i\|_2<\infty$, where 
$\Delta_i=\gamma_{2i}'(\bar{\Gamma}_2\bar{\Gamma}_2')^{-1}\bar{\Gamma}_2$.

\end{enumerate}
\end{Assumption}

Assumption \ref{A1}\ref{A1:a} requires that the variables $\varepsilon_{it},v_{it}$ are zero-mean
independent. Assumption \ref{A1}\ref{A1:b} specify that the factors are strictly stationary and $\phi$-mixing, and independent of $v_{it}$'s and $\epsilon_{it}$'s. Independence assumption 
can be relaxed to mixing conditions with more tedious technical arguments.
Assumption \ref{A1}\ref{A1:c} requires that the variables have finite $\alpha$-moments. Assumption 
\ref{A1}\ref{A1:d} requires that the vectors $\Delta_i$ based on 
``true" factor loadings are uniformly bounded.

The following assumption says that $\varphi^{(i)}_\nu$ are uniformly bounded
and $\varphi^{(i)}_\nu,\rho^{(i)}_\nu$ simultaneously diagonalize $V_i$ and $\langle\cdot,\cdot\rangle_{\mathcal{H}_i}$,
a standard assumption in kernel ridge regression literature, e.g., \cite{SC13}. 
This condition holds for polynomially diverging kernels, exponentially diverging kernels,
or finite rank kernels on compactly supported $\mathcal{X}_i$;
see, e.g., \cite{W90,SC13,ZCL16}. 
Besides, we need $1\notin\mathcal{H}_i$ for identifiability.

\begin{Assumption}\label{A2}
For any $i\in[N]$,
$\sup_{\nu\ge1}\sup_{x\in\mathcal{X}_i}|\varphi_\nu^{(i)}(x)|<\infty$
and 
\[
V_i(\varphi^{(i)}_\nu,\varphi^{(i)}_\mu)=\delta_{\nu\mu},\,\,\,\,
\langle\varphi^{(i)}_\nu,\varphi^{(i)}_\mu\rangle_{\mathcal{H}_i}=\rho_\nu^{(i)}\delta_{\nu\mu},
\,\,\,\,\nu,\mu\ge1,
\]
where $\delta_{\nu\mu}$ is the Kronecker's delta. Furthermore,
$1\notin\mathcal{H}_i$,
and any $g\in\mathcal{H}_i$ satisfies $g=\sum_{\nu=1}^\infty g_\nu\varphi^{(i)}_\nu$,
where $g_\nu=V_i(g,\varphi^{(i)}_\mu)$ is a real sequence satisfying $\sum_{\nu=1}^\infty\rho_\nu^{(i)}g_\nu^2<\infty$.
\end{Assumption}

For any $\theta=(\beta,g)\in\Theta_i$, define $\|\theta\|_{i,\sup}=\sup_{x\in\mathcal{X}_i}|g(x)|+\|\beta\|_2$. 
For $p,\delta>0$,
define $\mathcal{G}_i(p)=\{\theta=(\beta,g)\in\Theta_i: \|\theta\|_{i,\sup}\le1,\|g\|_{\mathcal{H}_i}\le p\}$
and the corresponding entropy integral
\[
J_i(p,\delta)=\int_0^\delta\psi_2^{-1}\left(D_i(\varepsilon,\mathcal{G}_i(p),\|\cdot\|_{i,\sup})\right)d\varepsilon
+\delta\psi_2^{-1}\left(D_i(\delta,\mathcal{G}_i(p),\|\cdot\|_{i,\sup})^2\right),
\]
where $\psi_2(s)=\exp(s^2)-1$ and $D_i(\varepsilon,\mathcal{G}_i(p),\|\cdot\|_{i,\sup})$
is the $\varepsilon$-packing number of $\mathcal{G}_i(p)$
in terms of $\|\cdot\|_{i,\sup}$-metric. 
Let $\theta_{i0}=(\beta_{i0},g_{i0})$ denote the ``true" value of $(\beta,g)$
in (\ref{basic:model 2}). 
Define
\[
h_i=\sum_{\nu=1}^\infty (1+\eta_i\rho_\nu^{(i)})^{-1},
\,\,\,\,r_{i,M}=(Th_i)^{-1/2}+\eta_i^{1/2}+(Nh_i)^{-1/2}.
\]
It can be shown that 
$h_i\asymp\eta_i^{1/(2k)}$ for $k$-order PDK;
$h_i\asymp(\log(1/\eta_i))^{-1/k}$ for $k$-order EDK.
We use $(N, T) \rightarrow\infty$ to represent 
both $N \rightarrow\infty$ and $T \rightarrow\infty$.
\begin{Theorem}\label{them:convergence:rate:result:1}
Suppose that Assumptions \ref{A3}--\ref{A2} are satisfied.
Furthermore, as $(N, T)\rightarrow\infty$, the following conditions hold:
\begin{eqnarray}\label{Rate:Cond:0}
&&\eta_i=o(1),\,\,\,\,h_i=o(1), \,\,\,\, T^{2/\alpha-1}h_i^{-1}=o(1), \,\,\,\, \eta_i+ \frac{1}{N h_i} = O(h_i),\nonumber\\
&&T^{-1/2+1/\alpha}h_i^{-1/2}\max\{h_i^{-1/2},T^{1/\alpha}\}
J_i((\eta_i^{-1}h_i)^{1/2},1)\nonumber\\
&&\times\sqrt{\log{N}+\log\log(TJ_i((\eta_i^{-1}h_i)^{1/2},1))}=o(1).
\end{eqnarray}
Then for any $i\in[N]$,
$\|\widehat{\theta}_i-\theta_{i0}\|_i=O_P(r_{i,M})$, as 
$(N, T)\rightarrow\infty$.
\end{Theorem} 

Theorem \ref{them:convergence:rate:result:1} presents a rate of convergence 
for $\widehat{\theta}_i$ in
$\|\cdot\|_i$-norm which relies on $h_i,N,T$.  
The $\|\cdot\|_i$-norm is stronger than the commonly used $L^2$-norm 
in literature; see \cite{suJin2012}, \cite{suZhang2013}.
The optimal choice of $h_i$,
denoted $h_i^\star$, 
relies on the type of kernels
and the relationship of $N,T$, i.e., $N\ge T$ or $N<T$. 
The optimal convergence rate, denoted $r^\star_{i,M}$, can be calculated accordingly;
see Table \ref{table:rate:hetero}.
We observe that $h_i^\star,r^\star_{i,M}$ only depend on the smaller value of $N,T$
in both PDK and EDK. Rate conditions (\ref{Rate:Cond:0}) are satisfied
under $h_i\asymp h_i^\star$.

\begin{table}[htp]
\begin{center}
\begin{tabular}{ccccc}
&\multicolumn{2}{c}{$N\ge T$}&\multicolumn{2}{c}{$N<T$}\\ \cline{2-3} \cline{4-5}
& PDK & EDK & PDK & EDK\\ \hline
$h^\star_i$ &$T^{-1/(2k+1)}$&$(\log{T})^{-1/k}$&$N^{-1/(2k+1)}$&$(\log{N})^{-1/k}$\\
$r^\star_{i,M}$&$T^{-k/(2k+1)}$ & $T^{-1/2}(\log{T})^{1/(2k)}$ 
&$N^{-k/(2k+1)}$ & $N^{-1/2}(\log{N})^{1/(2k)}$\\ \hline
\end{tabular}
\caption{A summary of $h^\star_i$ and $r^\star_i$ in $k$-order PDK and $k$-order EDK.}
\label{table:rate:hetero}
\end{center}
\end{table}

\subsection{Joint Asymptotic Distribution} \label{subsec:AsympDist-Hetero}

The aim of this section is to derive
joint asymptotic normality for $\widehat{\theta}_i=(\widehat{\beta}_i,\widehat{g}_i)$
which is new in literature. Our result
naturally leads to marginal asymptotic normality obtained by \cite{suJin2012}.
More importantly, our joint asymptotic result can be used to construct 
confidence interval for regression mean 
and prediction interval for future response variable.
As far as we know, these are the first asymptotic valid
intervals for prediction purposes.
Before proceeding further, we give a technical result, called as Functional Bahadur Representation (FBR),
to characterize the leading term of the estimator. 

\begin{Theorem}[FBR for Heterogeneous Model] \label{lemma: hetero -- leading term}
Suppose that Assumptions \ref{A3}--\ref{A2} are satisfied. Then we have 
\begin{equation}
\| \widehat{\theta}_i -\theta_{i0} + S_{i,M,\eta_i}(\theta_{i0}) \|_i = O_P(a_M),\,\,\,\,
\textrm{as $(N,T)\rightarrow\infty$,}
\end{equation}
where 
$
a_M=r_{i,M}h_i^{-1/2} T^{1/\alpha-1/2}(h_i^{-1/2}+T^{1/\alpha})J_i(p_i,1)\sqrt{\log{N}+\log\log(TJ_i(p_i,1))} + h_i^{-1/2}T^{2/\alpha-1}
$
and $p_i=(\eta_i^{-1}h_i)^{1/2}$. 
\end{Theorem}

Theorem \ref{lemma: hetero -- leading term} 
provides a higher order approximation for
$\widehat{\theta}_i -\theta_{i0}$ with leading term $S_{i,M,\eta_i}(\theta_{i0})$,
which generalizes \cite{S10} and \cite{SC13} to panel data settings.
The rate of the remainder term $a_M$ can be shown to be $O(N^{-1/2})$
or $O(T^{-1/2})$ if we choose $h_i\asymp h_i^\star$, where the values of $h_i^\star$
are summarized in Table \ref{table:rate:hetero}. Thus, the remainder term
is asymptotically negligible compared to $S_{i,M,\eta_i}(\theta_{i0})$. 
This lemma can be used to prove the following joint asymptotic normality 
for $\widehat{\theta}_i$. The proof relies on a central limit theorem
on $S_{i,M,\eta_i}(\theta_{i0})$.
Let $\theta_{i0}^\star=(\beta_{i0}^\star, g_{i0}^\star)\equiv(id-P_i)\theta_{i0}$.
\begin{Theorem} \label{theorem: joint distr 1.}
Suppose that Assumptions \ref{A3}--\ref{A2} are all satisfied. 
Furthermore, $h_i=o(1)$, $(Nh_i)^{-1}=o(1)$, $a_M=o(T^{-1/2}h_i^{1/2})$, 
and for $x_0 \in\mathcal{X}_i$, $h_iV_i(K^{(i)}_{x_0}, K^{(i)}_{x_0})
\to \sigma_{x_0}^2$, $h_i^{1/2}(W_iA_i)(x_0) \to \alpha_{x_0}$, and $h_i^{1/2}A_{i}(x_0)\to -\beta_{x_0}$, where $\sigma_{x_0}^2>0,\alpha_{x_0},\beta_{x_0}\in\bbR^{q_1+d}$
are nonrandom constants. Then, 
\begin{equation}
\left(
\begin{array}{c}
\sqrt{T}(\widehat{\beta}_i - \beta_{i0}^\star) - \sqrt{T}E\{e_{it}H^{(i)}_{U_{it}}\}\\
\sqrt{Th_i}(\widehat{g}_i(x_0) - g_{i0}^\star(x_0)) - \sqrt{Th_i} E\{e_{it} T_{U_{it}}^{(i)}(x_0)\}
\end{array}
\right) \overset{d}{\to} N\left(0 , \Psi^\star \right),\,\,\,\,\textrm{as $(N,T)\rightarrow\infty$,}
\end{equation} \label{eq: hetero- joint dist 1}
where 
\begin{equation}
\Psi^\star= 
\sigma^2_{\epsilon} \left(
\begin{array}{cc}
\Omega_i^{-1} & \Omega_i^{-1}(\alpha_{x_0} + \beta_{x_0}) \\
(\alpha_{x_0} + \beta_{x_0})'\Omega_i^{-1} & \sigma_{x_0}^2 + 2\beta_{x_0}'\Omega_i^{-1}\alpha_{x_0} + \beta_{x_0}'\Omega_i^{-1}\beta_{x_0}
\end{array}
\right),\,\,\textrm{and}\,\, 
\sigma^2_{\epsilon} = Var(\epsilon_{it}).
\end{equation}
\end{Theorem}
Theorem \ref{theorem: joint distr 1.} proves joint asymptotic normality
for $\widehat{\beta}_i$ and $\widehat{g}_i(x_0)$. The estimators are nonetheless not (asymptotically)
unbiased, i.e., they do not converge to the truth $\beta_{i0}$ and $g_{i0}(x_0)$.
To correct the bias, we need to assume $T/N=o(1)$
as in the following Theorem \ref{theorem: joint distr 3.}. 
This condition means that the number of observations within each individual unit
is strictly smaller than the number of units, which can provide
more cross section information.
We expect that the bias cannot be corrected if $T\ge N$.
Indeed, our theoretical analysis indicates 
a possibly sharp upper bound $\sqrt{T}E\{e_{it}H^{(i)}_{U_{it}}\}=O_P(\sqrt{T/N})$.
When $T\ge N$, this term will result in uncorrectable bias in estimating $\beta_i$.
Relevant assumptions exist in literature for bias correction,
e.g., $T/N^2=o(1)$ considered by \citet{pesaran2006} in parametric setting; $\kappa T/N=o(1)$ 
considered by \citet{suJin2012} in sieve estimation,
where $\kappa$ represents the discrete truncation parameter (or, number of basis functions).
Compared to the latter, our condition is weaker.
Another condition for bias correction is that
$G_i$ is sufficiently smooth, i.e., condition (\ref{smooth:G:i})
in Theorem \ref{theorem: joint distr 3.}.
Such condition holds if
the conditional distribution of the factor variable $f_{1t}$ given 
$X_{it}$ is smooth. As a by-product, $\widehat{\beta}_i$ and $\widehat{g}_i(x_0)$ 
become asymptotically independent which facilitates the applications,
e.g., one does not need to estimate the correlation between the two estimators. Define $G_i=(G_{i,1}, \dots , G_{i,q_1+d})'$. 
\begin{Theorem} \label{theorem: joint distr 3.}
Suppose that the conditions in Theorem \ref{theorem: joint distr 1.} hold, ${T\eta_i}=o(1)$
and $T/N=o(1)$. Furthermore, there exists a positive non-decreasing sequence 
$k_\nu$ with $\sum_{\nu\geq 1}1/k_\nu<\infty$ such that
\begin{equation}\label{smooth:G:i}
\sum_{\nu} |V_i(G_{i,k}, \varphi^{(i)}_{\nu})|^2 k_\nu< \infty, \qquad \textrm{for $k =  1,\ldots,q_1+d$.}
\end{equation}
Then we have, for any $x_0 \in\mathcal{X}_i$, 
\begin{equation}
\left(
\begin{array}{c}
\sqrt{T}(\widehat{\beta}_i - \beta_{i0}) \\
\sqrt{Th_i}(\widehat{g}_i(x_0) - g_{i0}(x_0) )
\end{array}
\right) \overset{d}{\to} N\left(0 , \Psi \right),
\,\,\,\,\textrm{as $(N,T)\rightarrow\infty$,}
\end{equation}
where 
\begin{equation}
\Psi= 
\sigma_\epsilon^2\left(
\begin{array}{cc}
\Omega_i^{-1} & 0 \\
0 & \sigma_{x_0}^2 
\end{array}
\right).
\end{equation}

\end{Theorem}

An application of Theorem \ref{theorem: joint distr 3.} is 
the construction of confidence interval for regression mean. 
Suppose $X_{i\, t+1}=x_{i0}$ and $f_{1\, t+1}=f_{10}$ with known $x_{i0}$ and $f_{10}$,
i.e., the predictor variables of each individual are observed at future time $t+1$.
By (\ref{basic:model 2}),
the conditional mean of $Y_{i\, t+1}$ is $\mu_{i0}\equiv E\{Y_{i\, t+1}|X_{i\, t+1}=x_{i0},
f_{1\, t+1}=f_{10}\}\approx g_{i0}(x_{i0})+z_{0}'\beta_{i0}$,
where $z_{0}=(f_{10}', N^{-1}\sum_{i=1}^{N}x_{i0}')'$. 
We propose the following $1-\alpha$ confidence interval for $\mu_{i0}$:
\begin{equation}\label{ci:for:mu:i}
\widehat{\mu}_i\pm z_{1-\alpha/2} \frac{\sigma_{x_0}\sigma_\epsilon}{\sqrt{Th_i}},
\,\,\,\,\textrm{where $\widehat{\mu}_{i}= \widehat{g}_{i}(x_{i0})+z_{0}'\widehat{\beta}_{i}$.}
\end{equation}
Here, $z_{1-\alpha/2}$ is the $(1-\alpha/2)$-percentile of standard normal distribution.
The following Corollary \ref{cor:ci:mu:i} guarantees the asymptotic validity of (\ref{ci:for:mu:i}).
\begin{Corollary}\label{cor:ci:mu:i}
Under the conditions of Theorem \ref{theorem: joint distr 3.}, we have, \textrm{as $(N,T)\rightarrow\infty$,}
\begin{equation}
\sqrt{Th_i}\left( \widehat{\mu}_i -\mu_{i0} \right) \overset{d}{\to} N\left(0, \sigma_{x_0}^2 \sigma_{\epsilon}^2\right).
\end{equation}
\end{Corollary}
Another application of Theorem \ref{theorem: joint distr 3.} is to construct 
the prediction interval for $Y_{i\, t+1}$. 
By (\ref{basic:model 2}),
\begin{equation}\label{pi:decom}
Y_{i\, t+1}-\widehat{\mu}_{i}=(\mu_{i0}-\widehat{\mu}_i)+
\epsilon_{i\, t+1}-\gamma_{2i}'(\bar{\Gamma}_2\bar{\Gamma}_2')^{-1}\bar{\Gamma}_2\bar{v}_{t+1}.
\end{equation}
The proof of Theorem \ref{them:convergence:rate:result:1} indicates
that the last term of (\ref{pi:decom}) is $O_P(N^{-1/2})$,
whereas the first term is $O_P((Th_i)^{-1/2})$ thanks to Corollary \ref{cor:ci:mu:i}.
If $Th_i=o(N)$, i.e., the last term of (\ref{pi:decom}) is asymptotically negligible,
then the asymptotic distribution of (\ref{pi:decom}) is 
a convolution of $F_\epsilon$ and the distribution of $N(0,\sigma_{x_0}^2\sigma_\epsilon^2/(Th_i))$,
where $F_\epsilon$ is the c.d.f. of $\epsilon_{i\, t+1}$.
Let $q_{\alpha/2}$ and $q_{1-\alpha/2}$
be the $\alpha/2$- and $(1-\alpha/2)$-percentiles of the convolution, then
a $1-\alpha$ prediction interval for $Y_{i\, t+1}$ is
\begin{equation}\label{pi:general}
[\widehat{\mu}_i+q_{\alpha/2},\widehat{\mu}_i+q_{1-\alpha/2}].
\end{equation}
In particular, if $\epsilon_{i\, t+1}\sim N(0,\sigma_\epsilon^2)$,
then (\ref{pi:general}) becomes the following
\[
\widehat{\mu}_i\pm z_{1-\alpha/2}\sigma_\epsilon\sqrt{\sigma_{x_0}^2/(Th_i)+1}.
\]
The asymptotic variance $\sigma_{x_0}^2$ has an explicit expression 
\[
\sigma_{x_0}^2=h_iV_i(K^{(i)}_{x_0},K^{(i)}_{x_0})=
\sum_{\nu\ge1}\frac{h_i|\varphi^{(i)}_\nu(x_0)|^2}{(1+\eta_i\rho^{(i)}_\nu)^2}.
\]
In practice, we can estimate $\sigma_{x_0}^2$ by replacing $\varphi^{(i)}_\nu$
and $\rho^{(i)}_\nu$ with their empirical counterparts such as kernel
eigenvalues and kernel eigenfunctions; see, e.g., \cite{braun06}. 
Meanwhile, we estimate $\sigma_\epsilon^2$ by the following
\begin{equation}\label{estimation:of:sigma_epsilon}
\widehat{\sigma}_\epsilon^2=\sum_{t=1}^T\{Y_{it}-\widehat{g}_i(X_{it})-Z_t'\widehat{\beta}_i\}^2/T.
\end{equation}
The following result shows that (\ref{estimation:of:sigma_epsilon}) is a consistent estimator.
\begin{Proposition}\label{proposition:estimation:sigma:i:epsilon}
Under conditions of Theorem \ref{them:convergence:rate:result:1},
if $r_{i,M}=o_P(h_i^{1/2})$ and $r_{i,M}=o_P(T^{-1/\alpha})$, 
then $\widehat{\sigma}_\epsilon^2 \to \sigma_\epsilon^2$ in probability, as $(N,T) \to \infty$.
\end{Proposition}

\section{Homogeneous Model} \label{sec: homo}

In this section, we consider a homogeneous case, i.e., $g_i=g$
for all $i\in[N]$. 
Assuming homogeneity, model (\ref{basic:model}) becomes the following
\begin{equation}\label{model:homogeneity}
Y_{it}=g(X_{it})+\boldsymbol{\gamma}_{1i}'{f}_{1t}+\boldsymbol{\gamma}_{2i}'{f}_{2t}+\epsilon_{it},
\,\,i\in[N], t\in[T].
\end{equation}
By (\ref{DGP of X}) and a similar statement as (\ref{basic:model 2}),
we can rewrite (\ref{model:homogeneity}) as the following 
\begin{equation} \label{basic:model:2:homogeneity}
Y_{it}=g(X_{it})+Z_t'\beta_i+e_{it},\,\,\,\,i\in[N], t\in[T],
\end{equation}
where $Z_t$, $\beta_i$ and $e_{it}$ are given in (\ref{basic:model 2}).

We will provide a procedure for estimating $g$
and explore its asymptotic property. Our theoretical results
hold when $M\to\infty$. Here $M\to\infty$ means
either $(N,T)\to\infty$ or $N\to\infty$, fixed $T$.
Whereas the estimation of $\beta_i$ is inconsistent when $T$ is fixed
due to insufficient data in each individual unit. The $\beta_i$ will be treated as nuisance parameters
throughout the whole section. 

\subsection{Estimation Procedure}\label{sec:est:homo}
Suppose $g$ belongs to an RKHS $\mathcal{H}\subset L^2(\mathcal{X})$
with inner product $\langle\cdot,\cdot\rangle_{\mathcal{H}}$ and 
kernel $\bar{K}(\cdot,\cdot)$, where $\mathcal{X}\subseteq\bbR^d$.
Our estimation is based on profile least squares:

\textbf{Step (a).} For any $g\in\mathcal{H}$,
we estimate $\beta_i$ through the following
\begin{equation}\label{profile:estimation}
\min_{\beta_i\in\bbR^{q_1+d}}\sum_{t=1}^T(Y_{it}-g(X_{it})-Z_t'\beta_i)^2=
\min_{\beta_i\in\bbR^{q_1+d}}(Y_i-\tau_i g-\Sigma'\beta_i)'(Y_i-\tau_i g-\Sigma'\beta_i),\,\,i\in[N],
\end{equation}
where $Y_i=(Y_{i1},\ldots,Y_{iT})'$, $\tau_i g=(g(X_{i1}),\ldots,g{(X_{iT})})'$
and $\Sigma=(Z_1,\ldots,Z_T)$. Recall that $\Sigma$ is $(q_1+d)\times T$.
Suppose $T \geq q_1+d$ so that $(\Sigma\Sigma')^{-1}$ exists. Then
(\ref{profile:estimation}) has solution
$\widehat{\beta}_i=(\Sigma\Sigma')^{-1}\Sigma (Y_{i}-\tau_i g)$.

\textbf{Step (b).} Plug the above $\widehat{\beta}_i$ into (\ref{profile:estimation}).
The minimum value of (\ref{profile:estimation})
is equal to $(Y_i-\tau_i g)'(I_T-\Sigma'(\Sigma\Sigma')^{-1}\Sigma)(Y_i-\tau_i g)$.
Then we estimate $g$ by the following 
\begin{equation} \label{eq: homo-objective fn}
	\widehat{g}=\arg\min_{g\in\mathcal{H}}\ell_{M,\eta}(g)
	\equiv\arg\min_{g\in\mathcal{H}}\left\{\frac{1}{2NT}\sum_{i=1}^N(Y_i-\tau_i g)'P(Y_i-\tau_i g)+\frac{\eta}{2}\|g\|^2_{\mathcal{H}}\right\},
\end{equation}
where $\eta>0$ is a penalty parameter and
$P=I_T-\Sigma'(\Sigma\Sigma')^{-1}\Sigma$ with $I_T$ the $T\times T$ identity.

The above Step (b) yields an explicit solution.
Specifically, by representer theorem, $\widehat{g}$ satisfies
\begin{equation}\label{represent:thm:homo}
{g}(x) = \sum_{i=1}^{N}\sum_{t=1}^{T} a_{it}\bar{K}(X_{it},x) = a'\bar{K}_x,
\end{equation}
where $a=(a_{11}, \cdots, a_{1T}, \cdots, a_{N1}, \cdots, a_{NT})'$ are constant scalars, and
$$\bar{K}_x=\left(\bar{K}(X_{11},x), \cdots,\bar{K}(X_{1T},x), \cdots,\bar{K}(X_{N1},x), \cdots,\bar{K}(X_{NT},x) \right)'.$$
Similar to (\ref{hetero:g:H:i}),
$\|g\|_{\mathcal{H}}^2=a' \bar{\mathcal{K}} \ a$,
where 
$\bar{\mathcal{K}}= ( \bar{K}_{X_{11}}, \cdots, \bar{K}_{X_{1T}}, \cdots, \bar{K}_{X_{N1}}, \cdots, \bar{K}_{X_{NT}} ) \in\bbR^{NT\times NT}$.
Therefore, we can rewrite $\ell_{M,\eta}(g)$ as
$\ell_{M,\eta}(g)=\left(Y- \bar{\mathcal{K}} a \right)' P_N \left(Y- \bar{\mathcal{K}} a \right)/(2NT) +\eta a' \bar{\mathcal{K}} a/2$,
where $Y=(Y_{11},\cdots,Y_{1T},\cdots,Y_{N1},\cdots,Y_{NT})'$ is an $NT$-vector and $P_N=I_N \otimes P$ is $NT \times NT$.
The minimizer $\widehat{a}$ has an expression
$\widehat{a} = \left( P_N \bar{\mathcal{K}} + NT\eta I_{NT} \right)^{-1} P_N Y$.
Then $\widehat{g}(x) = \widehat{a}' \bar{K}_x$ for any $x\in\mathcal{X}$.
\begin{Remark}\label{gcv:2}
We propose the following GCV method for choosing $\eta$
in the above estimation:
\begin{equation*}
\widehat{\eta}=\arg\min_{\eta>0}\textrm{GCV}(\eta)\equiv\arg\min_{\eta>0}
\frac{\|(I_{NT}-B_\eta)Y\|_2^2}{NT[1-tr(B_\eta)/(NT)]^2},
\end{equation*}
where $B_\eta$ is the $NT\times NT$ smoothing matrix
defined similar to Remark \ref{gcv:1}.
\end{Remark}

\subsection{Rate of Convergence}\label{subsec:rate:homo}
To derive the rate of convergence for $\widehat{g}$, let us 
adapt the framework of Section \ref{sec:model:preliminary}
to the homogeneous setting.
Define $V(g,\widetilde{g})=\sum_{i=1}^NE\left\{(\tau_i g)'P(\tau_i\widetilde{g})\big|\mathcal{F}_1^T\right\}/(NT)$ for any $g,\widetilde{g}\in\mathcal{H}$.
Suppose that $V(\cdot,\cdot)$ and $\langle\cdot,\cdot\rangle_{\mathcal{H}}$ are simultaneously diagonalizable.
\begin{Assumption}\label{A5}
There exist eigenfunctions $\varphi_\nu\in\mathcal{H}$ and a nondecreasing positive 
sequence of eigenvalues $\rho_\nu$ such that 
$V(\varphi_\nu,\varphi_\mu)=\delta_{\nu\mu},\,\,\langle\varphi_\nu,\varphi_\mu\rangle_{\mathcal{H}}=\rho_\nu
\delta_{\nu\mu}$, for any $\nu,\mu\ge1$.
Furthermore, $1\notin\mathcal{H}$ and any function $g\in\mathcal{H}$ admits a generalized Fourier expansion $g=\sum_{\nu\ge1}V(g,\varphi_\nu)\varphi_\nu$. Both $\varphi_\nu$ and $\rho_\nu$
are $\mathcal{F}_1^T$-measurable and 
$c_\varphi\equiv\sup_{\nu\ge1}\|\varphi_\nu\|_{\sup}=O_P(1)$,
i.e., $\varphi_\nu$ are stochastic uniformly bounded.
\end{Assumption}
\ref{A5} type conditions are commonly used in literature to derive the rate of convergence 
for smoothing splines or kernel ridge regression;
see \cite{GC93,SC13,CS15,ZCL16}. Classic ways to verify such conditions 
rely on variational methods; see \cite{weinberger1974}.
Nevertheless, Assumption \ref{A5} differs from literature in that the functional $V$ and
the eigenpairs $(\rho_\nu,\varphi_\nu)$ are random.
Fortunately, we can still verify Assumption \ref{A5} by adapting
the classic variational method to this new setting.
The exact verification is deferred to Lemma \ref{lemma: V & g} in appendix.

Define
$\Sigma_\star=(Z_1^\star,\ldots,Z_T^\star)$, a square matrix of dimension $q_1+d$, 
where $Z_t^\star=(f_{1t}',(\bar{X}_t^\star)')'$ and
$\bar{X}_t^\star=\bar{X}_t-\bar{v}_t$ for $t\in[T]$.
By (\ref{DGP of X 2}) and Assumption \ref{A1}, 
$\Sigma_\star$ is independent of the variables $v_{it}$.
Hence, $\Sigma_\star$ can be viewed as a ``noiseless" analogy of $\Sigma$.
In the below we impose a moment condition on the spectral norms
of various matrices. 

\begin{Assumption}\label{A7}
There exist constants $\zeta>4$ and $c>0$  such that
\begin{eqnarray*}
&&E\left(\|(\Sigma_\star\Sigma_\star'/T)^{-1}\|_{\textrm{op}}^\zeta\right)\le c,\,\,\,\,
E\left(\|(\Sigma\Sigma'/T)^{-1}\|_{\textrm{op}}^\zeta\right)\le c\,\,\,\,
\textrm{and}\,\,\\
&&E\left(\|\Sigma_\star\Sigma_\star'/T\|_{\textrm{op}}^{2\zeta/(\zeta-4)}\right)\le c,\,\,\,\,
E\left(\|\sum_{t=1}^T f_{2t}f_{2t}'/T\|_{\textrm{op}}^{2\zeta/(\zeta-4)}\right)\le c,
\end{eqnarray*}
where $\|\cdot\|_{\textrm{op}}$ represents the operator norm of square matrices.
\end{Assumption}

For any $g,\widetilde{g}\in\mathcal{H}$, define 
$\langle g,\widetilde{g}\rangle=V(g,\widetilde{g})+\eta\langle g,\widetilde{g}\rangle_{\mathcal{H}}$.
Following \cite{CS15}, $(\mathcal{H},\langle\cdot,\cdot\rangle)$ is an RKHS
with reproducing kernel denoted $K$.
For convenience, define $\mathbb{X}_i=(X_{i1},\ldots,X_{iT})'$ 
and $K_{\mathbb{X}_i}=(K_{X_{i1}},\ldots,K_{X_{iT}})'$ for $i\in[N]$.
Similar to Section \ref{sec:prelim:RKHS},
there exists a positive definite self-adjoint operator 
$W_\eta: \mathcal{H} \rightarrow \mathcal{H}$ such that 
$\langle W_{\eta}g, \widetilde{g} \rangle = \eta \langle g, \widetilde{g} \rangle_\mathcal{H}$,
$g,\widetilde{g}\in\mathcal{H}$. 
Then the Fr\'{e}chet derivatives of $\ell_{M,\eta}(g)$ have the following expressions
\begin{eqnarray*}
D\ell_{M,\eta}(g)\Delta\theta
&=&\langle -\frac{1}{NT}\sum_{i=1}^N(Y_{i}-\langle K_{\mathbb{X}_i},
g \rangle)'P K_{\mathbb{X}_i}+W_\eta g,
\Delta g \rangle
\equiv\langle S_{M,\eta}(g),\Delta g \rangle,\\
DS_{M,\eta}(g)\Delta g &=&\frac{1}{NT}\sum_{i=1}^N\langle K_{\mathbb{X}_i},
\Delta g \rangle' P K_{\mathbb{X}_i}+W_\eta \Delta g,\\
D^2S_{M,\eta}(g)&=&0.
\end{eqnarray*}
For $p,\delta>0$, define $\mathcal{G}(p)=\{g\in\mathcal{H}: \|g\|_{\sup}\le1,\|g\|_{\mathcal{H}}\le p\}$
and an entropy integral
\[
J(p,\delta)=\int_0^\delta\psi_2^{-1}\left(D(\varepsilon,\mathcal{G}(p),\|\cdot\|_{\sup})\right)d\varepsilon+\delta\psi_2^{-1}\left(D(\delta,\mathcal{G}(p),\|\cdot\|_{\sup})^2\right),
\]
where recall that $D(\varepsilon,\mathcal{G}(p),\|\cdot\|_{\sup})$
is the $\varepsilon$-packing number of $\mathcal{G}(p)$
in terms of $\|\cdot\|_{\sup}$-metric. 
Define
\[
h=\left(\sum_{\nu\ge1}\frac{1}{1+\eta\rho_\nu}\right)^{-1},
\,\,\,\,
b_{N,p}=\sqrt{\log\log\left(NJ(p,1)\right)}J(p,1),\,\,\,\,
p=(c_\varphi\sqrt{h^{-1}\eta})^{-1}.
\]
\begin{Theorem}\label{theorem:homo convergence:rate:result:1}
Suppose that Assumptions \ref{A1}, \ref{A5} and \ref{A7} are satisfied. Furthermore, 
$b_{N,p}=o_P(N^{1/2}h)$.
Then, as $M\to\infty$,
$\|\widehat{g}-g_{0}\|=O_P(r_M)$, where $r_M=(NTh)^{-1/2}+(Nh^{1/2})^{-1}+\eta^{1/2}$.
\end{Theorem} 
Theorem \ref{theorem:homo convergence:rate:result:1} provides a rate of convergence for
$\widehat{g}$. 
Like in Theorem \ref{them:convergence:rate:result:1}, to yield optimal rate of convergence
(denoted $r_M^\star$), the optimal choice of $h$ (denoted $h^\star$) relies on kernels and 
relationship of $N,T$.
The following Table \ref{table:rate:homo} summarizes 
the values of $h^\star$ and $r_M^\star$
in PDK and EDK. 
Interestingly, when $N\ge T$, $r_M^\star$ depends on $NT$;
whereas $N<T$, $r_M^\star$ only depends on $N$.
The latter implies that, when $N<T$, increasing 
time points will not change convergence rate. 
Moreover, it can be examined that the condition $b_{N,p}=o_P(N^{1/2}h)$
in Theorem \ref{theorem:homo convergence:rate:result:1} holds true
when $h\asymp h^\star$. 
\begin{table}[htp]
\begin{center}
\begin{tabular}{ccccc}
&\multicolumn{2}{c}{$N\ge T$}&\multicolumn{2}{c}{$N<T$}\\ \cline{2-3} \cline{4-5}
& PDK & EDK & PDK & EDK\\ \hline
$h^\star$ &$(NT)^{-1/(2k+1)}$&$(\log{(NT)})^{-1/k}$&$N^{-2/(2k+1)}$&$(\log{N})^{-1/k}$\\
$r^\star_M$&$(NT)^{-k/(2k+1)}$ & $(NT)^{-1/2}(\log{(NT)})^{1/(2k)}$ 
&$N^{-2k/(2k+1)}$ & $N^{-1}(\log{N})^{1/(2k)}$\\ \hline
\end{tabular}
\caption{A summary of $h^\star$ and $r^\star_M$ in $k$-order PDK and $k$-order EDK.}
\label{table:rate:homo}
\end{center}
\end{table}

\subsection{Asymptotic Normality}\label{subsec:normality:homo}

In this section, we will derive the asymptotic normality for $\widehat{g}$ in the proposed RKHS framework
which can be used to construct the confidence interval for $g(x)$ at any $x\in\mathcal{X}$. 
Our results are applicable in a general class of models
including nonparametric models, semiparametric models or additive models.
This is in sharp contrast to \cite{suJin2012} whose results were obtained
in nonparametric sieve estimation.
Our asymptotic normality result relies on the following theorem
which characterizes the leading term of $\widehat{g}-g_0$.
Define $D_m=\sum_{\nu=m+1}^\infty 1/(1+\eta \rho_\nu)$ for any $m\ge0$.

\begin{Theorem}[FBR for Homogeneous Model]\label{theorem: rate:ghat:minus:g0:plus:SMetag0}
Suppose that Assumptions \ref{A1}, \ref{A5}  and \ref{A7}
hold, and $h^{-1}=o_P(N^{1/2})$.
Furthermore, there exists a sequence of positive integers
$m=m_M$ such that $D_m=o_P(1)$. Then for any $x_0 \in \mathcal{X}$, we have
\begin{equation*}
	\sqrt{NT} A_{NT}|\widehat{g}_0(x_0)-g_0(x_0)+S_{M,\eta}(g_0)(x_0)-W_{\eta}g_0(x_0)|=O_P\left(\frac{b_{N,p}}{h}\left(\frac{1}{\sqrt{N}h}+\frac{\sqrt{T}}{Nh}+\sqrt{\frac{T\eta}{h}}\right)\right).
\end{equation*}
Furthermore, if ${b_{N,p}}=o_P(N^{1/2}h^2), b_{N,p}=o_P(Nh^2T^{-1/2})$  
and $b_{N,p}=o_P(h^{3/2}(T\eta)^{-1/2})$, then
\begin{equation}
	\sqrt{NT} A_{NT}|\widehat{g}_0(x_0)-g_0(x_0)+S_{M,\eta}(g_0)(x_0)-W_{\eta}g_0(x_0)|=o_P(1).
\end{equation}
\end{Theorem}
It follows from Theorem \ref{theorem: rate:ghat:minus:g0:plus:SMetag0} that 
$\widehat{g}(x_0)-g_0(x_0)$ and 
$S_{M,\eta}(g_0)(x_0)-W_\eta g_0(x_0)$ are asymptotically equivalent. 
The latter will be used 
to derive the limit distribution of $\widehat{g}(x_0)$, i.e., Theorem \ref{theorem:normality:ghat:minus:g0} below.
Let $\gamma_2=(\gamma_{21}',\gamma_{22}',...,\gamma_{2N}')'$,
a $q_2N$-vector of unknown factor loadings.
\begin{Theorem}\label{theorem:normality:ghat:minus:g0}
Suppose that Assumptions \ref{A1}, \ref{A5} and \ref{A7} hold, 
and 
$h^{-1}=o_P(N^{1/2})$.
Furthermore, 
there exists a sequence of positive integers
$m=m_M$ with $m=o(N^{1/2})$ and $mT\lambda_{\max}(\gamma_2\gamma_2')=o(N)$
such that $D_m=o_P(h^{1/2})$ and $D_m^2T=o_P(Nh)$. Then, for all $x_0 \in \mathcal{X}$, it follows that
$$\sqrt{NT}A_{NT}(\widehat{g}(x_0)-g_0(x_0)+W_{\eta}g_0(x_0)) \overset{d}{\to} N(0,\sigma_{\epsilon}^2),
\,\,\,\,\textrm{as $M\to\infty$,}$$
where $A_{NT}=A_{NT}(x_0)\equiv(\frac{1}{NT}\sum_{i=1}^N \kxi'(x_0) P \kxi(x_0))^{-1/2}$.
\end{Theorem}

Theorem \ref{theorem:normality:ghat:minus:g0} shows that $\widehat{g}(x_0)$
is asymptotically normal at any $x_0\in\mathcal{X}$. 
The rate conditions $h=o_P(N^{1/2})$, $m=o(N^{1/2})$, $D_m=o_P(h^{1/2})$ and $D_m^2T=o_P(Nh)$
are reasonable and can be verified in concrete settings.
For instance, when $T\le N$, for $k$-order PDK, 
the conditions hold if $h\asymp (NT)^{-1/(2k+1)}$,
$m=N^{1/2}/\log(N)$, and
correspondingly, $D_m\asymp h^{-2k}m^{-2k+1}$. 
For $k$-order EDK, the conditions hold if $h\asymp (\log(NT))^{-1/k}$, 
$m=N^{1/4}$, and correspondingly,
$D_m\asymp e^{-cN^{k/4}}TN^{(5-k)/4}$ .
The condition $mT\lambda_{\max}(\gamma_2\gamma_2')=o(N)$ says that the signal of the unobserved factors
is not strong so that the asymptotic normal part from $\widehat{g}$ can be filtered out.

Nonetheless, $\widehat{g}(x_0)$ does not converge to the truth
due to the bias $W_\eta g_0(x_0)$.
Following \cite{SC15}, it can be verified that 
$W_\eta g_0(x_0)=o_P((\eta/h)^{1/2})$.
Therefore,
we need to assume $NT\eta A_{NT}^2=O_P(h)$ for bias correction,
a version of ``undersmoothing condition."

\begin{Corollary}\label{cor:nobias:homo}
Suppose that the conditions in Theorem \ref{theorem:normality:ghat:minus:g0} are satisfied and  $NT\eta A_{NT}^2=O_P(h)$.
Then, as $M\to\infty$,
$\sqrt{NT}A_{NT}(\widehat{g}(x_0)-g_0(x_0)) \overset{d}{\to} N(0,\sigma_{\epsilon}^2)$.
\end{Corollary}
Corollary \ref{cor:nobias:homo} provides 
asymptotic normality for $\widehat{g}(x_0)$
where the estimator converges to the truth.
We can show that the undersmoothing condition $NT\eta A_{NT}^2=o_P(h)$
holds true when we properly choose $h$ with
$h=o(h^\star)$ in both PDK and EDK.
A direct consequence is the following $1-\alpha$ confidence interval for $g(x_0)$:
\begin{equation}\label{ci:homo:g}
\widehat{g}(x_0)\pm z_{1-\alpha/2}\frac{\sigma_\varepsilon}{\sqrt{NT}A_{NT}}.
\end{equation}
In practice, we estimate $\sigma^2_\varepsilon$ by
$\widehat{\sigma}^2_\varepsilon=\frac{1}{N(T-q_1-d)}\sum_{i=1}^N(Y_i-\tau_i\widehat{g})'P(Y_i-\tau_i\widehat{g})$,
which is consistent as demonstrated in the following result.
The rate conditions in Proposition \ref{prop:homo sigma} hold true
when $h\asymp h^\star$ in both PDK and EDK.

\begin{Proposition}\label{prop:homo sigma}
Suppose that Assumptions \ref{A1},\ref{A5},\ref{A7}
hold. Moreover, $b_{N,p}=o_P(N^{1/2}h)$,
$h^{-1}=o_P((NT)^{1/2})$, $h^{-1}=o_P(N)$ and $\eta=o_P(h)$. 
Then $\widehat{\sigma}^2_\varepsilon\to \sigma_{\epsilon}^2$ in probability,
as $M\to\infty$.
\end{Proposition}

\begin{Remark}\label{rem:additive:rkhs}
It is of interest to apply our results to RKHS of delicate structures,
e.g., additivity.  
Suppose that, for $l=1,\ldots,r$, $\mathcal{H}_l$
is an RKHS with inner products and reproducing kernels denoted $\langle\cdot,\cdot\rangle_{\mathcal{H}_l}$
and $\bar{K}_l$, respectively. 
Define $\mathcal{H}=\{g_1(x_1)+\cdots+g_r(x_r): g_1\in \mathcal{H}_1,\cdots,g_r\in\mathcal{H}_r\}$. 
Then $\mathcal{H}$ is an Additive RKHS with kernel $\bar{K}((x_1,\ldots,x_r),(y_1,\ldots,y_r))\equiv
\bar{K}_1(x_1,y_1)+\cdots+\bar{K}_1(x_r,y_r)$ 
and inner product $\langle f, g \rangle_{\mathcal{H}}\equiv
\langle f_1, g_1 \rangle_{\mathcal{H}_1}+\cdots+\langle f_r, g_r \rangle_{\mathcal{H}_r}$.
In particular, $(\mathcal{H},\bar{K})$ becomes \textit{Partial Linear}
if some of $\mathcal{H}_l$'s are generated by linear kernels, i.e., $1$-order polynomial. 
\end{Remark}

\begin{Remark}\label{rem:semipara:ci}
When $g$ is partial linear, i.e., $g(x_1,x_2)=x_1\beta+g_2(x_2)$,
(\ref{ci:homo:g}) can be used to construct a confidence interval for $\beta$.
To illustrate this, suppose $x_1,x_2$ are univariate for simplicity.
Choose $x_2$ such that $g_2(x_2)=0$,
then $\beta=g(1,x_2)$. By (\ref{ci:homo:g}), the $1-\alpha$ confidence interval for $\beta$ is
\begin{equation}\label{ci:plm}
\widehat{g}(1,x_2)\pm z_{1-\alpha/2}\frac{\widehat{\sigma}_\varepsilon}{\sqrt{NT}A_{NT}}.
\end{equation}
Extensions can be easily done when $x_1,x_2$ are multidimensional.
\end{Remark}

\section{Numerical Study}\label{sec:numerical:study}

We examine our methods using simulated datasets and a real dataset.
 
\subsection{Simulation}\label{simulation:consistency}
A comparison will be first performed between our estimation procedure and Su and Jin's sieve estimation. 
We considered the same data generating process as \cite{suJin2012} for fair comparison.
That is, the data were generated as follows: for $i\in [N], t \in [T]$,
\begin{eqnarray*}
	y_{it}&=&g_i(x_{it,1},x_{it,2})+\gamma_{1i}+\gamma_{2i,1}f_{2t,1}+\gamma_{2i,2}f_{2t,2}+\epsilon_{it},\\
	g_i(x_{it,1},x_{it,2})&=&\exp(x_{it,1})/(1+\exp(x_{it,1}))+\delta_i(0.5x_{it,2}-0.25x_{it,2}^2),\\
	x_{it,s}&=&\Gamma_{1i,s}+\Gamma_{2i,s1}f_{2t,1}+\Gamma_{2i,s2}f_{2t,2}+v_{it,s}, s=1,2,
\end{eqnarray*}
where $\epsilon_{it}=\rho_{i}\epsilon_{i,t-1}+\sigma_i(1-\rho_i^2)\xi_{it}$
with $\rho_i,\sigma_i^2\overset{iid}{\sim}Unif[0,0.95]$ and
$\xi_{it}\overset{iid}{\sim}N(0,1)$;
$v_{it,1},v_{it,2}$ were generated similar to $\epsilon_{it}$;
$\delta_i\overset{iid}{\sim}Unif[0,1]$;
$f_{2t,s}=0.5f_{2,t-1,s}+(1-0.5^2)^{1/2}\zeta_{t,s}$ 
with $\zeta_{t,s}\overset{iid}{\sim}N(0,1)$ for $s=1,2$;
$\gamma_{1i}=0.5T^{-1}\sum_{t=1}^Tx_{it,1}+0.5T^{-1}\sum_{t=1}^Tx_{it,2}$ and 
$$\gamma_{21},\ldots,\gamma_{2N},\Gamma_{11},\ldots,\Gamma_{1N}\overset{iid}{\sim} N\left( \begin{pmatrix}
0\\0
\end{pmatrix}, \begin{pmatrix}
1&0.5\\
0.5&1
\end{pmatrix} \right);$$ 
the entries of $\Gamma_{2i}$ were generated from a multivariate normal distribution 
with mean $(1,0,0,1)$ and identity covariance matrix.

We chose an additive RKHS with kernel $\bar{K}((x_1,x_2),(y_1,y_2))=\bar{K_1}(x_1, y_1)+\bar{K}_2(x_2,y_2)$, 
where $\bar{K}_1$ is Gaussian kernel and $\bar{K}_2$ is a
2-order polynomial kernel. See Section \ref{sec:prelim:RKHS} for 
definitions of these kernels.
The smoothing parameter $\eta$ was chosen by the proposed GCV; see Remark \ref{gcv:1}. 
We considered $N=25,50,100$ in both heterogeneous and homogeneous cases, whereas
$T=25, 50, 100$ in heterogeneous case and $T=8, 25, 100$ in homogeneous case. 
Mean squared errors (MSE) were computed based on 1000 replications.

Table \ref{table:simulation:RMSE}
compares our estimator $\widehat{g}$ with Su and Jin's sieve estimator $\widehat{g}_{\textrm{sieve}}$.
We observe that, in heterogeneous setting, 
our estimator yields slightly smaller MSE when $N,T\le 50$,
and becomes comparable with $\widehat{g}_{\textrm{sieve}}$ when $N$ or $T$ is 100.
In homogeneous setting, it can be seen that
our estimator yields slightly smaller MSE when $T=8$,
and becomes comparable  when $T=25,100$.
\begin{table}[H]
\centering
\caption{MSE of two estimators in various settings.}
\label{table:simulation:RMSE}
\begin{tabular}{lllllllll}
\hline \hline
                           & \multicolumn{4}{c}{Heterogeneous Setting} & \multicolumn{4}{c}{Homogeneous Setting} \\\cline{3-5}\cline{7-9}
Estimator                  & N/T    & 25       & 50       & 100     & N/T    & 8       & 25      & 100     \\
$\widehat{g}$                  & 25     & 0.813    & 0.536    & 0.419   & 25     & 0.291   & 0.219   & 0.035   \\
                           & 50     & 0.793    & 0.518    & 0.394   & 50     & 0.171   & 0.133   & 0.030   \\
                           & 100    & 0.980    & 0.545    & 0.402   & 100    & 0.123   & 0.118   & 0.019   \\
$\widehat{g}_{\textrm{sieve}}$ & 25     & 1.061    & 0.736    & 0.538   & 25     & 0.528   & 0.245   & 0.143   \\
                           & 50     & 0.932    & 0.646    & 0.457   & 50     & 0.344   & 0.164   & 0.095   \\
                           & 100    & 0.996    & 0.674    & 0.47    & 100    & 0.245   & 0.115   & 0.065  \\
                           \hline \hline
\end{tabular}
 \begin{tablenotes}
      \small
      \centering
      \item
    \end{tablenotes}
\end{table}

Next, we examined the proposed confidence interval (\ref{ci:homo:g}). We 
only considered the homogeneous setting (\ref{basic:model:2:homogeneity}) with  $N=25, 50, 100, T=8, 25$,
and ``true" function $g(x)= 0.6\beta_{30,17}(x)+0.4\beta_{3,11}(x)$, where
$\beta_{a,b}(\cdot)$ is the beta density function with shape and scale $a$ and $b$;
$\epsilon_{it}, v_{it}\overset{iid}{\sim}N(0,1)$;
$f_{1t,s}=0.5f_{1t-1,s}+(1-0.5^2)^{1/2}\zeta_{t,s}$ for $s=1,2$,
where $\zeta_{t,1},\zeta_{t,2}\overset{iid}{\sim}N(0,1)$;
$f_{2t,s}$ was generated the same way as $f_{1t,s}$;
$\gamma_{1i}=T^{-1}\sum_{t=1}^Tx_{it}$ and $\Gamma_{1i}\overset{iid}{\sim}N(0,1)$;
$\gamma_{2i} \overset{iid}{\sim}N(0,1)$ and  $\Gamma_{2i}\overset{iid}{\sim}N(1,1)$.

Confidence intervals for $g(x)$ at $x\in[0,1]$ were constructed
based on Gaussian kernel and a 10-order polynomial kernel.
The smoothing parameter was
selected by the proposed GCV; see Remark \ref{gcv:2}.
The coverage probabilities (CP) of the intervals were examined
based on 1000 independent replications.
Figures \ref{figure:simulation:cover:rate:Gaussian} and \ref{figure:simulation:cover:rate:Poly}
display the CP of the 95\% confidence intervals
for $g(x)$ at 100 evenly spaced points in $[0,1]$ based on Gaussian kernel and polynomial kernel,
respectively.  
It can be seen that, when $N=100$ or $N=50,T=25$,
the CP approaches the 95\% nominal level at any $x\in[0,1]$,
demonstrating the validity of the confidence intervals. 
When $N=25$ or $N=50,T=8$, the CP 
is significantly less than the nominal level at $x\in[0,0.2]$. This is due to the peaks/trouts
of the true function which affect the small sample performance of the intervals. 
Such effect quickly vanishes in large sample setting, e.g., $N=100$ or $N=50,T=25$.


\begin{figure}[t!]
	\includegraphics[width=5 in, height=4 in ]{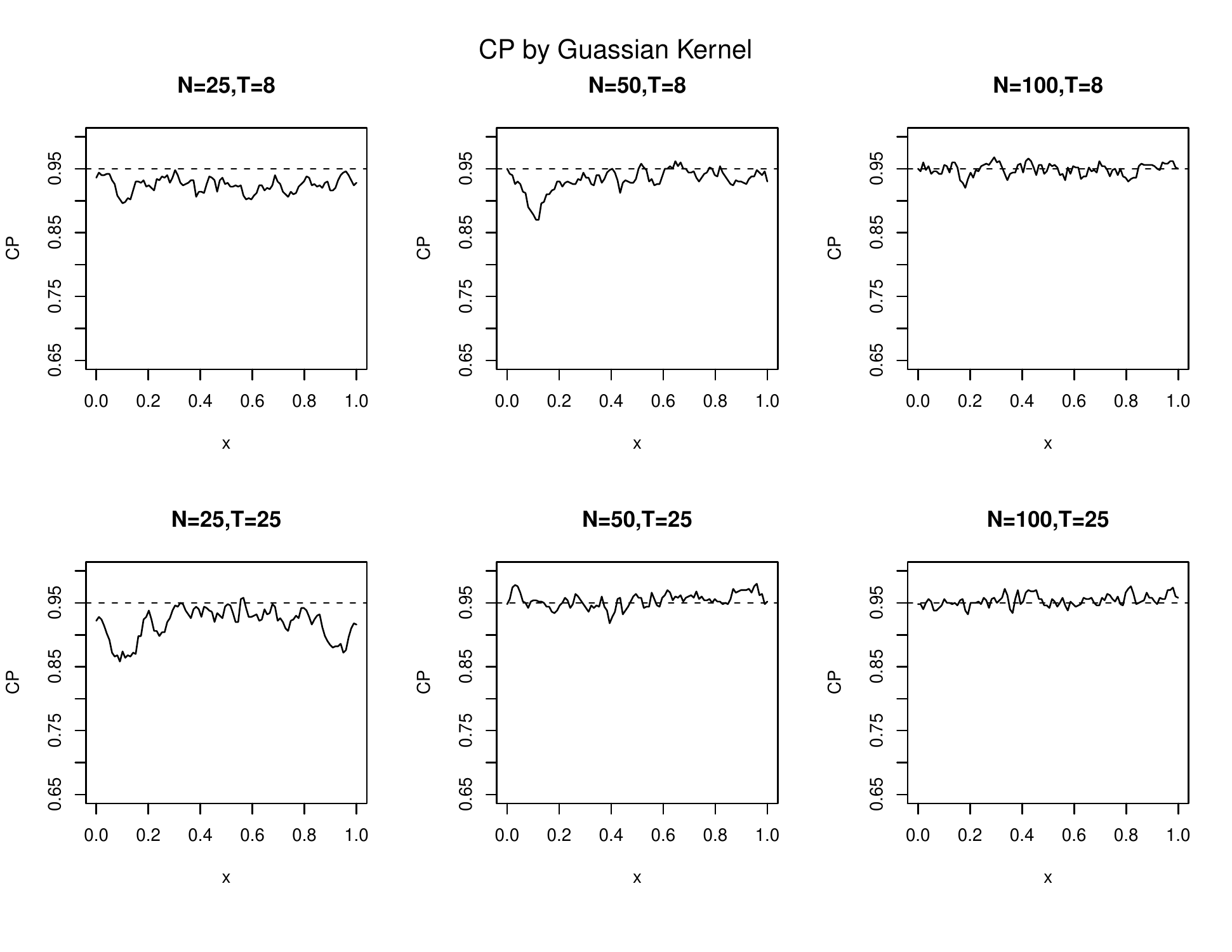}
	\caption{CP of the 95\% confidence intervals for $g(x)$ at $x\in[0,1]$ based on Gaussian kernel.
	Dashed lines indicate 95\% nominal level.}
	\label{figure:simulation:cover:rate:Gaussian}
\end{figure}

\begin{figure}[t!]
\includegraphics[width=5 in, height=4 in ]{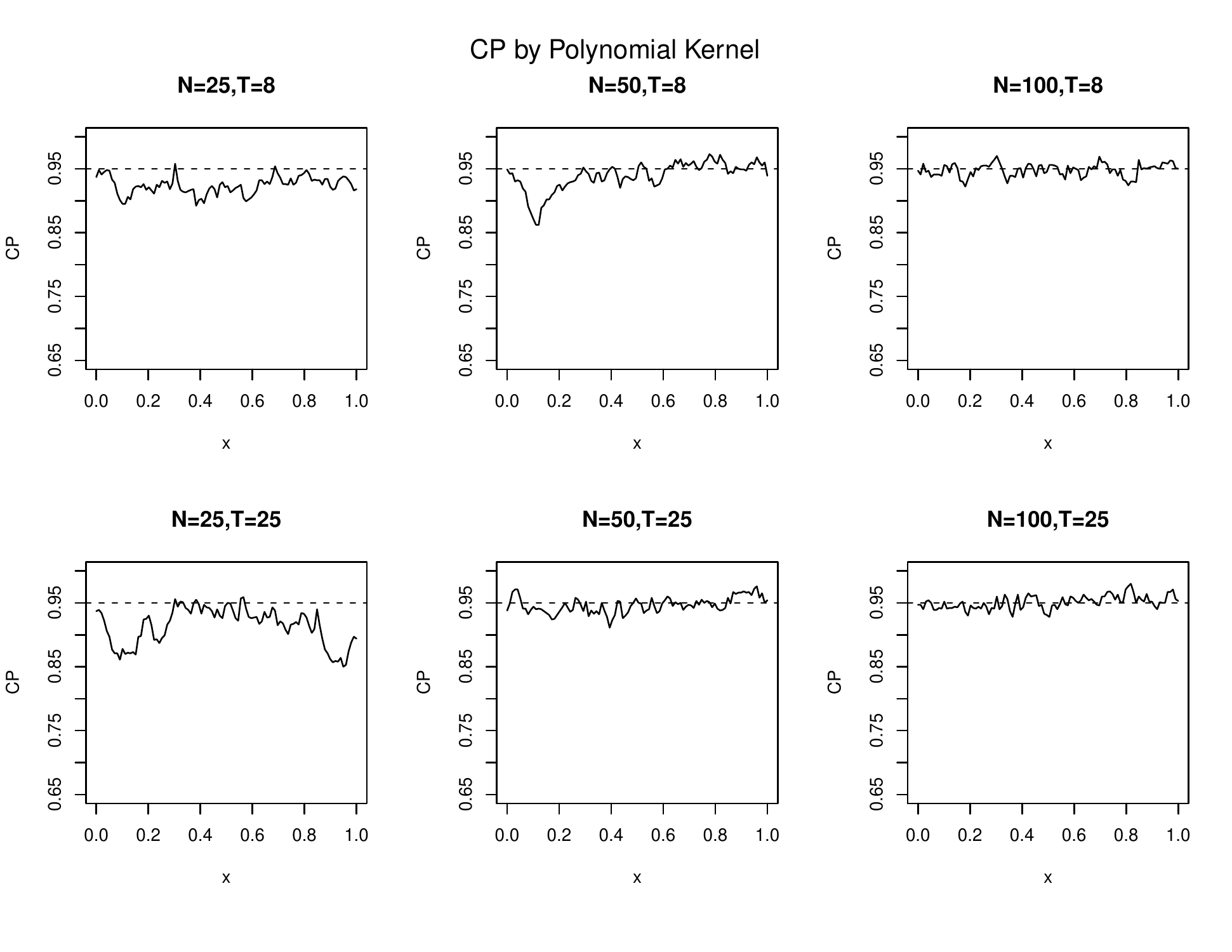}
	\caption{CP of the 95\% confidence intervals for $g(x)$ at $x\in[0,1]$ based on polynomial kernel.
	Dashed lines indicate 95\% nominal level.}
	\label{figure:simulation:cover:rate:Poly}
\end{figure}

\subsection{Export Productivity Premium}\label{sec:real:data}
 
In this section, we apply our method to examine the firm-level productivity difference between exporters and non-exporters based on a real dataset from computer and peripheral equipment manufacturing industry of Chinese Industrial Enterprises Database.
The data include observations collected from $N=100$ continuously operating firms in
$T=9$ years (1998 to 2007). For firm $i$ in year $t$, $Y_{it}$
is the log gross output,
$X_{it}=(X_{it1},X_{it2},X_{it3},X_{it4})$ 
with 
$X_{it1}$ the log capital defined as the net fixed asset,
$X_{it2}$ the log materials defined as the value of the intermediate inputs,
$X_{it3}$ the log labor defined as the total wage bill plus benefits,
and $X_{it4}$ the export intensity defined as the ratio of the
export value to the gross output value.
The aim is to investigate a relationship between $X_{it}$ and $Y_{it}$.

To enhance model flexibility, suppose that 
the log gross output and the export intensity are nonlinearly related. 
This leads us to consider the following model 
\begin{equation}\label{model:real:data}
Y_{it}=\beta_1X_{it1}+\beta_2X_{it2}+\beta_3X_{it3}+f(X_{it4})+\gamma_{1i}+\gamma_{2i}f_{2t}+
\textrm{error},
\end{equation}
where $\beta_1,\beta_2,\beta_3$ are unknown regression coefficients and
$f$ is unknown belonging to an RKHS $\mathcal{H}$ which represents productivity difference between exporters and non-exporters.
The variables $f_{2t}$ represent the unobserved common shocks, such as unobserved policy changes,
and $\gamma_{1i},\gamma_{2i}$
represent the individual specific responses to factor $f_{1t}=1$ and $f_{2t}$. 
The semiparametric structure of the regression function
$g(X_{it})\equiv X_{it1}\beta_1+\beta_2X_{it2}+\beta_3X_{it3}+f(X_{it4})$
can be naturally incorporated in an additive RKHS generated by
a polynomial kernel and a general RKHS $\mathcal{H}$;
see Remark \ref{rem:additive:rkhs}. 
In practice, we chose $\mathcal{H}$ as generated by linear kernel or polynomial kernel.

Table \ref{table:simulation:coef:k:l:m} summarizes the estimates and 95\% confidence intervals
of $\beta_1,\beta_2,\beta_3$. The intervals were calculated based on (\ref{ci:plm}). 
Overall, the results based on
linear kernel and polynomial kernel are quite similar.
The confidence intervals all exclude zero indicating the significance of the 
linear predictors, consistent with literature about Chinese manufacturing industries \citep[e.g.][]{hashiguchi2015}. 
Figure \ref{figure:realdata} displays the
95\% confidence intervals for export productivity premium
versus export intensity, based on linear kernel (left panel) and polynomial kernel (right panel). 
The red dashed lines display the upper and lower bounds of the intervals,
and the central dark lines demonstrate the estimations of $f$.
Overall, the estimations of $f$ are both increasing, consistent with the
folklore that ``exports stimulate productivity," e.g., \cite{melitz2003}. 
The red dashed lines are above zero,
indicating the significance of the export intensity effect on productivity.

\begin{table}
\centering
\caption{Estimation and 95\% confidence intervals for $\beta_1,\beta_2,\beta_3$
based on two kernels. }
\label{table:simulation:coef:k:l:m}
\begin{tabular}{llllllll}
\hline \hline
\multicolumn{4}{c}{Linear Kernel} & \multicolumn{4}{c}{Polynomial Kernel} \\
          & Estimate   & 95\% CI  & &    & Estimate  & 95\% CI  &   \\ 
$\beta_1$ & 0.1022 & [0.0624, 0.1420] &        & $\beta_1$  &   0.1010      &     [0.0606, 0.1414]      &           \\ 
$\beta_2$ & 0.0994 & [0.0672, 0.1316] &        & $\beta_2$  &      0.0989     &      [0.0663, 0.1315] &          \\ 
$\beta_3$ & 0.7300 & [0.6958, 0.7642] &       & $\beta_3$  &    0.7395       &   [0.7049, 0.7741] &        \\ \hline \hline
\end{tabular}
\end{table}

\begin{figure}[htp!]
	\includegraphics[width=3 in, height=2.5 in ]{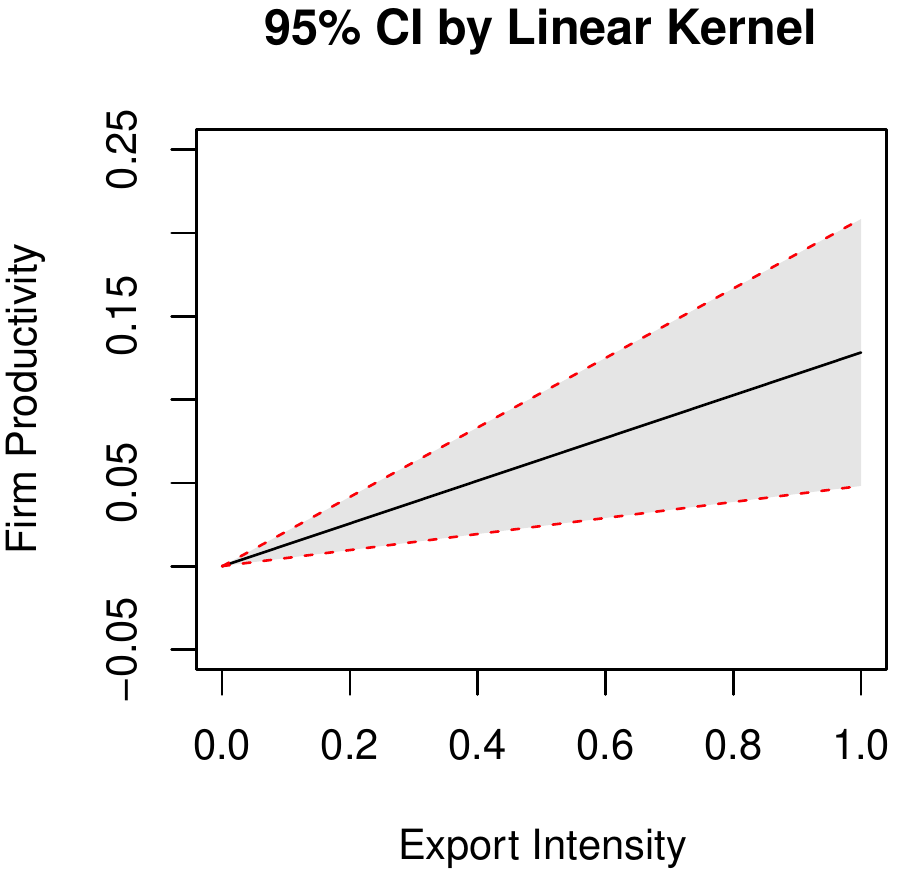}
	\includegraphics[width=3 in, height=2.5 in ]{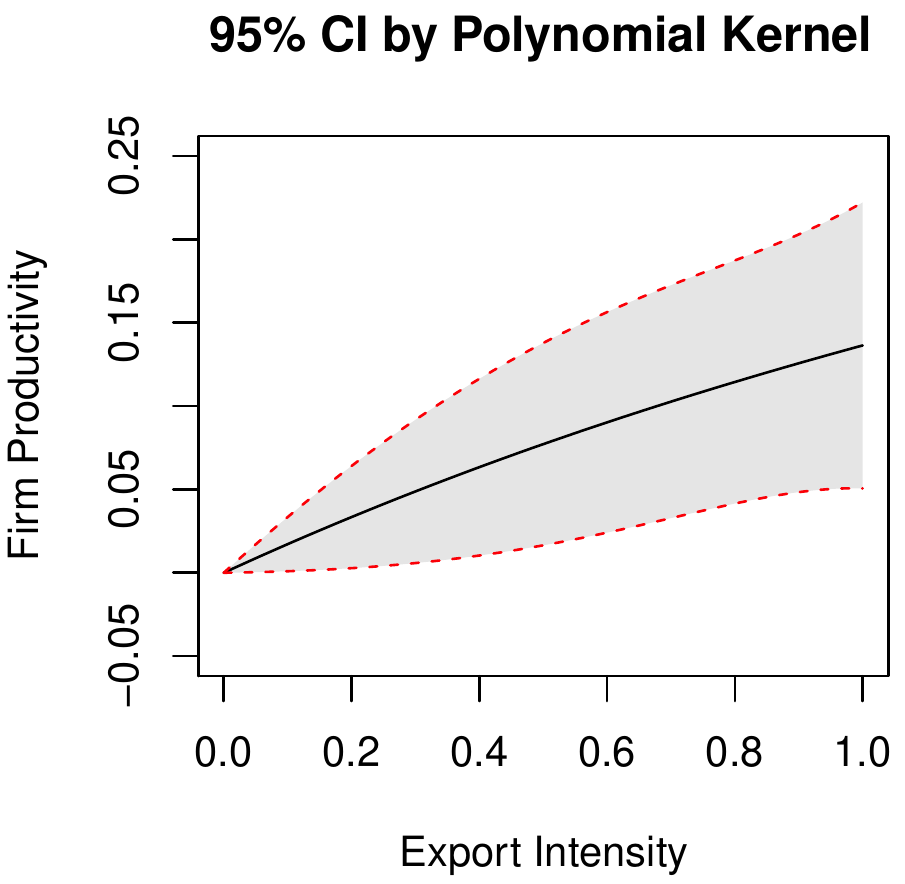}
	\caption{95\% confidence intervals for firm productivity
	versus export intensity. }
	\label{figure:realdata}
\end{figure}

\section{APPENDIX}\label{sec:appendix}
\setcounter{subsection}{0}
\renewcommand{\thesubsection}{A.\arabic{subsection}}
\setcounter{equation}{0}
\renewcommand{\theequation}{A.\arabic{equation}}
\setcounter{lemma}{0}
\renewcommand{\thelemma}{A.\arabic{lemma}}
\setcounter{theorem}{0}
\renewcommand{\thetheorem}{A.\arabic{theorem}}
\setcounter{proposition}{0}
\renewcommand{\theproposition}{A.\arabic{proposition}}

This appendix contains the proofs of the main results. In Section \ref{proof:section:consist-Hetero}, a proof of convergence rate in heterogeneous model is provided (Theorem \ref{them:convergence:rate:result:1}) 
and some auxiliary lemmas are stated. In Section \ref{proof:section:AsympDist-Hetero}, we prove FBR 
for heterogeneous model
(Theorem \ref{lemma: hetero -- leading term}) and joint asymptotic distributions of our estimators 
(Theorems \ref{theorem: joint distr 1.} and \ref{theorem: joint distr 3.}). 
Section \ref{proof:subsec:rate:homo} includes the
proof of convergence rate in homogeneous model
(Theorem \ref{theorem:homo convergence:rate:result:1}) as well as some auxiliary lemmas. In Section \ref{proof:section:normality:homo}, proofs of FBR in homogeneous model (Theorem \ref{theorem: rate:ghat:minus:g0:plus:SMetag0})
and corresponding asymptotic normality 
(Theorem \ref{theorem:normality:ghat:minus:g0}) are given. We also show that the variance estimator
is consistent.

\subsection{Proofs in Section \ref{subsec:consist-Hetero}}\label{proof:section:consist-Hetero}
In this section, we derive the rate of convergence for our estimator in the heterogeneous setting,
i.e., Theorem \ref{them:convergence:rate:result:1}. 
Before proving the results, we provide some preliminary results. 

 
\begin{lemma}\label{lemma:id}
For any $\theta\in\Theta_i$, $DS_{i,M,\eta_i}^\star(\theta)=id$,
where $S_{i,M,\eta_i}^\star(\theta)=E\{S_{i,M,\eta_i}(\theta)\}$.
\end{lemma}

\begin{lemma}\label{bound:RiUit}
There exist universal constants $C_1,C_2,\ldots,C_N$ such that, 
\begin{eqnarray}
\|R_iU_{it}\|_i^2&\le& C_i^2(h_i^{-1}+Z_t'Z_t),\,\,\,\,\textrm{for any $i\in[N],t\in[N]$,}\label{bound:eqn:RiUit}\\
\|\theta\|_{i,\sup}&\le&C_i(1+h_i^{-1/2})\|\theta\|_i,\,\,\,\,\textrm{for any $\theta\in\Theta_i$.}
\label{two:norms:isup:i}
\end{eqnarray}
\end{lemma}

\begin{proposition}\label{prop:maximum:term}
Under Assumption \ref{A1}, as $T\rightarrow\infty$, 
$\max_{1\le t\le T}\|Z_t\|_2=O_P(T^{1/\alpha})$.
\end{proposition}

The following proposition holds for both (1) $T,N\rightarrow\infty$; (2) $N\rightarrow\infty$, $T$ is fixed.
That is, the result holds for $M\to\infty$.
\begin{proposition}\label{prop:concentration}
Let Assumptions \ref{A1}--\ref{A2} hold.
For $i\in[N]$ and $t\in[T]$, let $p_i=p_i(M)\ge1$ be a deterministic sequence indexed by $M$, 
and let $\psi_{i,M,t}(U_{it};\theta)$ 
be a real-valued function defined on $\Theta_i$ 
such that $\psi_{i,M,t}(U_{it};0)\equiv0$, and for any
$\theta_1,\theta_2\in\Theta_i$,
\[
\|(\psi_{i,M,t}(U_{it};\theta_1)
-\psi_{i,M,t}(U_{it};\theta_2))R_iU_{it}\|_i\le \|\theta_1-\theta_2\|_{i,\sup}.
\]
Then there exists a universal constant $C_0>0$ such that, as $N\rightarrow\infty$,
\[
P\left(\max_{i\in[N]}\sup_{\theta\in\mathcal{G}_i(p_i)}\frac{\sqrt{T}\|\mathbb{Z}_{iM}(\theta)\|_i}
{\sqrt{T}J_i(p_i,\|\theta\|_{i,\sup})+1}\ge C_0\sqrt{\log{N}+\log\log(TJ_i(p_i,1))}\right)\rightarrow0,
\]
where
\[
\mathbb{Z}_{iM}(\theta)=\frac{1}{\sqrt{T}}\sum_{t=1}^T[\psi_{i,M,t}(U_{it};\theta)R_iU_{it}
-E\left(\psi_{i,M,t}(U_{it};\theta)R_iU_{it}\right)],\,\,\theta\in\Theta_i.
\]
\end{proposition}
Proofs of Lemmas \ref{lemma:id} and \ref{bound:RiUit},
Propositions \ref{prop:maximum:term} and \ref{prop:concentration} 
can be found in supplement document.

\begin{proof}[Proof of Theorem \ref{them:convergence:rate:result:1}]

Since $Y_{it}=g_{i0}(X_{it})+Z_t'\beta_{i0}+e_{it}$, it follows that
\[ S_{i,M,\eta_i}^\star(\theta_{i0})=E\{S_{i,M,\eta_i}(\theta_{i0})\}=E\{-\frac{1}{T}\sum_{t=1}^{T}e_{it}R_iU_{it}+P_i\theta_{i0}\}.
\]
Also  $e_{it}=\epsilon_{it} - \gamma_{2i}'(\bar{\Gamma}_2\bar{\Gamma}_2')^{-1}\bar{\Gamma}_2\bar{v}_t =  \epsilon_{it} - \Delta_i \bar{v}_t$, so we have
\begin{eqnarray*}
\|S_{i,M,\eta_i}^\star(\theta_{i0})\|_i&=&\|E\{(\epsilon_{it} - \Delta_i \bar{v}_t)R_iU_{it}-P_i\theta_{i0}\}\|_i \\
&\leq& \| E\{(\epsilon_{it} - \Delta_i \bar{v}_t)R_iU_{it} \|_i + \|P_i\theta_{i0}\|_i \\
&=& \sup_{\|\theta\|_i=1} |\langle E\{(\epsilon_{it} - \Delta_i \bar{v}_t)R_iU_{it}\}, \theta \rangle_i| + \|P_i\theta_{i0}\|_i \\
&=& \sup_{\|\theta\|_i=1} | E\{(\epsilon_{it} - \Delta_i \bar{v}_t) (g(X_{it})+Z_t'\beta_i)\}| + \|P_i\theta_{i0}\|_i \\
&=& \sup_{\|\theta\|_i=1} | E\{\Delta_i \bar{v}_t (g(X_{it})+Z_t'\beta_i)\}| + \|P_i\theta_{i0}\|_i.
\end{eqnarray*}
Since 
\[
|g(X_{it})+Z_t'\beta_i| \leq (1 + \|Z_t\|_2)\|\theta\|_{i,sup} \leq C_i(1 + \|Z_t\|_2)(1+h_i^{-1/2})\|\theta\|_i
\]
and
\[
E\{\Delta_i \bar{v}_t h_i^{-1/2}\} \leq E\{(\Delta_i \bar{v}_t)^2 \}^{1/2} h_i^{-1/2}=O((Nh_i)^{-1/2}),
\]
there exists a constant $C'$, such that
\begin{equation}\label{prop:con:eqn1}
\sup_{\|\theta\|_i=1} | E\{\Delta_i \bar{v}_t (g(X_{it})+Z_t'\beta_i)\}| \leq \frac{C'}{(Nh_i)^{1/2}}.
\end{equation}
In the meantime, we have
\begin{equation}\label{prop:con:eqn2}
\|P_i\theta_{i0}\|_i=
\sup_{\|\theta\|_i=1}|\langle P_i\theta_{i0},\theta\rangle_i| =\sup_{\|\theta\|_i=1}|\eta_i\langle g_{i0},g\rangle_i|
\le\sqrt{\eta_i}\|g_{i0}\|_{\mathcal{H}_i},\,\,\,\, i\in[N].
\end{equation}

Consider an operator
\[
T_{1i}(\theta)=\theta-S_{i,M,\eta_i}^\star(\theta+\theta_{i0}),\,\,\,\,\theta\in\Theta_i.
\]
By Lemma \ref{lemma:id} we have for any $\theta\in\Theta_i$,
\[
T_{1i}(\theta)=\theta-DS_{i,M,\eta_i}^\star(\theta_{i0})\theta-S_{i,M,\eta_i}^\star(\theta_{i0})
=-S_{i,M,\eta_i}^\star(\theta_{i0}).
\]

Since $T_{1i}$ takes constant value and by (\ref{prop:con:eqn1}) and (\ref{prop:con:eqn2}), 
$T_{1i}$ is a contraction mapping
from $\mathbb{B}_i(\sqrt{\eta_i}\|g_{i0}\|_{\mathcal{H}_i} + \frac{C'}{(Nh_i)^{1/2}})$
to itself, where $\mathbb{B}_i(r)$ represents the $r$-ball in $(\Theta_i,\|\cdot\|_i)$.
By Contraction mapping theorem,
there exists a unique fixed point $\theta'\in\mathbb{B}_i(\sqrt{\eta_i}\|g_{i0}\|_{\mathcal{H}_i}+ \frac{C'}{(Nh_i)^{1/2}})$ such that
$T_{1i}(\theta')=\theta'$. Let $\theta_{\eta_i}=\theta'+\theta_{i0}$,
then $S_{i,M,\eta_i}^\star(\theta_{\eta_i})=0$.
Obviously, $\|\theta_{\eta_i}-\theta_{i0}\|_i\le \sqrt{\eta_i}\|g_{i0}\|_{\mathcal{H}_i} + \frac{C'}{(Nh_i)^{1/2}}$.

We fix an $i\in[N]$ and assume both $T,N$ to approach infinity.
Let $\mathcal{E}_M=\{\max_{1\le t\le T}\|Z_t\|_2\le \widetilde{C}T^{1/\alpha}\}$.
Proposition \ref{prop:maximum:term} says that when $\widetilde{C}$ is large,
$\mathcal{E}_M$ has probability approaching one.
Write $\mathcal{E}_{M,t}=\{\|Z_t\|_2\le \widetilde{C}T^{1/\alpha}\}$.
Then $\mathcal{E}_M=\cap_{t=1}^T\mathcal{E}_{M,t}$.
By Lemma \ref{bound:RiUit}, $\mathcal{E}_{M,t}$ implies that
$\|R_iU_{it}\|_i\le C_i(h_i^{-1/2}+\widetilde{C}T^{1/\alpha})$.

Consider another operator
\[
T_{2i}(\theta)=\theta-S_{i,M,\eta_i}(\theta_{\eta_i}+\theta),\,\,\,\,\theta\in\Theta_i.
\]
For $i\in[N],t\in[T]$, define
\[
\psi_{i,M,t}(U_{it};\theta)=\frac{\langle R_iU_{it},\theta\rangle_i I_{\mathcal{E}_{M,t}}}{\widetilde{C}C_iT^{1/\alpha}(h_i^{-1/2}+\widetilde{C}T^{1/\alpha})},\,\,\theta\in\Theta_i.
\]
It is easy to see that on $\mathcal{E}_M$, for any $\theta_1=(\beta_1,g_1),\theta_2=(\beta_2,g_2)\in\Theta_i$, by Proposition \ref{prop:Ri:Pi},
\begin{eqnarray}\label{prop:convergence:rate:eqn:1}
&&\|(\psi_{i,M,t}(U_{it};\theta_1)-\psi_{iMt}(U_{it};\theta_2))R_iU_{it}\|_i\nonumber\\
&=&\frac{|\langle R_iU_{it},\theta_1-\theta_2\rangle_i|\times\|R_iU_{it}\|_i}{C_i\widetilde{C}T^{1/\alpha}(h_i^{-1/2}
+\widetilde{C}T^{1/\alpha})}I_{\mathcal{E}_{M,t}}\nonumber\\
&=&\frac{|(g_1-g_2)(X_{it})+Z_t'(\beta_1-\beta_2)|\times\|R_iU_{it}\|_i}{C_i\widetilde{C}T^{1/\alpha}(h_i^{-1/2}
+\widetilde{C}T^{1/\alpha})}I_{\mathcal{E}_{M,t}}\nonumber\\
&\le&\frac{\|\theta_1-\theta_2\|_{i,\sup}\widetilde{C}T^{1/\alpha}C_i(h_i^{-1/2}+\widetilde{C}T^{1/\alpha})}{C_i\widetilde{C}T^{1/\alpha}(h_i^{-1/2}+\widetilde{C}T^{1/\alpha})}I_{\mathcal{E}_{M,t}}
\le \|\theta_1-\theta_2\|_{i,\sup}.
\end{eqnarray}

Notice the following decomposition:
\begin{eqnarray*}
T_{2i}(\theta)&=&\theta-S_{i,M,\eta_i}(\theta+\theta_{\eta_i})+S_{i,M,\eta_i}(\theta_{\eta_i})-S_{i,M,\eta_i}(\theta_{\eta_i})\\
&=&\theta-DS_{i,M,\eta_i}(\theta_{\eta_i})\theta-S_{i,M,\eta_i}(\theta_{\eta_i}).
\end{eqnarray*}
We first examine $S_{i,M,\eta_i}(\theta_{\eta_i})$ as follows:
\begin{eqnarray*}
S_{i,M,\eta_i}(\theta_{\eta_i})&=&S_{i,M,\eta_i}(\theta_{\eta_i})-E(S_{i,M,\eta_i}(\theta_{\eta_i}))\\
&=&-\frac{1}{T}\sum_{t=1}^T[(Y_{it}-\langle R_iU_{it},\theta_{\eta_i}\rangle_i)R_iU_{it}
-E((Y_{it}-\langle R_iU_{it},\theta_{\eta_i}\rangle_i)R_iU_{it})]\\
&=&-\frac{1}{T}\sum_{t=1}^T[e_{it}R_iU_{it}-E(e_{it}R_iU_{it})]\\
&&+\frac{1}{T}\sum_{t=1}^T[\langle R_iU_{it},\theta_{\eta_i}-\theta_{i0}\rangle_iR_iU_{it}
-E(\langle R_iU_{it},\theta_{\eta_i}-\theta_{i0}\rangle_iR_iU_{it})]
\end{eqnarray*}
Define $\xi_{it}=e_{it}R_iU_{it}$.
Following \citet[eqn. (3.2)]{dehling1983} and \citet[eqn. (1.11)]{bradley2005}, 
\begin{eqnarray*}
&&E\|\sum_{t=1}^T[e_{it}R_iU_{it}-E(e_{it}R_i U_{it})]\|_i^2\\
&=&E\|\sum_{t=1}^T[\xi_{it}-E(\xi_{it})]\|_i^2\\
&=&\sum_{t,t'=1}^T[E(\langle\xi_{it},\xi_{it'}\rangle_i)-\langle E(\xi_{it}),E(\xi_{it'})\rangle_i]\\
&\le& \sum_{t,t'=1}^T 15(\phi(|t-t'|)/2)^{1-4/\alpha}E(\|\xi_{it}\|_i^{\alpha/2})^{4/\alpha}.
\end{eqnarray*}
It follows from Assumption \ref{A1} \ref{A1:a}, \ref{A1:c}, \ref{A1:d},  and Lemma \ref{bound:RiUit} that
\begin{eqnarray*}
E(\|\xi_{it}\|_i^{\alpha/2})^2&=&
E(|e_{it}|^{\alpha/2}\|R_iU_{it}\|_i^{\alpha/2})^2\\
&\le& E(|e_{it}|^{\alpha/2})E(\|R_iU_{it}\|_i^\alpha)\le  c_0h_i^{-\alpha/2},
\end{eqnarray*}
where $c_0$ is an absolute constant. The existence of such $c_0$
is due to the fact $E(|e_{it}|^\alpha)<\infty$ and 
$E(\|Z_t\|_2^\alpha)<\infty$.
Therefore, it follows from Assumption \ref{A1} (b) that there exists an absolute constant $c_1$ such that
\[
E\|\sum_{t=1}^T[e_{it}R_iU_{it}-E(e_{it}R_iU_{it})]\|_i^2\le c_1Th_i^{-1}.
\]
Similarly, it can be shown that
\begin{eqnarray*}
&&E\|\sum_{t=1}^T[\langle R_iU_{it},\theta_{\eta_i}-\theta_{i0}\rangle_iR_iU_{it}
-E(\langle R_iU_{it},\theta_{\eta_i}-\theta_{i0}\rangle_iR_iU_{it})]\|_i^2\\
&\le&\sum_{t,t'=1}^T15(\phi(|t-t'|)/2)^{1-4/\alpha}E(\|R_iU_{it}\|_i^\alpha)^{4/\alpha}\|\theta_{\eta_i}-\theta_{i0}\|_i^2\\
&\le&c_1'Th_i^{-1},
\end{eqnarray*}
where $c_1'$ is an absolute constant. The last step follows from Proposition \ref{bound:RiUit},
i.e., 
\[
E(\|R_iU_{it}\|_i^\alpha)=O(h_i^{-\alpha/2}),
\]
and the fact $\|\theta_{\eta_i}-\theta_{i0}\|^2_i=O(\eta_i+\frac{1}{Nh_i})$, and the condition $\eta_i+\frac{1}{Nh_i}=O(h_i)$.

Therefore, we can choose $c_2$ to be large such that, with probability approaching one,
\[
\|S_{i,M,\eta_i}(\theta_{\eta_i})\|_i\le c_2(Th_i)^{-1/2}.
\]

On $\mathcal{E}_{M,t}$, for any unequal $\theta_1,\theta_2\in\Theta_i$, define
\[
\theta=\frac{\theta_1-\theta_2}{C_i(1+h_i^{-1/2})\|\theta_1-\theta_2\|_i}.
\] 
Write $\theta=(\beta,g)$.
Hence,  by Lemma \ref{bound:RiUit} (\ref{two:norms:isup:i}),
\begin{eqnarray*}
\|\theta\|_{i,\sup}&\le& 1,\\
\eta_i\|g\|_{\mathcal{H}_i}^2
&\le&\|\theta\|_i^2=\frac{\|\theta_1-\theta_2\|_i^2}{C_i^2(1+h_i^{-1/2})^2\|\theta_1-\theta_2\|_i^2}
\le C_i^{-2}h_i.
\end{eqnarray*}
This means that $\theta\in \mathcal{G}_i(p_i)$ with $p_i=C_i^{-1}(\eta_i^{-1}h_i)^{1/2}$.
Since $\eta_i^{-1}h_i$ tends to infinity as $(N,T)$ does, it is not of loss of
generality to assume that $p_i\ge1$.
Define 
\[
\mathbb{Z}_{iM}(\theta)=\frac{1}{\sqrt{T}}\sum_{t=1}^T[\psi_{i,M,t}(U_{it};\theta)R_iU_{it}
-E(\psi_{i,M,t}(U_{it};\theta)R_iU_{it})].
\]
It follows from (\ref{prop:convergence:rate:eqn:1}) and Proposition \ref{prop:concentration} that, with probability approaching one,
\begin{equation}\label{prop:convergence:rate:eqn:2}
\sup_{\theta\in\mathcal{G}_i(p_i)}\frac{\sqrt{T}\|Z_{iM}(\theta)\|_i}{\sqrt{T}J_i(p_i,\|\theta\|_{i,\sup})+1}\le C_0\sqrt{\log{N}+\log\log(TJ_i(p_i,1))}.
\end{equation}
Since $h_i=o(1)$, assume that $h_i^{-1}\ge1$. It follows from Lemma \ref{bound:RiUit} (\ref{bound:eqn:RiUit}) that
\begin{eqnarray}\label{prop:convergence:rate:eqn:3}
&&\|E\left(\langle R_iU_{it},\theta\rangle_iI_{\mathcal{E}_{M,t}^c}R_iU_{it}\right)\|_i\nonumber\\
&\le&E\left(|\langle R_iU_{it},\theta\rangle_i|I_{\mathcal{E}_{M,t}^c}
\|R_iU_{it}\|_i\right)\nonumber\\
&\le&E\left((1+\|Z_t\|_2)I_{\mathcal{E}_{M,t}^c}C_i(h_i^{-1/2}+\|Z_t\|_2)\right)\nonumber\\
&\le&C_ih_i^{-1/2}E\left((1+\|Z_t\|_2)^2I_{\mathcal{E}_{M,t}^c}\right)\nonumber\\
&\le&C_ih_i^{-1/2}E\left((1+\|Z_t\|_2)^\alpha\right)^{2/\alpha}
P(\mathcal{E}_{M,t}^c)^{1-2/\alpha}\nonumber\\
&\le&C_ih_i^{-1/2}E\left((1+\|Z_t\|_2)^\alpha\right)^{2/\alpha}\left(\frac{1}{\widetilde{C}^\alpha T}E(\|Z_t\|_2^\alpha)\right)^{1-2/\alpha}.
\end{eqnarray}
Consequently, with probability approaching one,
for any unequal $\theta_1,\theta_2\in\Theta_i$ On $\mathcal{E}_{M,t}$, it follows from 
(\ref{prop:convergence:rate:eqn:2}) and (\ref{prop:convergence:rate:eqn:3}) that
\begin{eqnarray}\label{prop:convergence:rate:eqn:4}
&&\|T_{2i}(\theta_1)-T_{2i}(\theta_2)\|_i\nonumber\\
&=&\bigg\|-\frac{1}{T}\sum_{t=1}^T[\langle R_iU_{it},\theta_1-\theta_2\rangle_iR_iU_{it}-E(\langle R_iU_{it},\theta_1-\theta_2\rangle_iR_iU_{it})]\bigg\|_i\nonumber\\
&=&\bigg\|-\frac{1}{T}\sum_{t=1}^T[\langle R_iU_{it},\theta\rangle_i R_iU_{it}-E(\langle R_iU_{it},
\theta\rangle_iR_iU_{it})]\times\|\theta_1-\theta_2\|_iC_i(1+h_i^{-1/2})\bigg\|_i\nonumber\\
&=&\|\theta_1-\theta_2\|_iC_i(1+h_i^{-1/2})
\bigg\|\left(-\frac{1}{T}\sum_{t=1}^T[\langle R_iU_{it},\theta\rangle_i I_{\mathcal{E}_{M,t}}R_iU_{it}-E(\langle R_iU_{it},
\theta\rangle_i I_{\mathcal{E}_{M,t}}R_iU_{it})]\right.\nonumber\\
&&\left.+E(\langle R_iU_{it},\theta\rangle_i I_{\mathcal{E}_{M,t}^c}R_iU_{it})\right)\bigg\|_i\nonumber\\
&=&\|\theta_1-\theta_2\|_iC_i(1+h_i^{-1/2})
\bigg\|\left(-T^{-1/2} C_i\widetilde{C}T^{1/\alpha}(h_i^{-1/2}+\widetilde{C}T^{1/\alpha})\mathbb{Z}_{iM}(\theta)+E(\langle R_iU_{it},\theta\rangle_i I_{\mathcal{E}_{M,t}^c}R_iU_{it})\right)\bigg\|_i\nonumber\\
&\le&\|\theta_1-\theta_2\|_iC_i(1+h_i^{-1/2})
\left(T^{-1/2} C_0C_i\widetilde{C}T^{1/\alpha}(h_i^{-1/2}+\widetilde{C}T^{1/\alpha})(J_i(p_i,1)+T^{-1/2})\right.\nonumber\\
&&\left.\times\sqrt{\log{N}+\log\log(TJ_i(p_i,1))}
+C_ih_i^{-1/2}E\left((1+\|Z_t\|_2)^\alpha\right)^{2/\alpha}\left(\frac{1}{\widetilde{C}^\alpha T}E(\|Z_t\|_2^\alpha)\right)^{1-2/\alpha}\right)\nonumber\\
&\le& c_3\|\theta_1-\theta_2\|_i,\nonumber\\
\end{eqnarray}
where $c_3$ is a constant in $(0,1/2)$. Note that (\ref{prop:convergence:rate:eqn:4})
holds also for $\theta_1=\theta_2$. The existence of such $c_3$ follows by 
condition $b_{N,p}=o_P(\sqrt{N}h)$.

In particular, letting $\theta_2=0$, one gets that for any $\theta_1\in\mathbb{B}(2c_2(Th_i)^{-1/2})$,
\begin{eqnarray*}
\|T_{2i}(\theta_1)\|_i&\le&\|T_{2i}(\theta_1)-T_{2i}(0)\|_i+\|T_{2i}(0)\|_i\\
&\le&c_3\|\theta_1\|_i+\|S_{i,M,\eta_i}(\theta_{\eta_i})\|_i\\
&\le&2c_2c_3(Th_i)^{-1/2}+c_2(Th_i)^{-1/2}<2c_2(Th_i)^{-1/2}.
\end{eqnarray*}
This implies that, with probability approaching one, $T_{2i}$ is a contraction mapping from
$\mathbb{B}(2c_2(Th_i)^{-1/2})$ to itself.
By contraction mapping theorem,
there exists uniquely a $\theta''\in\mathbb{B}(2c_2(Th_i)^{-1/2})$ such that
$T_{2i}(\theta'')=\theta''$, implying that $S_{i,M,\eta_i}(\theta_{\eta_i}+\theta'')=0$.
Thus, $\widehat{\theta}_i=\theta_{\eta_i}+\theta''$ is the penalized MLE of $\ell_{i,M,\eta_i}$.
This further shows that $\|\widehat{\theta}_i-\theta_{\eta_i}\|_i\le 2c_2(Th_i)^{-1/2}$.
Combined with $\|\theta_{\eta_i}-\theta_{i0}\|_i=O(\eta_i^{1/2}+(Nh_i)^{-1/2})$,
we have
$\|\widehat{\theta}_i-\theta_{i0}\|_i=O_P((Th_i)^{-1/2}+\eta_i^{1/2}+(Nh_i)^{-1/2})$.
\end{proof}

\subsection{Proofs in Section \ref{subsec:AsympDist-Hetero}}\label{proof:section:AsympDist-Hetero}
In this section, we prove Theorems \ref{lemma: hetero -- leading term},
\ref{theorem: joint distr 1.} and \ref{theorem: joint distr 3.}, and Corollary \ref{proposition:estimation:sigma:i:epsilon}.
\begin{proof}[Proof of Theorem \ref{lemma: hetero -- leading term}]
Define $$S_{i,M}(\theta) \equiv -\frac{1}{T}\sum_{t=1}^{T}(Y_{it}-\langle R_iU_{it}, \theta \rangle_i)R_iU_{it}$$ and $$S_{i}(\theta) \equiv E\{S_{i,M}(\theta)\}=E\{-\frac{1}{T}\sum_{t=1}^{T}(Y_{it}-\langle R_iU_{it}, \theta \rangle_i)R_iU_{it}\}.$$ 
Recall $S_{i,M,\eta_i}=S_{i,M}+P_i\theta$ and $S_{i,M,\eta_i}^\star(\theta)=S_i(\theta)+P_i\theta$.
Denote $\theta_i=\widehat{\theta}_i-\theta_{i0}$. Since $S_{i,M,\eta_i}(\widehat{\theta_i})=0$, we have 
$S_{i,M,\eta_i}(\theta_i+\theta_{i0})=0$. 
Therefore, 
\begin{eqnarray}
&& \| S_{i,M}(\theta_i+\theta_{i0})-S_{i}(\theta_i+\theta_{i0}) - (S_{i,M}(\theta_{i0})-S_{i}(\theta_{i0})) \|_i \nonumber \\
&=& \| S_{i,M,\eta_i}(\theta_i+\theta_{i0})-S_{i,M,\eta_i}^{\star}(\theta_i+\theta_{i0}) - (S_{i,M,\eta_i}(\theta_{i0})-S_{i,M,\eta_i}^{\star}(\theta_{i0})) \|_i \nonumber \\
&=& \|S_{i,M,\eta_i}^{\star}(\theta_i+\theta_{i0}) + S_{i,M,\eta_i}(\theta_{i0})-S_{i,M,\eta_i}^{\star}(\theta_{i0}) \|_i \nonumber \\
&=& \|DS_{i,M,\eta_i}^{\star}(\theta_{i0})\theta_i + S_{i,M,\eta_i}(\theta_{i0}) \|_i \nonumber \\
&=& \|\theta_i + S_{i,M,\eta_i}(\theta_{i0}) \|_i .
\end{eqnarray}

Consider an event $B_{i,M}=\{ \|\theta\|_i \leq r_{i,M} \equiv C_B ((Th_i)^{-1/2})+\eta_i^{1/2}+(Nh_i)^{-1/2} \}$. For some $C_B$ large enough, $B_{i,M}$ has probability approaching one. Let $d_{i,M}=C_i r_{i,M}(1+h_i^{-1/2})$, 
where $C_i$ is defined in lemma \ref{bound:RiUit}. We have $d_{i,M}=o(1)$. For any $\theta \in \Theta_i$, we further define $\bar{\theta}=(\bar{\beta}, \bar{g})=d_{i,M}^{-1} \theta /2$, where $\bar{\beta}=d_{i,M}^{-1}\beta/2$ and $\bar{g}=d_{i,M}^{-1}g/2$. Then, on event $B_{i,M}$, we have 
\begin{eqnarray*}
\| \bar{\theta} \|_{i,sup} \leq C_i (1+h_i^{-1/2}) \|\bar{\theta}\|_i= C_i (1+h_i^{-1/2}) d_{i,M}^{-1} \| \theta \|_i/2\leq\frac{1}{2} .
\end{eqnarray*}
Meanwhile,
\[
\| \bar{g} \|_{H_i}^2 = \frac{d_{i,M}^{-2}}{4} \eta_i^{-1} (\eta_i \| g \|_{H_i}^2)
\leq \frac{d_{i,M}^{-2}}{4} \eta_i^{-1} \|\theta\|_i^2 \leq \frac{d_{i,M}^{-2}}{4} \eta_i^{-1} r_{i,M}^2 \leq C_i^{-2}h_i \eta_i^{-1} .
\]
Let $p_i=C_i^{-1} (h_i \eta_i^{-1})^{1/2}$. Then $\|\bar{g}\|_{H_i} \leq p_i$. Therefore $\bar{\theta} \in \mathcal{G}_i(p_i)$. Since $(\eta_i h_i^{-1}) \to \infty$ as $(N, T) \to \infty$, $p_i > 1$ in general.

Recall $\mathcal{E}_{M,t}=\{\|Z_t\|_2\le \widetilde{C}T^{1/\alpha}\}$, as defined in the proof of Theorem \ref{them:convergence:rate:result:1}. Let
\[
\psi_{i,M,t}^d(U_{it};\bar{\theta})=\frac{\langle R_iU_{it},\theta\rangle_i I_{\mathcal{E}_{M,t}}}{2 d_M  \widetilde{C}C_iT^{1/\alpha}(h_i^{-1/2}+\widetilde{C}T^{1/\alpha})},\,\,\theta\in\Theta_i.
\]
Following the proof of Theorem \ref{them:convergence:rate:result:1}, on $\mathcal{E}_M$, for any $\theta_1=(\beta_1,g_1),\theta_2=(\beta_2,g_2)\in\Theta_i$, we have
\begin{eqnarray}
\|(\psi_{i,M,t}^d(U_{it};\bar{\theta}_1)-\psi_{iMt}^d(U_{it};\bar{\theta}_2))R_iU_{it}\|_i \le \|\bar{\theta}_1-\bar{\theta}_2\|_{i,\sup}.
\end{eqnarray}

Define
\[
\mathbb{Z}_{iM}^d(\bar{\theta})=\frac{1}{\sqrt{T}}\sum_{t=1}^T[\psi_{i,M,t}^d(U_{it};\bar{\theta})R_iU_{it}
-E(\psi_{i,M,t}^d(U_{it};\bar{\theta})R_iU_{it})].
\]
It follows from Proposition \ref{prop:concentration} that, with probability approaching one,
\begin{equation}
\sup_{\theta\in\mathcal{G}_i(p_i)}\frac{\sqrt{T}\|Z_{iM}^d(\bar{\theta})\|_i}{\sqrt{T}J_i(p_i,\|\bar{\theta}\|_{i,\sup})+1}\le C_0\sqrt{\log{N}+\log\log(TJ_i(p_i,1))}.
\end{equation}

Therefore 
\small
\begin{eqnarray}
&&\|\theta_i + S_{i,M,\eta_i}(\theta_{i0}) \|_i  \nonumber \\
&=& \| S_{i,M}(\theta_i+\theta_{i0})-S_{i}(\theta_i+\theta_{i0}) - (S_{i,M}(\theta_{i0})-S_{i}(\theta_{i0})) \|_i \nonumber \\
&=& \bigg\|\frac{1}{T} \sum_{t=1}^{T} [\langle R_i U_{it}, \theta_i \rangle_i R_iU_{it} - E\{\langle R_i U_{it}, \theta_i \rangle_i R_iU_{it}\}] \bigg\|_i \nonumber \\
&=& \bigg\|\frac{1}{T}\sum_{t=1}^T[\langle R_iU_{it},\theta_i\rangle_i I_{\mathcal{E}_{M,t}}R_iU_{it}-E(\langle R_iU_{it},
\theta_i\rangle_i I_{\mathcal{E}_{M,t}}R_iU_{it})] - E(\langle R_iU_{it},\theta_i\rangle_i I_{\mathcal{E}_{M,t}^c}R_iU_{it})\bigg\|_i \nonumber \\
&=& \bigg\| 2 d_M  \widetilde{C}C_iT^{1/\alpha-1}(h_i^{-1/2}+\widetilde{C}T^{1/\alpha})\left(\sqrt{T} Z_{i,M}^d(\bar{\theta}) \right) - E(\langle R_iU_{it},\theta_i\rangle_i I_{\mathcal{E}_{M,t}^c}R_iU_{it}) \bigg\|_i \nonumber \\
&\leq& 2 d_M  C_0\widetilde{C}C_iT^{1/\alpha-1}(h_i^{-1/2}+\widetilde{C}T^{1/\alpha})\left(\sqrt{T}J_i(p_i,\|\bar{\theta}\|_{i,\sup})+1\right)\sqrt{\log{N}+\log\log(TJ_i(p_i,1))} \nonumber \\
&& + C_ih_i^{-1/2}E\left((1+\|Z_t\|_2)^\alpha\right)^{2/\alpha}\left(\frac{1}{\widetilde{C}^\alpha T}E(\|Z_t\|_2^\alpha)\right)^{1-2/\alpha} .
\end{eqnarray}
\end{proof}


\begin{proof}[Proof of Theorem \ref{theorem: joint distr 1.}]
Define $\widehat{\theta}_i^h = ( \widehat{\beta}_i, h_i^{1/2}\widehat{g}_i )$, $\theta_{i0}^{*h} = ( \beta_{i0}^{\star}, h_i^{1/2}g_{i0}^{\star} )$, and $R_i^h u=( H_u^{(i)}, h_i^{1/2}T_u^{(i)} )$, where $\theta_{i0}^{\star}=(id-P_i)\theta_{i0}$. From Theorem \ref{lemma: hetero -- leading term}, we have 
\begin{equation}
\| \widehat{\theta}_i -\theta_{i0} + S_{i,M,\eta_i}(\theta_{i0}) \|_i = O_P(a_M).
\end{equation}
Since 
\[
S_{i,M,\eta_i}(\theta_{i0}) = - \frac{1}{T} \sum_{t=1}^{T} \left( Y_{it} - \langle R_iU_{it}, \theta_{i0} \rangle_i \right) R_iU_{it} + P_i\theta_{i0}
= - \frac{1}{T} \sum_{t=1}^{T} e_{it} R_iU_{it} + P_i\theta_{i0},
\]
Theorem \ref{lemma: hetero -- leading term} can be re-written as 
\begin{equation}
\| \widehat{\theta}_i -\theta_{i0}^{\star}  - \frac{1}{T} \sum_{t=1}^{T} e_{it} R_iU_{it} \|_i = O_P(a_M).
\end{equation} 
It implies $\| \widehat{\beta}_i - \beta_{i0}^{\star} - \frac{1}{T} \sum_{t=1}^{T}e_{it}H_{it}^{(i)} \|_{2} = O_P(a_M)$.
Further, we define $Rem = \widehat{\theta}_i -\theta_{i0}^{\star}  - \frac{1}{T} \sum_{t=1}^{T} e_{it} R_iU_{it}$ and $Rem^h = \widehat{\theta}_i^h -\theta_{i0}^{*h}  - \frac{1}{T} \sum_{t=1}^{T} e_{it} R_i^hU_{it}$. 
Then 
\begin{align*}
\| Rem^h - h_i^{1/2}Rem \|_i & = \left\| \left( (1-h_i^{1/2})(\widehat{\beta}_i - \beta_{i0}^{\star} - \frac{1}{T} \sum_{t=1}^{T} e_{it} H_{U_{it}}^{(i)}), 0 \right) \right\|_i \\
& \leq (1-h_i^{1/2})O\left(\left\| \widehat{\beta}_i - \beta_{i0}^{\star} - \frac{1}{T} \sum_{t=1}^{T} e_{it} H_{U_{it}}^{(i)}) \right\|_{2}\right)= O_P(a_M).
\end{align*}
Therefore, $\| Rem^h \|_i \leq \| Rem^h - h_i^{1/2}Rem \|_i + \|h_i^{1/2}Rem \|_i = O_P(a_M)$. 

The idea is to employ the Cram$\acute{\textrm{e}}$r-Wold device. For any $z$, we will obtain the limiting distribution of $T^{1/2}z'(\widehat{\beta}_i - \beta_{i0}^{\star}) + (Th_i)^{1/2}(\widehat{g}_i(x_0)-g_{i0}^{\star}(x_0))$, which is $T^{1/2}\langle R_iu, \widehat{\theta}_i^h - \theta_{i0}^{*h} \rangle_i$ by Proposition \ref{prop:Ri:Pi}, 
where $u=(x_0, z)$.

Since $T^{1/2}h^{-1/2}a_M=o(1)$, we have 
\begin{align*}
& \left| T^{1/2}\langle R_iu, \widehat{\theta}_i^h - \theta_{i0}^{*h} - \frac{1}{T}\sum_{t=1}^{T} e_{it}R_i^hU_{it} \rangle_i  \right| \\
& \leq T^{1/2} \| R_iu \|_i \| Rem^h \|_i \\
& = O_P(T^{1/2}h^{-1/2}a_M) = o_P(1).
\end{align*}
Then, to find the limiting distribution of $T^{1/2}\langle R_iu, \widehat{\theta}_i^h - \theta_{i0}^{*h} \rangle_i$, we only need to find the limiting distribution of $T^{1/2}\langle R_iu, \frac{1}{T}\sum_{t=1}^{T} e_{it}R_i^hU_{it} \rangle_i = T^{-1/2}\sum_{t=1}^{T} e_{it}(z'H_{U_{it}^{(i)}} + h^{1/2}T_{U_{it}}^{(i)}(x_0))$. Next we will use CLT to find its limiting distribution. 

Define $L(U_{it})=z'H_{U_{it}}^{(i)} + h^{1/2}T_{U_{it}}^{(i)}(x_0)$. Since $\epsilon_{it}$ and $v_{it}$ are 
i.i.d. across $t$, we have
\begin{align} \label{eq: variance}
s_T^2 & = Var\left( \sum_{t=1}^{T} e_{it}(z'H_{U_{it}}^{(i)} + h^{1/2}T_{U_{it}}^{(i)}(x_0)) \right) \nonumber \\
& = T \cdot Var\left( e_{it}L(U_{it}) \right) +  \sum_{t_1\neq t_2}^{T} Cov\left( e_{it_1}L(U_{it_1}), e_{it_2}L(U_{it_2}) \right) \nonumber \\
& = T\cdot E\left\{ e_{it}^2L(U_{it})^2 \right\} - T \cdot  E\left\{ e_{it}L(U_{it}) \right\}^2 +  \sum_{t_1\neq t_2}^{T} Cov\left( e_{it_1}L(U_{it_1}), e_{it_2}L(U_{it_2}) \right).
\end{align}

For the first term, 
\begin{align*}
E\left\{ e_{it}^2 L(U_{it})^2 \right\} = E\left\{ \left( \epsilon_{it} - \Delta_i \bar{v}_t \right)^2 L(U_{it})^2 \right\}
= \sigma^2_{\epsilon} E\left\{ L(U_{it})^2 \right\} + E\left\{ (\Delta_i \bar{v}_t)^2 L(U_{it})^2 \right\}.
\end{align*}
From Cauchy-Schwarz and H$\ddot{\textrm{o}}$lder's inequality, we can show that 
\begin{align*}
E\left\{ (\Delta_i \bar{v}_t)^2 L(U_{it})^2 \right\} & \leq \|\Delta_i\|_{2}^2 E\left\{ \|\bar{v}_t\|_{2}^2 L(U_{it})^2 \right\} \nonumber \\
& \leq \|\Delta_i\|_{2}^2 \left( E\left\{ \|\bar{v}_t\|_{2}^\alpha \right\} \right)^{2/\alpha} \left( E\left\{ |L(U_{it})|^{\frac{2\alpha}{\alpha-2}} \right\} \right)^{\frac{\alpha-2}{\alpha}}.
\end{align*}
We next will find the upper bound of $|L(U_{it})|$. 
\begin{align*}
L(U_{it}) = & z'H_{U_{it}}^{(i)} + h^{1/2}T_{U_{it}}^{(i)}(x_0) \\
= & z'(\Omega_i + \Sigma_i)^{-1} Z_t - z'(\Omega_i + \Sigma_i)^{-1} A_i(X_{it}) + h_i^{1/2}K_{X_{it}}^{(i)}(x_0) \\
& - h_i^{1/2} A_i'(x_{0})(\Omega_i + \Sigma_i)^{-1} Z_t + h_i^{1/2} A_i'(x_{0})(\Omega_i + \Sigma_i)^{-1} A_i(X_{it}).
\end{align*}
By the proof of lemma \ref{bound:RiUit}, we have
\begin{align*}
\left| z'(\Omega_i + \Sigma_i)^{-1} Z_t \right| &  \leq \left| c_1^{-1}z'Z_t \right| \leq c_{1z} \|Z_t\|_{2}, \\
\left| z'(\Omega_i + \Sigma_i)^{-1} A_i(X_{it}) \right| & = \left| z'(\Omega_i + \Sigma_i)^{-1/2} (\Omega_i + \Sigma_i)^{-1/2} A_i(X_{it}) \right| \\
& \leq \sqrt{z'(\Omega_i + \Sigma_i)^{-1} z} \sqrt{A_i'(X_{it})(\Omega_i + \Sigma_i)^{-1} A_i(X_{it})} \\
& \leq c_{2z}h_i^{-1/2}, \\
\left| A_i'(x_{0})(\Omega_i + \Sigma_i)^{-1} Z_t \right| & \leq c_{3z}h_i^{-1/2}\|Z_t\|_{2},  \\
\left|K_{X_{it}}^{(i)}(x_0)\right| & \leq C_{\varphi,i}^2h_i^{-1}, \\
\left| A_i'(x_{0})(\Omega_i + \Sigma_i)^{-1} A_i(X_{it}) \right| & \leq c_1^{-1}C_{\varphi,i}^2 C_{G_i}^2 h_i^{-1},
\end{align*}
where $c_{1z}=c_1^{-1}\|z\|_{2}$, $c_{2z}=c_1^{-1}C_{\varphi,i}C_{G_i} \|z\|_{2}$, and $c_{3z}=c_1^{-1}C_{\varphi,i}C_{G_i}$. Combine all these inequality together, we have 
\begin{align*} \label{eq: LU_it}
\left| L(U_{it}) \right| \leq c_{4z} \| Z_t \|_{2} + c_{5z}h_i^{-1/2},
\end{align*}
where $c_{4z}=c_{1z}+c_{3z}$ and $c_{5z}=c_{2z}+C_{\varphi,i}^2+c_1^{-1}C_{\varphi,i}^2C_{G_i}^2$.
Therefore, there exists a constant $c_{6z}$, such that 
\begin{align*}
E\left\{ |L(U_{it})|^{\frac{2\alpha}{\alpha-2}} \right\} & = E\left\{ \left|c_{4z} \| Z_t \|_{2} + c_{5z}h_i^{-1/2} \right|^{\frac{2\alpha}{\alpha-2}} \right\} \\
& \leq c_{6z} \left( E\left\{ \|Z_t\|_{2}^{\frac{2\alpha}{\alpha-2}} \right\} + h_i^{-\frac{\alpha}{\alpha-2}} \right) \\
& \leq c_{6z} \left( E\left\{ \|Z_t\|_{2}^{\alpha} \right\}^{\frac{2}{\alpha-2}} + h_i^{-\frac{\alpha}{\alpha-2}} \right).
\end{align*}
Since $Z_t=({f}_{1t}', \bar{X}_t')'$ and $\bar{X}_t  = \bar{\Gamma}_1' {f}_{1t} + \bar{\Gamma}_2' {f}_{2t} + \bar{{v}}_t$, we have
\begin{align*}
 E\left\{ \|Z_t\|_{2}^{\alpha} \right\} & = E\left\{ \left( \|f_{1t}\|_{2}^2 +  \|\bar{X}_{t}\|_{2}^2 \right)^{\alpha/2} \right\} \\
 & \leq c_{7} \left( E\left\{  \|f_{1t}\|_{2}^\alpha \right\} + E\left\{ \|\bar{X}_{t}\|_{2}^{\alpha}  \right\} \right) \\
 & \leq c_{8} \left( E\left\{  \|f_{1t}\|_{2}^\alpha \right\} + E\left\{ \|f_{1t}\|_{2}^{\alpha} \right\} + E\left\{ \|f_{2t}\|_{2}^{\alpha} \right\} + E\left\{ \|\bar{v}_{t}\|_{2}^{\alpha}  \right\} \right), 
\end{align*}
where $c_7$ and $c_8$ are constants. By Assumption \ref{A1}, $E\left\{ \|f_{1t}\|_{2}^{\alpha} \right\}$, $E\left\{ \|f_{2t}\|_{2}^{\alpha} \right\}$, and $E\left\{ |v_{it}|^{\alpha} \right\}$ are finite. Using Marcinkiewicz and Zygmund inequality in \citet{shao2003}, we can show that
\begin{align*}
E\left\{ \|\bar{v}_{t}\|_{2}^{\alpha}  \right\} & = \frac{1}{N^\alpha} E\left\{ \left( \sum_{l=1}^{d} \left(\sum_{i=1}^{N} v_{ilt}\right)^2 \right)^{\alpha/2} \right\} \\
& \leq c_9 \frac{1}{N^\alpha} \sum_{l=1}^{d} E\left\{ \left( \left(\sum_{i=1}^{N} v_{ilt}\right)^2 \right)^{\alpha/2} \right\} \\
& = c_9 \frac{1}{N^\alpha} \sum_{l=1}^{d} E\left\{ \left( \sum_{i=1}^{N} v_{ilt}\right) ^{\alpha} \right\} \\
& \leq c_{10} \frac{d}{N^\alpha} \frac{1}{N^{1-\alpha/2}} \sum_{i=1}^{N} E\left\{  v_{ilt} ^{\alpha} \right\} \\
& = O_P\left(\frac{1}{N^{\alpha/2}}\right),
\end{align*}
where $c_9$ and $c_{10}$ are two constants. As $N \to \infty$, $E\left\{ \|Z_t\|_{2}^{\alpha} \right\} = O_P(1)$ and $E\left\{ |L(U_{it})|^{\frac{2\alpha}{\alpha-2}} \right\} = O_P(h_i^{-\frac{\alpha}{\alpha - 2}})$.
Thus, we have 
\begin{equation} \label{eq: vt LUit}
E\left\{ (\Delta_i \bar{v}_t)^2 L(U_{it})^2 \right\} \leq \|\Delta_i\|_{2}^2 \left( E\left\{ \|\bar{v}_t\|_{2}^\alpha \right\} \right)^{2/\alpha} \left( E\left\{ |L(U_{it})|^{\frac{2\alpha}{\alpha-2}} \right\} \right)^{\frac{\alpha-2}{\alpha}} = O_P(\frac{1}{Nh_i}).
\end{equation} 
By assumption $(Nh_i)^{-1} = o_P(1)$, we have $E\left\{ (\Delta_i \bar{v}_t)^2 L(U_{it})^2 \right\} = o_P(1)$.

Next, we will find the order of $E\left\{ L(U_{it})^2 \right\}$. As it is shown in \citet{CS15}, as $\eta_i \to 0$ and $T \to \infty$, $E\left( |L(U_{it})|^2 \right) \to \alpha^2_{x_0} + 2 (z + \beta_{x_0})' \Omega^{-1}_i \alpha_{x_0} + (z + \beta_{x_0})'\Omega_i^{-1}(z + \beta_{x_0})$ for any given $z$.

From above derivatives, we have found the leading term of the first term in equation (\ref{eq: variance}). Now, we turn to the second term. It is straightforward to obtain that
$$E\left\{ e_{it}L(U_{it})\right\}^2 = E\left\{ (\epsilon_{it}-\Delta_i \bar{v}_t)L(U_{it})\right\}^2 \leq \|\Delta_i\|_2^2 E\left\{ \|\bar{v}_{t}\|_2^2\right\} E\left\{ L(U_{it})^2 \right\},$$ where $E\left\{ \|\bar{v}_{t}\|_2^2\right\} \leq E\left\{ \|\bar{v}_{t}\|_2^ \alpha\right\}^{2/\alpha} = O_P(1/N)$ and $E\left\{ L(U_{it})^2 \right\}=O_P(1)$. So $E\left\{ e_{it}L(U_{it})\right\}^2 = O_P(1/N)$. So the second term is of a smaller order than the first term.

For the last term of equation (\ref{eq: variance}), we can show that
\begin{align*}
&\sum_{t_1\neq t_2}^{T} Cov\left( e_{it_1}L(U_{it_1}), e_{it_2}L(U_{it_2}) \right) \\
\leq & \sum_{t_1\neq t_2}^{T} | Cov\left( \Delta_i \bar{v}_{t_1}L(U_{it_1}),  \Delta_i \bar{v}_{t_2}L(U_{it_2}) \right) | \\
\leq &  8 \sum_{t_1\neq t_2}^{T} \alpha(|t_1-t_2|)^{1-4/ \alpha} E\left\{ |\Delta_i\bar{v}_{t_1}L(U_{it_1})|^{\alpha/2} \right\} ^ {4/\alpha} \\
\leq & 8 \cdot 2^{1-4/ \alpha} \sum_{t_1=1}^{T} \ \sum_{t_2=1, t_2\neq t_1}^{T} \phi(|t_1-t_2|)^{1-4/ \alpha} E\left\{ |\Delta_i\bar{v}_{t_1}L(U_{it_1})|^{\alpha/2} \right\} ^ {4/\alpha},
\end{align*}
where $\alpha(|t_1-t_2|)$ and $\phi(|t_1-t_2|)$ are the $\alpha-mixing$ and $\phi-mixing$ coefficients for $\Delta_i \bar{v}_t L(U_{it})$. We have $2 \cdot  \alpha(|t_1-t_2|)< \phi(|t_1-t_2|)$. The second inequality is from Proposition 2.5 in \citet{fanYao2003}. Similar to equation (\ref{eq: vt LUit}), we can show that $E\left\{ |\Delta_i\bar{v}_{t_1}L(U_{it_1})|^{\alpha} \right\}=O_P((Nh_i)^{-\alpha/2})$. So $E\left\{ |\Delta_i\bar{v}_{t_1}L(U_{it_1})|^{\alpha/2} \right\} ^ {4/\alpha} = O_P(\frac{1}{Nh_i})$. From Assumption \ref{A1}, we have $\sum_{t_2=1,t_2\neq t_1}^{\infty} \phi(|t_1-t_2|)^{1-4/ \alpha} < \infty$. Thus, 
\begin{equation*}
\sum_{t_1\neq t_2}^{T} Cov\left( e_{it_1}L(U_{it_1}), e_{it_2}L(U_{it_2}) \right)=O_P(\frac{T}{Nh_i}).
\end{equation*}
Again, the third term of equation (\ref{eq: variance}) is of a smaller order than the first term. 

To combine all above equations together, we have
\[
\frac{1}{T} s^2_T \to \sigma_s^2 \qquad
or  \qquad
\frac{1}{T} s^2_T \to (z', 1) \Psi^{\star} (z', 1)',
\] 
where $\sigma_s^2=\sigma^2_{\epsilon} \left( \alpha^2_{x_0} + 2 (z + \beta_{x_0})' \Omega^{-1}_i \alpha_{x_0} + (z + \beta_{x_0})'\Omega_i^{-1}(z + \beta_{x_0}) \right)$.

In order to use the CLT with mixing conditions, we need to show $E\left\{ \left|e_{it}L(U_{it})\right|^{\alpha/2} \right\}$ is finite. 
\begin{align*}
E\left\{ \left|e_{it}L(U_{it})\right|^{\alpha/2} \right\} & \leq E\left\{ |e_{it}|^\alpha \right\}^{1/2} E\left\{ |L(U_{it})|^\alpha \right\}^{1/2} \\
& \leq c_{6z} E\left\{ |e_{it}|^\alpha \right\}^{1/2} E\left\{ \| Z_t \|_2^\alpha + h_i^{-\alpha/2} \right\}^{1/2}.
\end{align*}
Since $E\left\{ \|Z_t\|_{2}^{\alpha} \right\} = O_P(1)$ and $E\left\{ |e_{it}|^\alpha \right\} \leq \infty$, then $E\left\{ \left| e_{it}L(U_{it})\right|^{\alpha/2} \right\} \leq \infty$. From the above proof, we have $E\left\{ e_{it}L(U_{it})\right\}^2 = O_P(1/N)$, which implies 
\begin{equation}\label{eq:bound:eit:L:Uit}
	E\left\{ e_{it}L(U_{it})\right\} = O_P(1/\sqrt{N}). 
\end{equation}
Then $E\left\{ \left| e_{it}L(U_{it}) - E\left\{ e_{it}L(U_{it})\right\} \right|^{\alpha/2} \right\} < \infty$. Since $\phi-$mixing condition is stronger than $\alpha-$mixing, by the Theorem 2.21 of \citet{fanYao2003}, we have 
\begin{align*}
\frac{1}{\sqrt{T}} \left( \sum_{t=1}^{T} \left( e_{it}L(U_{it}) -  E\left\{ e_{it}L(U_{it})\right\} \right) \right) & \overset{d}{\to} N(0, \sigma_s^2).
\end{align*}
\end{proof}


\begin{proof}[Proof of Theorem \ref{theorem: joint distr 3.}]
Notice 
$$\theta_{i0}-\theta_{i0}^{\star}=P_i\theta_{i0}=\begin{pmatrix}
-(\Omega_i+\Sigma_i)^{-1}V_i(G_i,W_ig_{i0})\\
W_ig_{i0}+A_i'(\Omega_i+\Sigma_i)^{-1}V_i(G_i,W_ig_{i0})
\end{pmatrix}.
$$
Hence the result of the theorem holds if we can show that, for any $x\in \mathcal{X}_i$,
\begin{eqnarray*}
	\sqrt{T}E(e_{it}U_{it}^{(i)})=o_P(1),\\
	\sqrt{Th_i}E(e_{it}H_{it}^{(i)}(x)),\\
	\sqrt{T}(\Omega_i+\Sigma_i)^{-1}V_i(G_i,W_ig_{i0})&=&o_P(1),\\
	\sqrt{Th_i}A_i'(x)(\Omega_i+\Sigma_i)^{-1}V_i(G_i,W_ig_{i0})&=&o_P(1),\\
	\alpha_x=\beta_x=0,\\
	\lim_{T\to \infty}\sqrt{Th_i}W_ig_{i0}(x)=0.
\end{eqnarray*}
First by (\ref{eq:bound:eit:L:Uit}), we can see that the follow hold true:
\begin{eqnarray*}
	\sqrt{T}E(e_{it}U_{it}^{(i)})&=&O_P(\sqrt{T/N})=o_P(1),\\
	\sqrt{Th_i}E(e_{it}H_{it}^{(i)}(x))&=&O_P(\sqrt{T/N})=o_P(1).
\end{eqnarray*}
Similar to \cite{SC13}, we have
\begin{equation}\label{eq:W:i:on:basis}
W_i\varphi_\nu^{(i)}=\frac{\eta_i \rho_\nu^{(i)}}{1+\eta_i \rho_\nu^{(i)}}\varphi_\nu^{(i)},\,\,\nu\ge1.
\end{equation}
Now we see from \ref{A2} and (\ref{eq:W:i:on:basis})
\begin{eqnarray*}
	V_i(G_{i,k},W_ig_{i0})=\sum_{\nu\geq 1}V_i(G_{i,k},\varphi_\nu^{(i)})V_i(g_{i0},\varphi_\nu^{(i)})\frac{\eta_i \rho_\nu^{(i)}}{1+\eta_i\rho_\nu^{(i)}}.
\end{eqnarray*}
So by Cauchy's inequality,
\begin{eqnarray*}
	|V_i(G_{i,k},W_ig_{i0})|^2&\leq& \sum_{\nu\geq 1}|V_i(G_{i,k},\varphi_\nu^{(i)})|^2 \frac{\eta_i  \rho_\nu^{(i)}}{1+\eta_i\rho_\nu^{(i)}}\sum_{\nu\geq 1}|V_i(g_{i0},\varphi_\nu^{(i)})|^2\frac{\eta_i \rho_\nu^{(i)}}{(1+\eta_i\rho_\nu^{(i)})}\\
	&\leq& \eta_i \textrm{ const }  \sum_{\nu\geq 1}|V_i(G_{i,k},\varphi_\nu^{(i)})|^2 \frac{\eta_i  \rho_\nu^{(i)}}{1+\eta_i\rho_\nu^{(i)}}\\
	&\leq& \eta_i \textrm{ const } \sum_{\nu\geq 1}|V_i(G_{i,k},\varphi_\nu^{(i)})|^2k_\nu \frac{\eta_i  \rho_\nu^{(i)}}{(1+\eta_i\rho_\nu^{(i)})k_\nu}\\
	&\leq& \eta_i \textrm{ const },
\end{eqnarray*}
For $\|A_{i,k}\|_{\sup}$, by definition,
\begin{eqnarray*}
	A_{i,k}(x)=\langle A_{i,k}, K_{x}^{(i)}\rangle_{\star,i}
	\leq V_i(G_{i,k},K_{x}^{(i)})
	=\sum_{\nu \geq 1}\frac{V_i(G_{i,k}, \varphi_\nu^{(i)})}{1+\eta_i \rho_\nu^{(i)}} \varphi_\nu^{(i)}(x).
\end{eqnarray*}
By boundedness condition of $\varphi_{\nu}^{(i)}$ (Assumption \ref{A2}) and Cauchy's inequality, we have
\begin{eqnarray*}
	|A_{i,k}(x)|^2&\leq& \sum_{\nu \geq 1}|V_i(G_{i,k}, \varphi_\nu^{(i)})|^2k_{\nu}|\varphi_\nu^{(i)}(x)|^2\sum_{\nu \geq 1}\frac{1}{k_\nu (1+\eta_i \rho_\nu^{(i)})^2}\\
	&\leq& \textrm{ const } \sum_{\nu \geq 1}|V_i(G_{i,k}, \varphi_\nu^{(i)})|^2k_{\nu}\sum_{\nu \geq 1}\frac{1}{k_\nu}=O(1).
\end{eqnarray*}
The above holds uniformly for all $x\in \mathcal{X}_i$.
So 
\begin{eqnarray*}
		\sqrt{T}(\Omega_i+\Sigma_i)^{-1}V_i(G_i,W_ig_{i0})=O(\sqrt{T\eta_i})=o(1),
\end{eqnarray*}
and
\begin{eqnarray*}
\sqrt{Th_i}A_i'(x)(\Omega_i+\Sigma_i)^{-1}V_i(G_i,W_ig_{i0})=O(\sqrt{Th_i\eta_i})=o(1).
\end{eqnarray*}
By the above uniform boundedness of $A_{i,k}$, we have $\beta_x=0$. 
Similarly, by (\ref{eq:W:i:on:basis}), we have
\begin{eqnarray*}
	W_iA_{i,k}(x)&=&\sum_{\nu \geq 1}\frac{V_i(G_{i,k}, \varphi_\nu^{(i)})}{1+\eta_i \rho_\nu^{(i)}} \eta_i\rho_\nu^{(i)}\varphi_\nu^{(i)}(x).
\end{eqnarray*}
Hence
\begin{eqnarray*}
	|W_iA_{i,k}(x)|^2&\leq& \sum_{\nu \geq 1}|V_i(G_{i,k}, \varphi_\nu^{(i)})|^2k_{\nu}|\varphi_\nu^{(i)}(x)|^2\sum_{\nu \geq 1}\frac{(\eta_i\rho_\nu^{(i)})^2}{k_\nu (1+\eta_i \rho_\nu^{(i)})^2}\\
	&\leq& \textrm{ const } \sum_{\nu \geq 1}|V_i(G_{i,k}, \varphi_\nu^{(i)})|^2k_{\nu}
	\sum_{\nu \geq 1}\frac{1}{k_\nu }=O_(1),
\end{eqnarray*}
and this rate is uniform for all $x \in \mathcal{X}_i$. So $\alpha_x=0$.
By definition of RKHS and $W_i$,
\begin{eqnarray*}
|W_ig_{i0}(x)|&=&|\langle W_ig_{i0},K_x^{(i)}\rangle_{\star,i}|\\
&=&|\eta_i\langle g_{i0},K_x^{(i)}\rangle_{\mathcal{H}_i}|\\
&\leq&  \|g_{i0}\|_{\mathcal{H}_i}\|K_x^{(i)}\|_{\mathcal{H}_i}\\
&\leq&  \sqrt{\eta_i}\|g_{i0}\|_{\mathcal{H}_i}\sqrt{\langle K_x^{(i)}, K_x^{(i)} \rangle_{\star,i}}\\
&\leq&  \sqrt{\eta_i}\|g_{i0}\|_{\mathcal{H}_i}O({h_i^{-1/2}})\\
&=&O(\sqrt{\eta_i/h_i}).
\end{eqnarray*}
Thus $\sqrt{Th_i}W_ig_{i0}(x)=o(1)$.
\end{proof}

\begin{proof}[Proof of Corollary \ref{proposition:estimation:sigma:i:epsilon}]
By (\ref{basic:model 2}), 
\begin{eqnarray*}
 	Y_{it}-\widehat{g}_i(X_{it})-Z_t'\widehat{\beta}_i&=&\epsilon_{it}-\Delta_i \bar{v}_t+(g_{i0}(X_{it})-\widehat{g}_i(X_{it}))+Z_t'(\beta_{i0}-\widehat{\beta}_i).
\end{eqnarray*}
Hence $\sum_{t=1}^T\{Y_{it}-\widehat{g}_i(X_{it})-Z_t'\widehat{\beta}_i-\epsilon_{it}\}^2/T \leq 4\sum_{t=1}^T(A_{1t}+A_{2t}+A_{3t})/T$, where
$A_{1t}=\|\Delta_i\|_2^2\|\bar{v}_t\|_2^2$,
$A_{2t}=(g_{i0}(X_{it})-\widehat{g}_i(X_{it}))^2$,
$A_{3t}=\|Z_t\|_2^2\|\widehat{\beta}_i-\beta_{i0}\|_2^2$.
By uniform boundedness of $\Delta_i$ and i.i.d. of $v_{it}$ in Assumption \ref{A1}, we have
$E(A_{t1})\le \|\Delta_i\|_2^2 tr(\Sigma_{v})/N$,
where $\Sigma_v$ is the covariance matrix of $v_{it}$. So $\sum_{t=1}^T A_{1t}=O_P(1/N)=o_P(1)$.
Also, by Lemma \ref{bound:RiUit} and Theorem \ref{them:convergence:rate:result:1},
$$\|g_{i0}-\widehat{g}_i\|_{\sup}=O_P(h_i^{-1/2}\|\theta_{i0}-\widehat{\theta}_i\|_i)=O_P(h_i^{-1/2}r_{i,M}).$$ 
So it follows that $\sum_{t=1}^T A_{2t}/T=O_P(h_i^{-1}r_{i,M}^2)=o_P(1)$.
By Proposition \ref{prop:maximum:term}, Lemma \ref{bound:RiUit} and Theorem \ref{them:convergence:rate:result:1} , we have $A_{3t}=O_P(T^{2/\alpha}r_{i,M}^2)$ uniformly for all $t=1,2,\ldots,T$. So $\sum_{t=1}^T A_{3t}/T=O_P(T^{2/\alpha}r_{i,M}^2)=o_P(1)$.
Hence $\sum_{t=1}^T\{Y_{it}-\widehat{g}_i(X_{it})-Z_t'\widehat{\beta}_i\}^2/T-\sum_{t=1}^T\epsilon_{it}^2/T=o_P(1)$. The result follows  by applying Law of Large Number:
$\sum_{t=1}^T\epsilon_{it}^2/T \to \sigma_\epsilon^2$, in probability.
\end{proof}

\subsection{Proofs in Section \ref{subsec:rate:homo}}\label{proof:subsec:rate:homo}
We prove Theorem \ref{theorem:homo convergence:rate:result:1}.
Let us first introduce some notation. Define
\begin{align*}
	F_1'=(f_{11},f_{12},...,f_{1T}), F_2'=(f_{21},f_{22},...,f_{2T}), \\
	\epsilon_i=(\epsilon_{i1}, \epsilon_{i2},..., \epsilon_{iT})', e_i=(e_{i1}, e_{i2},..., e_{iT}),\\
	\bar{X}=(\bar{X}_1, \bar{X}_2, ..., \bar{X}_T),\bar{v}=(\bar{v}_1, \bar{v}_2,..., \bar{v}_T), \bar{X}^{\star}=\bar{X}-\bar{v},\\
	P_\star=I_T-\Sigma_\star'(\Sigma_\star \Sigma_\star')^{-1}\Sigma_\star, \widetilde{\Sigma}=\Sigma-\Sigma_\star.
\end{align*}
We can rewrite (\ref{basic:model}) as $(Y_i-\ip{\kxi, g_0})'=\gamma_{1i}'F_1'+\gamma_{2,i}'F_2'+\epsilon_i'$, (\ref{DGP of X}) as $\bar{X}^{*}=\bar{X}-\bar{v}=\bar{\Gamma}_1'F_1'+\bar{\Gamma}_2'F_2'$ and (\ref{basic:model 2}) as $Y_i-\ip{\kxi, g_0}=\Sigma'\beta_i+e_i$.
Notice that we have $\Sigma=(F_1, \bar{X}')'$, $\Sigma_{\star}=(F_1, \bar{X}^{*\prime})'$. By definition, we have $\Sigma P=0, \Sigma_{\star} P_{\star}=0$. Hence $F_1'P=0, \bar{X}P=0, F_1'P_{\star}=0, \bar{X}^{\star}P_{\star}=0, F_2'P_{\star}=0$.
Define $S_{M,\eta}^{\star}(g)=E\{ S_{M,\eta}(g)|\mathcal{F}_1^T \}$. By the proof of Lemma \ref{lemma:id}, it can be easily shown that $DS_{M,\eta}^{\star}(g)=id$, for any $g \in \mathcal{H}$. 
We also have 
\begin{equation}\label{eq:bound:for:norm:K:sup:norm}
\| K_{x} \|^2 \leq C^2_\varphi h^{-1},\,\,\,\,
\sup_{x\in \mathcal{X}} \|g(x)\| \leq C_\varphi h^{-1/2}\|g\|.
\end{equation}

The proof of Theorem \ref{theorem:homo convergence:rate:result:1} also relies on the following
Lemmas \ref{lemma:P:minus:Pstar}, \ref{concentration:homo} and \ref{lemma:S:M:eta:g:eta}.
Proofs of these lemmas are provided in supplement document.
 
\begin{lemma}\label{lemma:P:minus:Pstar}
Suppose that Assumptions \ref{A1}, \ref{A5} and \ref{A7} hold.
Then the following holds:
\begin{equation}\label{P:and:Pstar}
E\left(\|P-P_\star\|_{\textrm{op}}\right)=O(N^{-1/2}),
\end{equation}
\begin{eqnarray}\label{F2:P:and:Pstar}
\max_{1\le i\le N}\|E\{\gamma_{2i}'F_2'(P-P_\star)K_{\mathbb{X}_i}|\mathcal{F}_1^T\}\|
=O_P\left(\sqrt{\frac{T}{Nh}}
+\frac{T}{N\sqrt{h}}\right),
\end{eqnarray}
\begin{equation}\label{F3:and:Pstar}
E(\|\widetilde{\Sigma}\widetilde{\Sigma}'\|_{\textrm{op}})=O(T/N).
\end{equation}
where $F_2=(f_{21},\ldots,f_{2T})'$.
\end{lemma}
\begin{lemma}\label{concentration:homo}
Suppose that Assumptions \ref{A1}, \ref{A5} and \ref{A7} hold.
Let $p=p(M)\ge1$ be an $\mathcal{F}_1^T$-measurable sequence indexed by $M$ and 
let $\psi(\mathbb{X},g): \bbR^T\times\mathcal{H}\to\bbR^T$ be a measurable function
satisfying $\psi(\mathbb{X},0)\equiv0$ and the following Lipschitz condition:
\[
\|\psi(\mathbb{X},g_1)-\psi(\mathbb{X},g_2)\|_2\le L\sqrt{h/T}\|g_1-g_2\|_{\sup},\,\,\,\,
\textrm{for any $g_1,g_2\in\mathcal{H}$,}
\] 
where $L>0$ is a constant.
Then with $M\rightarrow\infty$, 
\begin{eqnarray*}
\sup_{g\in\mathcal{G}(p)}\|Z_M(g)\|=
O_P\left(1+\sqrt{\log\log\left(NJ(p,1)\right)}(J(p,1)+N^{-1/2})\right),
\end{eqnarray*}
where 
\[
Z_M(g)=\frac{1}{\sqrt{N}}\sum_{i=1}^N[\psi(\mathbb{X}_i,g)'PK_{\mathbb{X}_i}-E\{\psi(\mathbb{X}_i,g)'P K_{\mathbb{X}_i}|\mathcal{F}_1^T\}].
\]
\end{lemma}

\begin{lemma}\label{lemma:S:M:eta:g:eta}
Under conditions in Theorem \ref{theorem:homo convergence:rate:result:1}, 
$$\|S_{M,\eta}(g_{\eta})\|=O_P(\frac{1}{\sqrt{NTh}}+\frac{1}{N\sqrt{h}})+o_P(\sqrt{\eta}).$$
\end{lemma}

\begin{proof}[Proof of Theorem \ref{theorem:homo convergence:rate:result:1}]
The proof consists of two parts.

\textit{Part one}: Define $T_1(g)=g-S^{\star}_{M,\eta}(g+g_0)$. So 
\[
T_{1}(g)=g-DS_{M,\eta}^\star(g_{0})g-S_{M,\eta}^\star(g_{0})
=-S_{M,\eta}^\star(g_{0}).
\]
So we have
\begin{align*}
\|S_{M,\eta}^\star(g_{0})\| &= \| E\{-\frac{1}{NT}\sum_{i=1}^N(Y_{i}-\langle K_{\mathbb{X}_i}
g_0 \rangle)'P K_{\mathbb{X}_i}+W_\eta g_0 | \mathcal{F}_1^T\} \|  \\
&= \| E\{-\frac{1}{NT}\sum_{i=1}^Ne_i'P K_{\mathbb{X}_i} - W_\eta g_0 | \mathcal{F}_1^T\} \|\\
&= \| E\{-\frac{1}{NT}\sum_{i=1}^N(\gamma_{1i}'F_1'+\gamma_{2,i}'F_2'+\epsilon_i')P K_{\mathbb{X}_i} - W_\eta g_0 | \mathcal{F}_1^T\} \|\\
&= \| E\{-\frac{1}{NT}\sum_{i=1}^N\gamma_{2,i}'F_2'P K_{\mathbb{X}_i} - W_\eta g_0 | \mathcal{F}_1^T\} \|,\\
&= \| E\{-\frac{1}{NT}\sum_{i=1}^N\gamma_{2,i}'F_2'(P-P_{\star}) K_{\mathbb{X}_i} - W_\eta g_0 | \mathcal{F}_1^T\}\|,
\end{align*}
where the second last equation is using independence of  $\epsilon_i$ and $\mathbb{X}_i, F_1, F_2$.
By directly calculations,
\begin{align*}
	\|W_{\eta}g_0\|=\sup_{\|g\|=1}|\ip{W_{\eta}g_0,g}|=\sup_{\|g\|=1}|\eta\ip{g_0,g}_{\mathcal{H}}|\leq \sup_{\|g\|=1}\sqrt{\eta}\|g_0\|_{\mathcal{H}} \sqrt{\eta}\|g\|_{\mathcal{H}}\leq \sqrt{\eta}\|g_0\|_{\mathcal{H}}.
\end{align*}
For the first term,  by Lemma \autoref{lemma:P:minus:Pstar},
\begin{eqnarray*}
\| E\{-\frac{1}{NT}\sum_{i=1}^N\gamma_{2,i}'F_2'(P-P_{\star}) K_{\mathbb{X}_i}|\f1t\}\|=O_P(\frac{1}{\sqrt{NTh}}+\frac{1}{N\sqrt{h}}).
\end{eqnarray*}
As a consequence, $\|S^{\star}_{M,\eta}(g_0)\|=O_P(1/\sqrt{NTh}+1/({N\sqrt{h}})+\sqrt{\eta})$.
Hence with probability approaching one, we have
$$\|T_1(g)\|\leq C(\frac{1}{\sqrt{NTh}}+\frac{1}{N\sqrt{h}}+\sqrt{\eta}) \equiv r,$$
for some constant $C>0$.
Notice that $T_1$ is also a contraction mapping from $\bar{B}(0,r)$ to $\bar{B}(0,r)$, so there exists a $\widetilde{g}_1 \in \bar{B}(0,r) \subset \mathcal{H}$ such that,
$T_1(\widetilde{g}_1)=\widetilde{g}_1$.
By Taylor expansion, 
\begin{align*}
	\widetilde{g}_1=T_1(\widetilde{g}_1)=\widetilde{g}_1-S^{\star}_{M,\eta}(\widetilde{g}_1+g_0),
\end{align*}
and hence it follows that
$S^{\star}_{M,\eta}(\widetilde{g}_1+g_0)=0$.
Let $g_{\eta}=\widetilde{g}_1+g_0$, we have with probability approaching one, 
\begin{align}
	\label{ghat:minus:g0:rate:geta:minus:g0}
	\|g_{\eta}-g_0\|\leq  C(\frac{1}{\sqrt{NTh}}+\frac{1}{N\sqrt{h}}+\sqrt{\eta}),\,\,\,\,
	S^{\star}_{M,\eta}(g_{\eta})=0.
\end{align}

\textit{Part two}: Define $T_2(g)=g-S_{M,\eta}(g_{\eta}+g)$, for any $g, \Delta g \in \mathcal{H}$. 
Since
\begin{eqnarray*}
	\Delta g&=& DS^{\star}_{M,\eta}(g)\Delta g\\
	&=&E(DS_{M,\eta}(g)\Delta g|\f1t)
	=\frac{1}{NT}\sum_{i=1}^N E\{\tau_i\Delta g'P\kxi+W_{\eta}\Delta g|\f1t\},
\end{eqnarray*}
we have that 
\begin{eqnarray*}
\|T_2(g_1)-T_2(g_2)\| &=& \|(g_1-g_2)-\left(S_{M,\eta}(g_{\eta}+g_1)-S_{M,\eta}(g_{\eta}+g_2)\right)\| \\
&=& \|DS^{\star}_{M,\eta}(g)(g_1-g_2)-\left(S_{M,\eta}(g_{\eta}+g_1)-S_{M,\eta}(g_{\eta}+g_2)\right)\| \\
& = &\|\frac{1}{NT}\sum_{i=1}^N\{\tau_i(g_1-g_2)'P\kxi+W_{\eta}(g_1-g_2)\\
&&-E\left( \tau_i(g_1-g_2)'P\kxi+W_{\eta}(g_1-g_2) |\f1t\right)\}\|\\
&=& \|\frac{1}{NT}\sum_{i=1}^N\{\tau_i(g_1-g_2)'P\kxi-E\left( \tau_i(g_1-g_2)'P\kxi |\f1t\right)\}\|\\
&\equiv&\|\kappa(g_1-g_2)\|,
\end{eqnarray*}
where $\kappa(g)=\sum_{i=1}^N(\tau_ig'P\kxi-E\{\tau_ig'P\kxi|\f1t\})/(NT)$.
Let $\Psi(\mathbb{X}_i,g)=\sqrt{h}\tau_ig/T$. It follows that, for any $g_1,g_2\in\mathcal{H}$,
\begin{align*}
	\|\Psi(\mathbb{X}_i,g_1)-\Psi(\mathbb{X}_i,g_1)\|_2&\leq \frac{\sqrt{h}}{T}\|\tau_i(g_1-g_2)\|_2\\
	&\leq \frac{\sqrt{h}}{T}\sqrt{\sum_{t=1}^T|g_1(X_{it})-g_2(X_{it})|^2}\leq \sqrt{\frac{h}{T}}\|g_1-g_2\|_{\sup}.
\end{align*}
Let $\widetilde{g} \equiv (g_1-g_2)/(c_\varphi h^{-1/2}\|g_1-g_2\|)$, then
$$\|\widetilde{g}\|_{\sup}=\frac{\|g_1-g_2\|_{\sup}}{c_\varphi h^{-1/2}\|g_1-g_2\|}\leq 1,$$
and
$$\|\widetilde{g}_{\eta}\|_{\mathcal{H}}=\frac{\|g_1-g_2\|_{\mathcal{H}}}{c_\varphi h^{-1/2}\|g_1-g_2\|}\leq \frac{1}{c_\varphi\sqrt{h^{-1}\eta}}.$$
So $\widetilde{g} \in \mathcal{G}(p)$, with $p=1/({c_\varphi\sqrt{h^{-1}\eta}})$.
By Lemma \ref{concentration:homo}, 
$$\frac{\|Z_M(g_1-g_2)\|}{c_\varphi h^{-1/2}\|g_1-g_2\|}=\|Z_M(\widetilde{g})\|=O_P\left(1+\sqrt{\log\log\left(NJ(p,1)\right)}(J(p,1)+N^{-1/2})\right)=O_P(b_{N,p}).$$
So by assumption $b_{N,p}=o_P(\sqrt{N}h)$, we have
\begin{align}
\|\kappa(g_1-g_2)\|&=\|\frac{1}{\sqrt{Nh}}Z_M(g_1-g_2)\|\nonumber\\
&=c_\varphi\frac{1}{\sqrt{N}h}\|g_1-g_2\|O_P\left(b_{N,p}\right)
\label{ghat:minus:g0:rate:kappa}
=O_P(\frac{b_{N,p}}{\sqrt{N}h})\|g_1-g_2\|=o_P(1)\|g_1-g_2\|,
\end{align}
where the terms $O_P, o_P$ do not depend on $g_1, g_2$.
Hence with probability approaching one, uniformly for any $g_1, g_2$,
$\|T_2(g_1)-T_2(g_2)\|\leq \frac{1}{2}\|g_1-g_2\|$.
Also with probability approaching one, uniformly for $g$,
$$\|T_2(g)\|\leq \|T_2(g)-T_2(0)\|+\|T_2(0)\|\leq \frac{1}{2}\|g\|+\|S_{M,\eta}(g_{\eta})\|.$$
By Lemma \ref{lemma:S:M:eta:g:eta}, with probability approaching one, 
$$\|S_{M,\eta}(g_{\eta})\|\leq C\left(\frac{1}{\sqrt{NTh}}+\frac{1}{N\sqrt{h}}+\sqrt{\eta}\right)\equiv R/2,$$
for some $C>0$.
Hence with probability approaching one, it follows that
$\sup_{\|g\|\leq R}\|T_2(g)\|\leq R/2+R/2=R$.
The above implies that $T_2$  is a contraction mapping from $\bar{B}(0,R)$ to itself.
By contraction mapping theorem, 
there exists $\widetilde{g}_2 \in \bar{B}(0,R)$ such that
$\widetilde{g}_2=T_2(\widetilde{g}_2)=\widetilde{g}_2-S_{M,\eta}(g_{\eta}+\widetilde{g}_2)$.
Hence $S_{M,\eta}(g_{\eta}+\widetilde{g}_2)=0$. So $\widehat{g}=g_{\eta}+\widetilde{g}_2$.
Therefore,
\[
\|\widehat{g}-g_0\| \leq \| g_{\eta}-g_0\|+\|\widetilde{g}_2\|
=O_P\left(\frac{1}{\sqrt{NTh}}+\frac{1}{N\sqrt{h}}+\sqrt{\eta}\right).
\]
\end{proof}

\subsection{Proofs in Section \ref{subsec:normality:homo}}\label{proof:section:normality:homo}
We will prove Theorems \ref{theorem: rate:ghat:minus:g0:plus:SMetag0} and \ref{theorem:normality:ghat:minus:g0},
and Proposition \ref{prop:homo sigma}.
Let us introduce some additional notation and preliminaries.
Let $m$ be the increasing sequence of integers provided in Theorems
\ref{theorem: rate:ghat:minus:g0:plus:SMetag0} and \ref{theorem:normality:ghat:minus:g0}. 
For any fixed $x_0 \in \mathcal{X}$, define
$V_{NT}=\frac{1}{NT}\sum_{i=1}^N\kxi(x_0)'P\kxi(x_0), A_{NT}=V_{NT}^{-1/2}$.
Let $V_{NTm}=\sum_{i=1}^N\phi_m^{\prime}\Phi_i'P\Phi_i\phi_m/(NT)$, $A_{NTm}=V_{NTm}^{-1/2}$
and $H_{NTm}=\sum_{i=1}^N\Phi_i' P \Phi_i/(NT)$,
where
\begin{eqnarray*}
\Phi_i=	\left(\begin{matrix}
\varphi_1(X_{i1}) &\varphi_2(X_{i1})&\cdots&\varphi_m(X_{i1})\cr
\varphi_1(X_{i2}) &\varphi_2(X_{i2})&\cdots&\varphi_m(X_{i2})\cr
&&\cdots&\cr
\varphi_1(X_{iT}) &\varphi_2(X_{iT})&\cdots&\varphi_m(X_{iT})\cr
\end{matrix}\right),\,\,\,\,
\phi_m=\left(\frac{\varphi_1(x_0)}{1+\eta \rho_1}, \frac{\varphi_2(x_0)}{1+\eta \rho_2},\cdots, 
\frac{\varphi_m(x_0)}{1+\eta \rho_m}\right)'.
\end{eqnarray*}

The proof of Theorems \ref{theorem: rate:ghat:minus:g0:plus:SMetag0} and
\ref{theorem:normality:ghat:minus:g0} rely on the following Lemmas \ref{lemma:ANTm:2:Im},
\ref{lemma:bound:ANT} and \ref{lemma:normality:epsilon:P:Kxi}.
Proofs of these lemmas can be found in supplement document.

\begin{lemma}\label{lemma:ANTm:2:Im}
Under Assumptions \ref{A1}, \ref{A5}  and \ref{A7}, suppose $m=o(\sqrt{N})$, then
$\|H_{NTm}-I_m\|_F=o_P(1)$, $\lambda_{\min}^{-1}(H_{NTm})=O_P(1)$ and $\lambda_{\max}(H_{NTm})=O_P(1)$.
\end{lemma}
\begin{lemma}\label{lemma:bound:ANT}
Under Assumptions \ref{A1}, \ref{A5} and \ref{A7}, suppose $h^{-1}=o_P(\sqrt{N}), D_m=o_P(1)$ , then for any $x_0 \in \mathcal{X}$,
$A_{NT}=O_P(1)$.
\end{lemma}
\begin{lemma}\label{lemma:normality:epsilon:P:Kxi}
Under Assumptions \ref{A1}, \ref{A5} and \ref{A7}, suppose $D_m=o_P(\sqrt{N}), m=o_P(\sqrt{N})$, for ant $x_0 \in \mathcal{X}$, we have
$$\sqrt{NT}A_{NTm}(\frac{1}{NT}\sum_{i=1}^N\phi_m' \Phi_i' P \epsilon_i)\overset{d}{\to} N(0,\sigma_{\epsilon}^2).$$
\end{lemma}

\begin{proof}[Proof of Theorem \ref{theorem: rate:ghat:minus:g0:plus:SMetag0}]
	Let $g=\widehat{g}-g_0$. By the fact $DS_{M,\eta}^{\star}(g_0)=id$ (see Section \ref{proof:subsec:rate:homo}) 
	and $S_{M,\eta}(g+g_0)=0$, we have
	\begin{eqnarray*}
	\|g+S_{M,\eta}(g_0)\|&=&\|DS_{M,\eta}^{\star}(g_0)g+S_{M,\eta}(g_0)\|\\
	&=&\|S_{M,\eta}^{\star}(g+g_0)-S_{M,\eta}^{\star}(g_0)-S_{M,\eta}(g+g_0)+S_{M,\eta}(g_0)\|\\
	&=&\|\frac{1}{NT}\sum_{i=1}^N(\tau_ig'P\kxi-E(\tau_ig'P\kxi)) +W_\eta g -E(W_\eta g|\f1t)\|\\
	&\leq& \|\kappa(g)\|,
	\end{eqnarray*}
	where $\kappa(g)$ is defined in Part two of the proof of Theorem \ref{theorem:homo convergence:rate:result:1}.
	By (\ref{ghat:minus:g0:rate:kappa}) with $g_1-g_2$ therein replaced by $g$ we have $\|\kappa(g)\|=O_P(\frac{b_{N,p}}{\sqrt{N}h})\|g\|.$ Hence 
	\[
	\|g+S_{M,\eta}(g_0)\|=O_P\left(\frac{b_{N,p}}{\sqrt{N}h}\right)\|g\|=O_P\left(\frac{b_{N,p}}{\sqrt{N}h}
	\left(\frac{1}{\sqrt{NTh}}+\frac{1}{N\sqrt{h}}+\sqrt{\eta}\right)\right).
	\]
For fixed $x_0\in \mathcal{X}$,
	\begin{eqnarray*}
	A_{NT}|g(x_0)+S_{M,\eta}(g_0)(x_0)|&\leq&A_{NT}\|g+S_{M,\eta}(g_0)\|_{\sup}\\
	&\leq& c_\varphi h^{-1/2}\|g+S_{M,\eta}(g_0)\|\\
	&=&O_P\left(\frac{b_{N,p}}{\sqrt{N}h}\left(\frac{1}{\sqrt{NT}h}+\frac{1}{Nh}+\sqrt{\frac{\eta}{h}}\right)\right).
 	\end{eqnarray*}		
	Since $g_0 \in \mathcal{H}$ is fixed,
	\begin{eqnarray}
	\label{eq:opnorm:Weta}
		\|W_\eta g_0\|=\sup_{\|\widetilde{g}\|=1}\ip{W_\eta g_0, \widetilde{g}}=\sup_{\|\widetilde{g}\|=1} \eta\ip{g_0, \widetilde{g}}_{\mathcal{H}}\leq \sup_{\|\widetilde{g}\|=1}\eta \|g_0\|_{\mathcal{H}}\|\widetilde{g}\|_{\mathcal{H}}\leq \sqrt{\eta}\|g_0\|_{\mathcal{H}}.
	\end{eqnarray}
It follows from (\ref{eq:opnorm:Weta}) that $\|W_\eta g_0\|\leq \sqrt{\eta}\|g_0\|_{\mathcal{H}}=O_P(\sqrt{\eta})$, and 
hence 
\[
|W_\eta g_0(x_0)|\le c_\varphi h^{-1/2}\|W_\eta g_0\|=O_P(\sqrt{\eta/h}).
\] Hence, we have
	$$A_{NT}|g(x_0)+S_{M,\eta}(g_0)(x_0)-W_\eta g_0(x_0)|=O_P\left(\frac{b_{N,p}}{\sqrt{N}h}
	\left(\frac{1}{\sqrt{NT}h}+\frac{1}{Nh}+\sqrt{\frac{\eta}{h}}\right)\right).$$
\end{proof}

\begin{proof}[Proof of Theorem \ref{theorem:normality:ghat:minus:g0}]
	By Theorem \ref{theorem: rate:ghat:minus:g0:plus:SMetag0}, the asymptotic distribution of $\sqrt{NT}A_{NT}(\widehat{g}(x_0)-g_0(x_0)+W_{\eta}g_0)$ is the same as $\sqrt{NT}A_{NT}[-S_{M,\eta}(g_0)(x_0)+W_{\eta}g_0(x_0)]=\sqrt{NT}A_{NT}[\frac{1}{NT}\sum_{i=1}^N (Y_i-\tau_ig_0)'P\kxi(x_0)]$.
	
	By $Y_i=\tau_ig_0+\Sigma'\beta_i+e_i$ (\ref{basic:model:2:homogeneity}), 
	$\Sigma P=0$
	(see Section \ref{proof:subsec:rate:homo}) and $e_i=\epsilon_i-\bar{v}'\Delta_i'$, it yields that
	\begin{eqnarray*}
&&	A_{NT}[\frac{1}{NT}\sum_{i=1}^N(Y_i-\tau_ig_0)'P\kxi(x_0)]\\
	&=&A_{NT}(\frac{1}{NT}\sum_{i=1}^N\kxi'(x_0) P \epsilon_i)-A_{NT}(\frac{1}{NT}\sum_{i=1}^N\kxi'(x_0) P \bar{v}'\Delta_i')\\
	&\equiv& A_{NT}\zeta-A_{NT}\xi.
\end{eqnarray*}
	Let $\zeta_m=\sum_{i=1}^N \phi_m'\Phi_i' P \epsilon_i/(NT)$ and $\xi_m=\sum_{i=1}^N\phi_m'\Phi_i' P \bar{v}'\Delta_i'/(NT)$.
To prove the result of the theorem, it is sufficient to prove the following:
	\begin{eqnarray}
		\label{eq:normality:ghat:minus:g0:claim1}
		A_{NT}-A_{NTm}&=&o_P(1),\\
		\label{eq:normality:ghat:minus:g0:claim2}		
		A_{NT}&=&O_P(1),\\
		\label{eq:normality:ghat:minus:g0:claim3}
		\sqrt{NT} A_{NTm}\xi_m&=&o_P(1),\\
		\label{eq:normality:ghat:minus:g0:claim4}
		\sqrt{NT}(\xi-\xi_m)&=&o_P(1),\\
		\label{eq:normality:ghat:minus:g0:claim5}
		\sqrt{NT}(\zeta-\zeta_m)&=&o_P(1),\\
		\label{eq:normality:ghat:minus:g0:claim6}
		\sqrt{NT}A_{NTm}\zeta_m &\overset{d}{\to}& N(0,1).
	\end{eqnarray}
(\ref{eq:normality:ghat:minus:g0:claim1}) and (\ref{eq:normality:ghat:minus:g0:claim2}) are guaranteed by Lemma \ref{lemma:bound:ANT}; (\ref{eq:normality:ghat:minus:g0:claim6}) follows from Lemma \ref{lemma:normality:epsilon:P:Kxi}.	
For (\ref{eq:normality:ghat:minus:g0:claim3}), by the expression of $A_{NTm}$ and Lemma \ref{lemma:ANTm:2:Im}, we get,
\begin{eqnarray*}
	|A_{NTm}\xi_m|&=&|A_{NTm}\phi_m' \frac{1}{NT}\sum_{i=1}^N \Phi_i'P \bar{v}'\Delta_i'|\\
	&\leq&  \|A_{NTm}\phi_m \|_2\times\|\frac{1}{NT}\sum_{i=1}^N \Phi_i'P\bar{v}'\Delta_i'\|_2 \\
	&\leq& \sqrt{\frac{\|\phi_m\|_2^2}{\phi_m' H_{NTm} \phi_m}} \|\frac{1}{NT}\sum_{i=1}^N \Phi_i'P\bar{v}'\Delta_i'\|_2\\
	&\leq& \lambda_{\min}^{-1/2}(H_{NTm}) \|\frac{1}{NT}\sum_{i=1}^N \Phi_i'P\bar{v}'\Delta_i'\|_2.
\end{eqnarray*}
Recall $\Delta_i=\gamma_{2i}'(\bar{\Gamma}_2\bar{\Gamma}_2')^{-1}\bar{\Gamma}_2\equiv \gamma_{2i}'\widetilde{M}$
(see Assumption \ref{A1}\ref{A1:d}), which leads to 
$$\|\sum_{i=1}^N \Phi_i'P\bar{v}'\Delta_i'\|_2^2=Tr\{(\sum_{i=1}^N \Phi_i'P\bar{v}'\Delta_i')(\sum_{i=1}^N \Delta_i \bar{v}P \Phi_i)\}.$$ In matrix form, the above becomes
\begin{eqnarray*}
	&&Tr\{(\Phi_1', \Phi_2',...,\Phi_N'){P}_N {\bar{v}_N}'{\widetilde{M}_N}'{\gamma_2}{\gamma_2}'{\widetilde{M}_N}{\bar{v}_N}{P}_N (\Phi_1', \Phi_2',...,\Phi_N')'\}\\
	&\leq&\lambda_{\max}({\gamma_2}{\gamma_2}')\lambda_{\max}({\widetilde{M}_N}'{\widetilde{M}_N})\lambda_{\max}({\bar{v}_N}'{\bar{v}_N}) Tr\{\sum_{i=1}^N \Phi_i' P \Phi_i\},
\end{eqnarray*}
where ${P_N}$ is defined in Section \ref{sec:est:homo}, 
${\widetilde{M}_N}=I_N\otimes \widetilde{M}, {\bar{v}}_N=I_N\otimes \bar{v}$.

By Lemma \ref{lemma:P:minus:Pstar}, Assumption \ref{A1} and Lemma \ref{lemma:ANTm:2:Im} ,we have $$\lambda_{\max}({\bar{v}_N}'{\bar{v}_N})=\|\widetilde{\Sigma}\widetilde{\Sigma}'\|_{\textrm{op}}=O_P(T/N),
\,\,\,\, \lambda_{\max}({\widetilde{M}}_N'{\widetilde{M}}_N)=\lambda_{\max}((\bar{\Gamma}_2\bar{\Gamma}_2')^{-1})=O_P(1),$$ and $$Tr\{\sum_{i=1}^N \Phi_i' P \Phi_i\}=NTTr(H_{NTm})=O_P(NTm).$$ So it can be seen that,
$$|A_{NTm}\xi_m|=O_P\left(\frac{\sqrt{\lambda_{\max}({\gamma_2}{\gamma_2}')m}}{N}\right),$$
thus
$$\sqrt{NT}A_{NTm}\xi_m=O_P\left(\sqrt{\frac{\lambda_{\max}({\gamma_2}{\gamma_2}')mT}{N}}\right)=o_P(1),$$
and hence (\ref{eq:normality:ghat:minus:g0:claim3}) is true.

Similar to (\ref{eq:lemma:S:M:eta:g:eta:Rate:T2}) in
the proof of Lemma \ref{lemma:S:M:eta:g:eta} (see supplement document), 
it can be shown that
\begin{equation}\label{a:very:complex:eqn}
|\xi-\xi_m|=O_P\left(\frac{D_m}{\sqrt{NTh}}+\frac{D_m}{N\sqrt{h}}\right).
\end{equation} 
More explicitly, the proof of (\ref{a:very:complex:eqn}) follows by replacing 
$K_{\mathbb{X}_i}$ in the expression of $T_2$ with
$\Phi_i\phi_m$, and by a line-by-line check.
Therefore, we have
$$\sqrt{NT}|\xi-\xi_m|=O_P\left(\frac{D_m}{\sqrt{h}}+\frac{D_m\sqrt{T}}{\sqrt{Nh}}\right)=o_P(1),$$
i.e., (\ref{eq:normality:ghat:minus:g0:claim4}) holds.

Let $\d1t=\sigma(f_{1t},f_{2t},X_{it}: t\in[T], i\in[N])$.
By independence of $\epsilon_i$ and $\d1t$, 
we have
$E(|\zeta-\zeta_m|^2|\d1t)=\sum_{i=1}^NR_i' P R_i/(NT)$,
where $R_i=(\kxi(x_0)-\Psi_i \phi_m)'P\epsilon_i$. 
Let $R_{x_0}(\cdot)=\sum_{\nu \geq m+1}\frac{\varphi_\nu(x_0)\varphi_\nu(\cdot)}{1+\eta \rho_\nu}$, since $\f1t \subset \d1t$, it follows that,
\begin{eqnarray*}
	E(|\zeta-\zeta_m|^2|\f1t)&=&E(\frac{1}{N^2T^2}\sum_{i=1}^NR_i' P R_i|\f1t)\\
	&=&\frac{1}{NT} V(R_{x_0}, R_{x_0})\\
	&=&\frac{1}{NT} \sum_{\nu=m+1}^\infty \frac{\varphi_\nu^2(x_0)}{(1+\eta\rho_\nu)^2}\\
	&\leq& \frac{1}{NT} c_\varphi^2D_m\\
	&=&O_P(\frac{D_m}{NT}),
\end{eqnarray*}
so $\zeta-\zeta_m=O_P(\sqrt{D_m/(NT)})$ which implies that (\ref{eq:normality:ghat:minus:g0:claim5}) is valid.
Proof completed.
\end{proof}

\begin{proof}[Proof of Proposition \ref{prop:homo sigma}]
	By $Y_i=\tau_ig_0+\Sigma'\beta_i+e_i, \Sigma P=0, e_i=\epsilon_i-\bar{v}'\Delta_i'$, we have
	\begin{eqnarray*}
		&&\frac{1}{NT}\sum_{i=1}^N(Y_i-\tau_i\widehat{g})'P(Y_i-\tau_i\widehat{g})\\
		&=&\frac{1}{NT}\sum_{i=1}^N (\tau_i(g_0-\widehat{g}))'P(\tau_i(g_0-\widehat{g}))+\frac{2}{NT}\sum_{i=1}^N (\tau_i(g_0-\widehat{g}))'Pe_i+\frac{1}{NT}\sum_{i=1}^N e_i'Pe_i\\
		&\equiv& T_1+T_2+T_3.
	\end{eqnarray*}
By Theorem \ref{theorem:homo convergence:rate:result:1}, $$\|\widehat{g}-g_0\|_{\sup}\leq c_\varphi h^{-1/2}\|\widehat{g}-g_0\|=O_P\left(\frac{1}{\sqrt{NT}h}+\frac{1}{Nh}+\sqrt{\frac{\eta}{h}}\right)=o_P(1).$$
So $|T_1|\leq \sum_{i=1}^N\sum_{t=1}^T |g_0(X_{it}-\widehat{g}(X_{it}))|^2/(NT) \leq \|\widehat{g}-g_0\|_{\sup}^2=o_P(1)$.

By the definitions of $\bar{v}$ and $\widetilde{\Sigma}$ and by Lemma \ref{lemma:P:minus:Pstar}, we have 
$E(\|\bar{v}\bar{v}'\|_{\textrm{op}})=E(\|\widetilde{\Sigma}\widetilde{\Sigma}'\|_{\textrm{op}})=O(T/N))$.
Hence it holds that
\begin{eqnarray*}
	E(\|e_i\|_2^2)&=&E(\epsilon_i'\epsilon_i)+E(\Delta_i'\bar{v}\bar{v}'\Delta_i)-2E(\epsilon_i'\bar{v}'\Delta_i)\\
	&\leq& T\sigma_{\epsilon}^2+E(\|\bar{v}\bar{v}'\|_{\textrm{op}})\sup_{1\leq i \leq N}\|\Delta_i\|_2^2=O(T).
\end{eqnarray*}
It then follows from Cauchy inequality that
\begin{eqnarray*}
 |T_2|&\leq& \frac{2}{NT}\sum_{i=1}^N \|\tau_i(g_0-\widehat{g})\|_2\|e_i\|_2\\
 &\leq& \frac{2}{NT}\sum_{i=1}^N\|e_i\|_2 \sqrt{T \|\widehat{g}-g_0\|_{\sup}}\\
 &=&O_P(1)\sqrt{\|\widehat{g}-g_0\|_{\sup}}=o_P(1).
\end{eqnarray*}
Meanwhile, the following decomposition holds
\begin{eqnarray*}
	T_3&=&\frac{1}{NT}\sum_{i=1}^N\epsilon_i'P\epsilon_i+\frac{1}{NT}\sum_{i=1}^N\Delta_i'\bar{v}P\bar{v}'\Delta_i-\frac{2}{NT}\sum_{i=1}^N\Delta_i'\bar{v}P\epsilon_i\\
	&=&\frac{1}{NT}\sum_{i=1}^N\epsilon_i'P_{\star}\epsilon_i+\frac{1}{NT}\sum_{i=1}^N\epsilon_i'(P-P_{\star})\epsilon_i+\frac{1}{NT}\sum_{i=1}^N\Delta_i'\bar{v}P\bar{v}'\Delta_i-\frac{2}{NT}\sum_{i=1}^N\Delta_i'\bar{v}P\epsilon_i\\
	&\equiv&T_{31}+T_{32}+T_{33}-T_{34}.
\end{eqnarray*}
We handle the above terms $T_{31},T_{32},T_{33},T_{34}$ respectively.
By Lemma \ref{lemma:P:minus:Pstar}, it follows that
\begin{eqnarray*}
	|T_{32}|\leq\frac{1}{NT}\|P-P_{\star}\|_{\textrm{op}}\sum_{i=1}^N\epsilon_i'\epsilon_i=O_P(N^{-1/2})=o_P(1).
\end{eqnarray*}
In the meantime,
\begin{eqnarray*}
	|T_{33}|\leq \frac{1}{NT}\sup_{1\leq i \leq N}\|\Delta_i\|_2^2\sum_{i=1}^N\|\bar{v}\bar{v}'\|_{\textrm{op}}=O_P(\frac{1}{N})=o_P(1),
\end{eqnarray*}
and
\begin{eqnarray*}
	|T_{34}|\leq \frac{2}{NT}\|\epsilon_i\|_2\|\bar{v}'\Delta_i\|\leq 2\sqrt{\frac{\sum_{i=1}^N \epsilon_i'\epsilon}{NT}}\sqrt{\frac{\sum_{i=1}^N \Delta_i'\bar{v}\bar{v}'\Delta_i}{NT}}=O_P(\frac{1}{\sqrt{N}})=o_P(1).
\end{eqnarray*}

Next we look at $T_{31}$. By direct examinations,
\begin{eqnarray*}
	E(T_{31}|\f1t)=\frac{1}{T}Tr\{P_{\star}E(\epsilon_i \epsilon_i' |\f1t)\}=\sigma_{\epsilon}^2\frac{1}{T}Tr(P_{\star})=\sigma_{\epsilon}^2\frac{T-(q_1+d)}{T},
\end{eqnarray*}
and by Chebyshev's inequality, for any $\epsilon>0$,
 \begin{eqnarray*}
 	P(|T_{31}-E(T_{31}|\f1t)|>\epsilon |\f1t)&\leq& \frac{1}{N}\frac{E\{|\epsilon_1'P_{\star}\epsilon_1/T-E(T_{31}|\f1t)|^2\}|\f1t}{\epsilon^2}\\
 	&\leq&\frac{1}{N\epsilon^2}E(|\epsilon_1'P_{\star}\epsilon_1/T|^2|\f1t)\\
 	&\leq& \frac{1}{NT^2\epsilon^2}E(|\epsilon_1'\epsilon_1|^2|\f1t)\\
 	&\leq &\frac{1}{NT^2\epsilon^2} (TE(\epsilon_{11}^4)+T^2\sigma_{\epsilon}^4)\\
 	&=&o(1).
 \end{eqnarray*}
So $T_{31} \to \sigma_{\epsilon}^2(T-q_1-d)/{T}$ in probability.
Since $T_{32},T_{33},T_{34}$ are all $o_P(1)$ as shown in the above, 
we have $T_3 \to  \sigma_{\epsilon}^2(T-q_1-d)/{T}$ in probability.
Proof completed.
\end{proof}
\bibliographystyle{apalike}
\bibliography{reference}

\begin{thebibliography}{}
\small{
\bibitem[Bai, 2009]{bai2009}
Bai, J. (2009).
\newblock Panel data models with interactive fixed effects.
\newblock {\em Econometrica}, 77(4):1229--1279.

\bibitem[Bai and Ng, 2006]{baiNg2006}
Bai, J. and Ng, S. (2006).
\newblock Confidence intervals for diffusion index forecasts and inference for
  factor-augmented regressions.
\newblock {\em Econometrica}, 74(4):1133--1150.

\bibitem[Berlinet and Thomas-Agnan, 2004]{BT04}
Berlinet, A. and Thomas-Agnan, C. (2004).
\newblock {\em Reproducing Kernel Hilbert Spaces in Probability and
  Statistics}.
\newblock Kluwer Academic, Boston, MA.

\bibitem[Bradley, 2005]{bradley2005}
Bradley, R.~C. (2005).
\newblock Basic properties of strong mixing conditions. a survey and some open
  questions.
\newblock {\em Probability Surveys}, 2(2):107--144.

\bibitem[Braun, 2006]{braun06}
Braun, M. (2006).
\newblock Accurate error bounds for the eigenvalues of the kernel matrix.
\newblock {\em Journal of Machine Learning Research}, 7:2303--2328.

\bibitem[Cai et~al., 2016]{caietal2016}
Cai, Z., Fang, Y., and Xu, q. (2016).
\newblock Inferences for varying-coefficient panel data models with
  cross-sectional dependence.
\newblock {\em Working Paper}.

\bibitem[Carneiro et~al., 2003]{carneiroetal2003}
Carneiro, P., Hansen, K., and Heckman, J. (2003).
\newblock Estimating distributions of treatment effects with an application to
  the returns to schooling and measurement of the effects of uncertainty on
  college choice.
\newblock {\em International Economic Review}, 44(2):361--422.

\bibitem[Chen, 2007]{chen2007}
Chen, X. (2007).
\newblock Large sample sieve estimation of semi-nonparametric models.
\newblock {\em Handbook of Econometrics}, 6:5549--5632.

\bibitem[Cheng and Shang, 2015]{CS15}
Cheng, G. and Shang, Z. (2015).
\newblock Joint asymptotics for semi-nonparametric regression models with
  partially linear structure.
\newblock {\em The Annals of Statistics}, 43(3):1351--1390.

\bibitem[Craven and Wahba, 1978]{cravenWahba1978}
Craven, P. and Wahba, G. (1978).
\newblock Smoothing noisy data with spline functions.
\newblock {\em Numerische Mathematik}, 31(4):377--403.

\bibitem[Cucker and Smale, 2001]{cuckerSmale2001}
Cucker, F. and Smale, S. (2001).
\newblock On the mathematical foundations of learning.
\newblock {\em American Mathematical Society}, 39(1):1--49.

\bibitem[Cunha et~al., 2005]{cunhaetal2005}
Cunha, F., Heckman, J., and Navarro, S. (2005).
\newblock Separating uncertainty from heterogeneity in life cycle earnings.
\newblock {\em Oxford Economic papers}, 57(2):191--261.

\bibitem[Dehling, 1983]{dehling1983}
Dehling, H. (1983).
\newblock Limit theorems for sums of weakly dependent banach space valued
  random variables.
\newblock {\em Probability Fields and Related Fields}, 63(3):393--432.

\bibitem[Dudley et~al., 1992]{dudleyetal1992}
Dudley, R.~M., Hahn, M.~G., and Kuelbs, J. (1992).
\newblock {\em Probability in Banach space}.
\newblock Springer Science Business Media.

\bibitem[Eberhardt et~al., 2013]{eberhardtetal2013}
Eberhardt, M., Helmers, C., and Strauss, H. (2013).
\newblock Do spillovers matter when estimating private returns to {R\&D}?
\newblock {\em Review of Economics and Statistics}, 95(2):436--448.

\bibitem[Fan and Yao, 2003]{fanYao2003}
Fan, J. and Yao, Q. (2003).
\newblock {\em Nonlinear Time Series: Nonparametric and Parametric Methods}.
\newblock Springer Science Business Media, LLC.

\bibitem[Freyberger, 2012]{freberger2012}
Freyberger, J. (2012).
\newblock Nonparametric panel data models with interactive fixed effects.
\newblock {\em Working Paper}.

\bibitem[Gu, 2013]{Gu11}
Gu, C. (2013).
\newblock {\em Smoothing spline ANOVA models}, vol 297.
\newblock Springer Science Business Media, LLC.

\bibitem[Gu and Qiu, 1993]{GC93}
Gu, C. and Qiu, C. (1993).
\newblock Smoothing spline density estimation: Theory.
\newblock {\em The Annals of Statistics}, pages 217--234.

\bibitem[Hashiguchi, 2015]{hashiguchi2015}
Hashiguchi, Y. (2015).
\newblock Allocation efficiency in china: an extension of the dynamic
  olley-pakes productivity decomposition.
\newblock {\em Working Paper}.

\bibitem[Hofmann et~al., 2008]{hofmann2008}
Hofmann, T., Schölkopf, B., and Smola, A.~J. (2008).
\newblock Kernel methods in machine learning.
\newblock {\em The Annals of Statistics}, pages 1171--1220.

\bibitem[Holly et~al., 2010]{hollyetal2010}
Holly, S., Pesaran, M.~H., and Yamagata, T. (2010).
\newblock A spatio-temporal model of house prices in the usa.
\newblock {\em Journal of Econometrics}, 158(1):160--173.

\bibitem[Huang, 2013]{huang2013}
Huang, X. (2013).
\newblock Nonparametric estimation in large panels with cross-sectional
  dependence.
\newblock {\em Econometric Reviews}, 32(5-6):754--777.

\bibitem[Jin and Su, 2013]{jinandsu2013}
Jin, S. and Su, L. (2013).
\newblock A nonparametric poolability test for panel data models with cross
  section dependence.
\newblock {\em Econometric Reviews}, 32(4):469--512.

\bibitem[Kosorok, 2008]{kosorok2008}
Kosorok, M. (2008).
\newblock {\em Introduction to empirical processes and semiparametric
  inference}.
\newblock Springer: New York.

\bibitem[Lu et~al., 2016]{LCL17}
Lu, J., Cheng, G., and Liu, H. (2016).
\newblock Nonparametric heterogeneity testing for massive data.
\newblock {\em arXiv preprint {arXiv}:1601}.

\bibitem[Ludvigson and Ng, 2009]{ludvigsonNg2009}
Ludvigson, S.~C. and Ng, S. (2009).
\newblock Macro factors in bond risk premia.
\newblock {\em Review of Financial Studies}, 22(12):5027--5067.

\bibitem[Ludvigson and Ng, 2016]{ludvigsonNg2016}
Ludvigson, S.~C. and Ng, S. (2016).
\newblock A factor analysis of bond risk premia.
\newblock {\em Handbook of Empirical Economics and Finance}, pages 313--372.

\bibitem[Mammen and van~de Geer, 1997]{MVG97}
Mammen, E. and van~de Geer, S. (1997).
\newblock Penalized quasi-likelihood estimation in partial linear models.
\newblock {\em The Annals of Statistics}, pages 1014--1035.

\bibitem[Melitz, 2003]{melitz2003}
Melitz, M.~J. (2003).
\newblock The impact of trade on intra-industry reallocations and aggregate
  industry productivity.
\newblock {\em Econometrica}, 71(6):1695--1725.

\bibitem[Moon and Weidner, 2010]{moonWeidner2010}
Moon, H.~R. and Weidner, M. (2010).
\newblock Dynamic linear panel regression models with interactive fixed
  effects.
\newblock {\em Econometric Theory}, pages 1--38.

\bibitem[Moon and Weidner, 2015]{moonWeidner2015}
Moon, H.~R. and Weidner, M. (2015).
\newblock Linear regression for panel with unknown number of factors as
  interactive fixed effects.
\newblock {\em Econometrica}, 83(4):1543--1579.

\bibitem[O'Brien, 1974]{OBrien1974}
O'Brien, G. (1974).
\newblock The maximum term of uniformly mixing stationary processes.
\newblock {\em Probability Theory and Related Fields}, 30(1):57--63.

\bibitem[Pesaran, 2006]{pesaran2006}
Pesaran, M.~H. (2006).
\newblock Estimation and inference in large heterogeneous panels with a
  multifactor error structure.
\newblock {\em Econometrica}, 74(4):967--1012.

\bibitem[Pinelis, 1994]{pinelis1994}
Pinelis, I. (1994).
\newblock Optimum bounds for the distributions of martingales in banach spaces.
\newblock {\em The Annals of Probability}, pages 1679--1706.

\bibitem[Ramsay and Silverman, 2005]{RamseySilverman2005}
Ramsay, J.~O. and Silverman, B.~W. (2005).
\newblock {\em Functional Data Analysis (2nd ED)}.
\newblock Springer: New York.

\bibitem[Shang, 2010]{S10}
Shang, Z. (2010).
\newblock Convergence rate and {Bahadur} type representation of general
  smoothing spline {M}-estimates.
\newblock {\em Electronic Journal of Statistics}, 4:1411--1442.

\bibitem[Shang and Cheng, 2013]{SC13}
Shang, Z. and Cheng, G. (2013).
\newblock Local and global asymptotic inference in smoothing spline models.
\newblock {\em The Annals of Statistics}, 41(5):2608--2638.

\bibitem[Shang and Cheng, 2015]{SC15}
Shang, Z. and Cheng, G. (2015).
\newblock Nonparametric inference in generalized functional linear models.
\newblock {\em The Annals of Statistics}, 43(4):1742--1773.

\bibitem[Shao, 2003]{shao2003}
Shao, J. (2003).
\newblock {\em Mathematical Statistics}.
\newblock Springer: New York.

\bibitem[Su and Jin, 2012]{suJin2012}
Su, L. and Jin, S. (2012).
\newblock Sieve estimation of panel data models with cross section dependence.
\newblock {\em Journal of Econometrics}, 169(1):34--47.

\bibitem[Su et~al., 2015]{suJinzhang2015}
Su, L., Jin, S., and Zhang, Y. (2015).
\newblock Specification test for panel data models with interactive fixed
  effects.
\newblock {\em Journal of Econometrics}, 186(1):222--244.

\bibitem[Su and Zhang, 2013]{suZhang2013}
Su, L. and Zhang, Y. (2013).
\newblock Nonparametric dynamic panel data models with interactive fixed
  effects: sieve estimation and specification testing.
\newblock {\em Working Paper}.

\bibitem[Sun, 2005]{sun2005}
Sun, H. (2005).
\newblock Mercer theorem for RKHS on noncompact sets.
\newblock {\em Journal of Complexity}, 21(3):337--349.

\bibitem[Wahba, 1990]{W90}
Wahba, G. (1990).
\newblock {\em Spline models for observational data}.
\newblock SIAM.

\bibitem[Wang, 2011]{wang2011}
Wang, Y. (2011).
\newblock {\em Smoothing splines: methods and applications}.
\newblock CRC Press.

\bibitem[Weinberger, 1974]{weinberger1974}
Weinberger, H.~F. (1974).
\newblock {\em Variational methods for eigenvalue approximation}, volume~15.
\newblock SIAM.

\bibitem[Zhao et~al., 2016]{ZCL16}
Zhao, T., Cheng, G., and Liu, H. (2016).
\newblock A partially linear framework for massive heterogeneous data.
\newblock {\em The Annals of Statistics}, 44(4):1400--1437.
}
\end{thebibliography}

\newpage
\setcounter{page}{1}
\setcounter{section}{0}
\renewcommand{\thesection}{S.\arabic{section}}
\setcounter{equation}{0}
\renewcommand{\theequation}{S.\arabic{equation}}
\setcounter{lemma}{0}
\renewcommand{\thelemma}{S.\arabic{lemma}}
\setcounter{proposition}{0}
\renewcommand{\theproposition}{S.\arabic{proposition}}

\begin{frontmatter}
\begin{center}
\textit{Supplement to}
\end{center}
\title{Statistical Inference on Panel Data Models:
A Kernel Ridge Regression Method}
\end{frontmatter}
This supplement document contains 
proofs and other relevant results that were not included in the main text and appendix.
In Section \ref{supplement:section:hetero:model}, we prove Lemmas \ref{lemma:id} and \ref{bound:RiUit}, Propositions \ref{prop:maximum:term} and \ref{prop:concentration}. In Section \ref{supplement:section:homo:model}, 
we prove Lemmas \ref{lemma:P:minus:Pstar}, \ref{concentration:homo}, \ref{lemma:S:M:eta:g:eta}
\ref{lemma:ANTm:2:Im}, \ref{lemma:bound:ANT} and \ref{lemma:normality:epsilon:P:Kxi}.
We also provide additional Lemmas \ref{lemma:Vfg:0}, \ref{lemma: V & g} and
\ref{lemma:P:Pstar} as well as their proofs.
Lemmas \ref{lemma:Vfg:0} and \ref{lemma: V & g} give mild conditions to guarantee the validity of Assumption \ref{A5};
Lemma \ref{lemma:P:Pstar} is useful for proving Lemma \ref{concentration:homo}. 

\section{Additional Proofs or Other Relevant Results for Heterogeneous Model}\label{supplement:section:hetero:model}
\begin{proof}[Proof of Lemma \ref{lemma:id}]
For any $\theta,\theta_k=(\beta_k,g_k)\in\Theta_i$ 
for $k=1,2$,  it holds from (\ref{prop:Ri:Pi}) that
\begin{eqnarray*}
&&\langle DS_{i,M,\eta_i}^\star(\theta)\theta_1,\theta_2\rangle_i\\
&=&\langle E\{DS_{i,M,\eta_i}(\theta)\theta_1\},\theta_2\rangle_i\\
&=&\frac{1}{T}\sum_{t=1}^TE\left(\langle R_iU_{it},\theta_1\rangle_i\langle R_iU_{it},\theta_2\rangle_i\right)+\langle P_i\theta_1,\theta_2\rangle_i\\
&=&\frac{1}{T}\sum_{t=1}^TE\left((g_1(X_{it})+Z_t'\beta_1)(g_2(X_{it})+Z_t'\beta_2)\right)
+\eta_i\langle g_1,g_2\rangle_{\mathcal{H}_i}\\
&=&E\left((g_1(X_i)+Z'\beta_1)(g_2(X_i)+Z'\beta_2)\right)
+\eta_i\langle g_1,g_2\rangle_{\mathcal{H}_i}\\
&=&\langle\theta_1,\theta_2\rangle_i,
\end{eqnarray*}
which implies that $DS_{i,M,\eta_i}^\star(\theta)=id$, the identity
operator on $\Theta_i$.
\end{proof}

\begin{proof}[Proof of Lemma \ref{bound:RiUit}]
It follows by Proposition \ref{prop:Ri:Pi} that
\[
\|R_iU_{it}\|_i^2=K^{(i)}(X_{it},X_{it})+\left(Z_t-A_i(X_{it})\right)'(\Omega_i+\Sigma_i)^{-1}(Z_t-A_i(X_{it})).
\]
By (\ref{kernel:expansion}) and $\langle A_i,g\rangle_{\star,i}=V_i(G_i,g)$ (see Section \ref{subsec:consist-Hetero}),
\begin{eqnarray*}
A_i(x)&=&\langle A_i,K_x^{(i)}\rangle_{\star,i}=V_i(G_i,K_x^{(i)})\\
&=&\sum_{\nu=1}^\infty\frac{\varphi_\nu^{(i)}(x)}{1+\eta_i\rho_\nu^{(i)}}V_i(G_i,\varphi_\nu^{(i)}).
\end{eqnarray*}
It follows by Assumption \ref{A2} that
$C_{\varphi,i}\equiv\sup_{\nu\ge1}\sup_{x\in\mathcal{X}_i}|\varphi_\nu^{(i)}(x)|<\infty$.
Then we have
\begin{eqnarray*}
K^{(i)}(X_{it},X_{it})&=&\sum_{\nu\ge1}\frac{|\varphi_\nu^{(i)}(X_{it})|^2}{1+\eta_i\rho_\nu^{(i)}}
\le C_{\varphi,i}^2h_i^{-1},\\
A_i(X_{it})'(\Omega_i+\Sigma_i)^{-1}A_i(X_{it})&\le&c_1^{-1}
A_i(X_{it})'A_i(X_{it})\\
&\le&c_1^{-1}\sum_{\nu\ge1}\frac{|\varphi_\nu^{(i)}(X_i)|^2}{(1+\eta_i\rho_\nu^{(i)})^2}
\sum_{\nu\ge1}V_i(G_i',\varphi_\nu^{(i)})V_i(G_i,\varphi_\nu^{(i)})\\
&\le&c_1^{-1}C_{\varphi,i}^2C_{G_i}^2h_i^{-1},\\
Z_t'(\Omega_i+\Sigma_i)^{-1}Z_t&\le& c_1^{-1}Z_t'Z_t,
\end{eqnarray*}
where $C_{G_i}^2=\sum_{\nu\ge1}V_i(G_i',\varphi_\nu^{(i)})V_i(G_i,\varphi_\nu^{(i)})$.
By Assumption \ref{A3}, $C_{G_i}^2$ is a finite positive constant.
Then (\ref{bound:eqn:RiUit}) holds for
$C_i^2=\max\{C_{\varphi,i}^2,2c_1^{-1}C_{\varphi,i}^2C_{G_i}^2,2c_1^{-1}\}$.

To show (\ref{two:norms:isup:i}), first notice that, for any $\theta=(\beta,g)\in\Theta_i$,
\[
\|\theta\|_{i,\sup}=\sup_{x\in\mathcal{X}_i,\|z\|_2=1}|g(x)+z'\beta|.
\]
The ``$\ge$" is obvious. To show ``$\le$", note that for any $x\in\mathcal{X}_i$, choose $z_x=\textrm{sign}(g(x))\beta/\|\beta\|_2$. Then
\[
|g(x)+z_x'\beta|=|g(x)|+\|\beta\|_2.
\]
Therefore,
\[
\sup_{x\in\mathcal{X}_i,\|z\|_2=1}|g(x)+z'\beta|\ge\sup_{x\in\mathcal{X}_i}|g(x)+z_x'\beta|=
\sup_{x\in\mathcal{X}_i}|g(x)|+\|\beta\|_2=\|\theta\|_{i,\sup}.
\]

Following Proposition \ref{prop:Ri:Pi} and the proof of (\ref{bound:eqn:RiUit}), for $u=(x,z)$ with $x\in\mathcal{X}_i$
and $\|z\|_2=1$,
\begin{eqnarray*}
|g(x)+z'\beta|=|\langle R_iu,\theta\rangle_i|\le\|R_iu\|_i\|\theta\|_i\le C_i(1+h_i^{-1/2})\|\theta\|_i.
\end{eqnarray*}
This proves (\ref{two:norms:isup:i}).
\end{proof}

\begin{proof}[Proof of Proposition \ref{prop:maximum:term}]
Since ${f}_{1t}$, ${f}_{2t}$, $v_{it}$ and $\epsilon_{it}$ 
all have finite $\alpha$th moments, it follows by (\ref{DGP of X}) that
$X_{it}$ and $Z_t$ both have finite $\alpha$th moments,
i.e., $E(\|X_i\|_2^\alpha)<\infty$ and $E(\|Z\|_2^\alpha)<\infty$.
Define
\[
C_T(\xi)=\inf\{x| TP(\|Z\|_2>x)\le\xi\},\,\,\,\,\xi>0.
\]
By Markov inequality,
\[
P\left(\|Z\|_2>[TE(\|Z\|_2^\alpha)/\xi]^{1/\alpha}\right)\le\frac{\xi}{T},
\]
therefore, 
\[
C_T(\xi)\le\left(\frac{TE(\|Z\|_2^\alpha)}{\xi}\right)^{1/\alpha}.
\]
Thanks to the $\phi$-mixing condition (see Assumption \ref{A1}), it follows by \citet[Theorem 1]{OBrien1974} that
for any $\xi>0$,
\[
\liminf_{T\rightarrow\infty}P\left(\max_{1\le t\le T}\|Z_t\|_2\le C_T(\xi)\right)=\exp(-b\xi),
\]
where $b>0$ is a constant.
For arbitrary $\varepsilon>0$, choose $\xi>0$ such that
$1-\exp(-b\xi)<\varepsilon/2$.
Then, as $T$ approaches infinity,
\[
P\left(\max_{1\le t\le T}\|Z_t\|_2\le C_T(\xi)\right)\ge\exp(-b\xi)-\varepsilon/2,
\]
leading to that
\[
P\left(\max_{1\le t\le T}\|Z_t\|_2> C_T(\xi)\right)\le 1-\exp(-b\xi)+\varepsilon/2\le\varepsilon.
\]
This proves that
\[
\max_{1\le t\le T}\|Z_t\|_2=O_P(C_T(\xi))=O_P(T^{1/\alpha}).
\]
\end{proof}

\begin{proof}[Proof of Proposition \ref{prop:concentration}]
For notation simplicity, denote 
\begin{eqnarray*}
\mathcal{F}_j&=&\mathcal{F}_1^j, \,\,\,\,j\in[T],\\
\mathcal{F}_0&=&\textrm{trivial $\sigma$-algebra consisting only of the empty set and full sample space.}
\end{eqnarray*}

For any $\theta_1,\theta_2\in\Theta_i$,
define $l_{it}=(\psi_{i,M,t}(U_{it};\theta_1)
-\psi_{i,M,t}(U_{it};\theta_2))R_iU_{it}$.
First of all, we will prove the following concentration inequality: for any $r>0$,
\begin{equation}\label{prop:concentration:eqn:0}
P\left(\bigg\|\sum_{t=1}^T[l_{it}-E(l_{it})]\bigg\|_i\ge r\right)
\le 2\exp\left(-\frac{r^2}{32TC_\phi^2\|\theta_1-\theta_2\|_{i,\sup}^2}\right),
\end{equation}
where $C_\phi\equiv\sum_{t=0}^\infty\phi(t)$.
It follows by Assumption \ref{A1} that $C_\phi$ is finite.
Clearly, (\ref{prop:concentration:eqn:0}) holds for $\theta_1=\theta_2$
since both sides equal to zero.
In what follows, we assume $\theta_1\neq\theta_2$.

Define $\mathbb{M}_{iT}=\sum_{t=1}^Tl_{it}$,
and $f_{iTj}=E(\mathbb{M}_{iT}|\mathcal{F}_j)-E(\mathbb{M}_{iT}|\mathcal{F}_{j-1})$,
$j\in[T]$.
It is easy to see that
\begin{eqnarray}\label{prop:concentration:eqn:1}
\mathbb{M}_{iT}-E(\mathbb{M}_{iT})&=&\sum_{j=1}^Tf_{iTj},\nonumber\\
f_{iTj}&=&\sum_{t=j}^T\left(E(l_{it}|\mathcal{F}_j)-E(l_{it}|\mathcal{F}_{j-1})\right).
\end{eqnarray}
Clearly, $f_{iTj}$ is $\mathcal{F}_j$-measurable.
For $k\in[T]$, define $\mathbb{N}_{iTk}=\sum_{j=1}^k f_{iTj}$
and $\mathbb{N}_{iT0}\equiv0$.
Then $\mathbb{N}_{iTk}=\mathbb{N}_{iTk-1}+f_{iTk}$.
For $\lambda>0$, 
let $u_{k-1}(x)=\lambda\|\mathbb{N}_{iTk-1}+xf_{iTk}\|_i$, $x\in[0,1]$.
Define 
\[
\varphi_{k-1}(x)=E\left(\cosh\left(u_{k-1}(x)\right)|\mathcal{F}_{k-1}\right),\,\,x\in[0,1].
\]
It is easy to see that
\begin{eqnarray*}
\varphi_{k-1}(1)&=&E\left(\cosh\left(\lambda\|\mathbb{N}_{iTk}\|_i\right)|\mathcal{F}_{k-1}\right)\\
\varphi_{k-1}(0)&=&E\left(\cosh\left(\lambda\|\mathbb{N}_{iTk-1}\|_i\right)|\mathcal{F}_{k-1}\right).
\end{eqnarray*}
By the proof of \citet[Theorem 3.2]{pinelis1994} and direct calculations, it can be shown that
\begin{eqnarray}\label{prop:concentration:eqn:2}
\varphi'_{k-1}(x)&=&E\left(\sinh(u_{k-1}(x))u'_{k-1}(x)|\mathcal{F}_{k-1}\right),\nonumber\\
\varphi''_{k-1}(x)&=&E\left(\cosh(u_{k-1})(u'_{k-1}(x))^2+\sinh(u_{k-1}(x))u''_{k-1}(x)|\mathcal{F}_{k-1}\right)\nonumber\\
&\le&E\left(\cosh(u_{k-1})(u'_{k-1}(x))^2+\cosh(u_{k-1}(x))u_{k-1}(x)u''_{k-1}(x)|\mathcal{F}_{k-1}\right)\nonumber\\
&=&\frac{1}{2}E\left(\cosh(u_{k-1}(x))(u_{k-1}(x)^2)''|\mathcal{F}_{k-1}\right)\nonumber\\
&=&\lambda^2E\left(\cosh(u_{k-1}(x))\|f_{iTk}\|_i^2|\mathcal{F}_{k-1}\right).
\end{eqnarray}
Next we will show that $\|f_{iTk}\|_i^2$ is almost surely bounded.
We will first examine the terms $E(l_{it}|\mathcal{F}_k)-E(l_{it})$ for $t\ge k$.
Arbitrarily choose $A\in\mathcal{F}_k$ and $\theta\in\Theta_i$ with $\|\theta\|_i=1$.
Define $X=\langle l_{it},\theta\rangle_i$. Write $X=X^+-X^-$,
where $X^+$ and $X^-$ represent the positive and negative parts of $X$, respectively.
Clearly, $|X|\le\|l_{it}\|_i\|\theta\|_i=\|l_{it}\|_i$ implying that both $X^+$ and $X^-$ belong to
$[0,\|l_{it}\|_i]$. Note that the $X^+$ is $\mathcal{F}_{t}^\infty$-measurable. Therefore,
\begin{eqnarray*}
|E(X^+|A)-E(X^+)|
&=&\bigg|\int_0^{\|l_{it}\|_i}[P(X^+>v|A)-P(X^+>v)]dv\bigg|\\
&\le&\int_0^{\|l_{it}\|_i}|P(X^+>v|A)-P(X^+>v)|dv
\le\|l_{it}\|_i\phi(t-k).
\end{eqnarray*}
Similarly, one can show that
$|E(X^-|A)-E(X^-)|\le\|l_{it}\|_i\phi(t-k)$.
Therefore,
\[
|E(X|A)-E(X)|\le 2\|l_{it}\|_i\phi(t-k).
\]
By arbitrariness of $A\in\mathcal{F}_k$ and by taking supremum over  $\theta\in\Theta_i$ with $\|\theta\|_i=1$, one gets that
\begin{equation}\label{prop:concentration:eqn:3}
\|E(l_{it}|\mathcal{F}_k)-E(l_{it})\|_i\le 2\|l_{it}\|_i\phi(t-k),\,\,t\ge k.
\end{equation}
Similar arguments lead to 
\[
\|E(l_{it}|\mathcal{F}_{k-1})-E(l_{it})\|_i\le 2\|l_{it}\|_i\phi(t-k+1),\,\,t\ge k.
\]
Therefore, for $t\ge k$,
\[
\|E(l_{it}|\mathcal{F}_k)-E(l_{it}|\mathcal{F}_{k-1})\|_i\le2\|l_{it}\|_i(\phi(t-k)+\phi(t-k+1)).
\]
Using (\ref{prop:concentration:eqn:1}) and
the assumption 
$\|l_{it}\|_i\le \|\theta_1-\theta_2\|_{i,\sup}$, it can be shown that
\begin{eqnarray*}
\|f_{iTk}\|_i&\le&\sum_{t=k}^T\|E(l_{it}|\mathcal{F}_k)-E(l_{it}|\mathcal{F}_{k-1})\|_i\\
&\le&\sum_{t=k}^T2\|l_{it}\|_i(\phi(t-k)+\phi(t-k+1))\\
&\le&2\|\theta_1-\theta_2\|_{i,\sup}\sum_{t=k}^T(\phi(t-k)+\phi(t-k+1))\\
&\le&4 C_\phi\|\theta_1-\theta_2\|_{i,\sup}.
\end{eqnarray*}
Therefore, it follows by (\ref{prop:concentration:eqn:2}) that
\[
\varphi''_{k-1}(x)\le 16\lambda^2C_\phi^2\|\theta_1-\theta_2\|_{i,\sup}^2\varphi_{k-1}(x).
\]
Meanwhile, note that $\mathbb{N}_{iTk-1}$ is $\mathcal{F}_{k-1}$-measurable, so we have
\begin{eqnarray*}
\varphi'_{k-1}(0)=\lambda\frac{\sinh(\lambda\|\mathbb{N}_{iTk-1}\|_i)}{\|\mathbb{N}_{iTk-1}\|_i}
E\left(\langle\mathbb{N}_{iTk-1},f_{iTk}\rangle_i|\mathcal{F}_{k-1}\right)=0,\,\,k\ge2,
\end{eqnarray*}
where the last equality follows from $E(f_{iTk}|\mathcal{F}_{k-1})=0$.
Directly using (\ref{prop:concentration:eqn:2}) one also has that $\varphi_{0}'(0)=0$.
So $\varphi'_{k-1}(0)=0$ for all $k\in[T]$.
By \citet[pp. 133, Lemma 3]{dudleyetal1992} we have for $k\in[T]$,
\[
\varphi_{k-1}(x)\le\varphi_{k-1}(0)
\exp\left(8\lambda^2C_\phi^2\|\theta_1-\theta_2\|_{i,\sup}^2x^2\right),
\,\,\,\,x\in[0,1].
\]
In particular, 
\[
\varphi_{k-1}(1)\le\varphi_{k-1}(0)
\exp\left(8\lambda^2C_\phi^2\|\theta_1-\theta_2\|_{i,\sup}^2\right).
\]
Taking expectations on both sides leading to that
\begin{equation}\label{prop:concentration:eqn:4}
E\left(\cosh\left(\lambda\|\mathbb{N}_{iTk}\|_i\right)\right)
\le\exp\left(8\lambda^2C_\phi^2\|\theta_1-\theta_2\|_{i,\sup}^2\right)
E\left(\cosh\left(\lambda\|\mathbb{N}_{iTk-1}\|_i\right)\right). 
\end{equation}
By repeatedly using (\ref{prop:concentration:eqn:4}) and the convention $\mathbb{N}_{iT0}=0$, and by (\ref{prop:concentration:eqn:1}),
we have
\begin{eqnarray*}
E\left(\cosh\left(\lambda\bigg\|\sum_{t=1}^T[l_{it}-E(l_{it})]\bigg
\|_i\right)\right)&=&E\left(\cosh\left(\lambda\|\mathbb{N}_{iTT}\|_i\right)\right)\\
&\le&\exp\left(8T\lambda^2C_\phi^2\|\theta_1-\theta_2\|_{i,\sup}^2\right).
\end{eqnarray*}
Therefore, 
\begin{eqnarray*}
P\left(\bigg\|\sum_{t=1}^T[l_{it}-E(l_{it})]\bigg\|_i\ge r\right)
&=&P\left(\cosh\left(\lambda\bigg\|\sum_{t=1}^T[l_{it}-E(l_{it})]\bigg\|_i\right)
\ge\cosh(\lambda r)\right)\\
&\le&\frac{1}{\cosh(\lambda r)}E\left(\cosh\left(\lambda\bigg\|\sum_{t=1}^T[l_{it}-E(l_{it})]\bigg
\|_i\right)\right)\\
&\le&e\exp\left(-\lambda r+8T\lambda^2C_\phi^2\|\theta_1-\theta_2\|_{i,\sup}^2\right).
\end{eqnarray*}
Then (\ref{prop:concentration:eqn:0}) follows by choosing 
\[
\lambda=\frac{r}{16TC_\phi^2\|\theta_1-\theta_2\|_{i,\sup}^2}.
\]

The rest of the proof follows by chaining argument. 
Let $\psi_2(x)=\exp(x^2)-1$. It follows by (\ref{prop:concentration:eqn:0}) and \citet[Theorem 8.1]{kosorok2008} that for any $\theta_1,\theta_2\in\Theta_i$,
\[
\bigg\|\|\mathbb{Z}_{iM}(\theta_1)-\mathbb{Z}_{iM}(\theta_2)\|_i\bigg\|_{\psi_2}
\le\sqrt{96}C_\phi\|\theta_1-\theta_2\|_{i,\sup}.
\]
It follows by \citet[Theorem 8.4]{kosorok2008} that there exists a universal constant $C>0$,
which only depends on $C_\psi$, such that
for any $\delta>0$,
\begin{eqnarray*}
&&\bigg\|\sup_{\substack{\theta_1,\theta_2\in\mathcal{G}_i(p_i)\\
\|\theta_1-\theta_2\|_{i,\sup}\le\delta}}\|\mathbb{Z}_{iM}(\theta_1)-\mathbb{Z}_{iM}(\theta_2)
\|_i\bigg\|_{\psi_2}\\
&\le& C\left(\int_0^\delta\psi_2^{-1}\left(D_i(\varepsilon,\mathcal{G}_i(p_i),\|\cdot\|_{i,\sup})\right)d\varepsilon+\delta\psi_2^{-1}\left(D_i(\delta,\mathcal{G}_i(p_i),\|\cdot\|_{i,\sup})^2\right)\right)=CJ_i(p_i,\delta).
\end{eqnarray*}
Therefore,
\[
\bigg\|\sup_{\substack{\theta\in\mathcal{G}_i(p_i)\\
\|\theta\|_{i,\sup}\le\delta}}\|\mathbb{Z}_{iM}(\theta)\|_i\bigg\|_{\psi_2}
\le CJ_i(p_i,\delta).
\]
It follows again from \citet[Lemma 8.1]{kosorok2008} that for all $\delta>0, s>0$,
\begin{equation}\label{prop:concentration:eqn:5}
P\left(\sup_{\theta\in\mathcal{G}_i(p_i),\|\theta\|_{i,\sup}\le\delta}
\|\mathbb{Z}_{iM}(\theta)\|_i>s\right)
\le2\exp\left(-\frac{s^2}{C^2J_i(p_i,\delta)^2}\right).
\end{equation}
It is easy to see that for any $\theta\in\mathcal{G}_i(p_i)$, $\|\theta\|_{i,\sup}\le1$.
Let $\sqrt{T}J_i(p_i,1)=\varepsilon^{-1}$, and $Q_\varepsilon=-\log\varepsilon-1$.
Let $\tau=3C\sqrt{\log{N}+\log\log(TJ_i(p_i,1))}$. Then it can be checked that
\begin{eqnarray*}
N(Q_\varepsilon+2)\exp\left(-\frac{\tau^2}{C^2\exp(2)}\right)\rightarrow0,\,\,\,\,N\rightarrow\infty.
\end{eqnarray*}
Since $J_i(p_i,\delta)$ is strictly increasing in $\delta$,
the function $J_i(\delta)\equiv J_i(p_i,\delta)$ has inverse denoted by $J_i^{-1}$.
Then we have
\begin{eqnarray}\label{eqn:chaining:arg}
&&P\left(\max_{i\in[N]}\sup_{\theta\in\mathcal{G}_i(p_i)}\frac{\sqrt{T}\|\mathbb{Z}_{iM}(\theta)\|_i}
{\sqrt{T}J_i(\|\theta\|_{i,\sup})+1}\ge\tau\right)\nonumber\\
&\le&\sum_{i=1}^N\left(P\left(\sup_{\|\theta\|_{i,\sup}\le J_i^{-1}(T^{-1/2})}\frac{\sqrt{T}\|\mathbb{Z}_{iM}(\theta)\|_i}
{\sqrt{T}J_i(\|\theta\|_{i,\sup})+1}\ge\tau\right)\right.\nonumber\\
&&\left.+\sum_{l=0}^{Q_\varepsilon}P\left(\sup_{J_i^{-1}(T^{-1/2}\exp(l))
\le\|\theta\|_{i,\sup}\le J_i^{-1}(T^{-1/2}\exp(l+1))}\frac{\sqrt{T}\|\mathbb{Z}_{iM}(\theta)\|_i}
{\sqrt{T}J_i(\|\theta\|_{i,\sup})+1}\ge\tau\right)\right)\nonumber\\
&\le&\sum_{i=1}^N\left(P\left(\sup_{\|\theta\|_{i,\sup}\le J_i^{-1}(T^{-1/2})}\|\mathbb{Z}_{iM}(\theta)\|_i\ge T^{-1/2}\tau\right)\right.\nonumber\\
&&\left.+\sum_{l=0}^{Q_\varepsilon}P\left(\sup_{\|\theta\|_{i,\sup}\le J_i^{-1}(T^{-1/2}\exp(l+1))}
\|\mathbb{Z}_{iM}(\theta)\|_i\ge T^{-1/2}\tau\exp(l)\right)\right)\nonumber\\
&\le&\sum_{i=1}^N\left(2\exp\left(-\tau^2/C^2\right)+\sum_{l=0}^{Q_\varepsilon}
2\exp\left(-\tau^2/(C^2\exp(2))\right)\right)\nonumber\\
&=&2N(Q_\varepsilon+2)\exp\left(-\frac{\tau^2}{C^2\exp(2)}\right)\rightarrow0,\,\,\textrm{ as $N\rightarrow\infty$.}
\end{eqnarray}
This proves the desirable conclusion with $C_0=3C$.
\end{proof}

\section{Additional Proofs or Other Relevant Results for Homogeneous Model}\label{supplement:section:homo:model}
The following lemma gives mild conditions that guarantee Assumption \ref{A5}. 
Before stating the lemma, we borrow the concept of complete continuity from \citet[page 50]{weinberger1974}. A bilinear functional $A(\cdot, \cdot)$  on $\mathcal{H} \times \mathcal{H}$ is said to be completely continuous w.r.t another bilinear functional $B(\cdot, \cdot)$ if for any $\epsilon>0$, there exists finite number of functionals $l_1,l_2,...,l_k$ on $\mathcal{H}$ such that $l_i(g)=0, i=1,2,..,k,$ implies $A(g,g)\leq \epsilon B(g,g)$. 

Let $U$ be an open subset of $\mathcal{X}$ and 
$U^{NT}\equiv\underbrace{U\times U\times\cdots\times U}_{\textrm{$NT$ items}}$. Let $C(\mathcal{X})$ be the set of all continuous functions on $\mathcal{X}$
and $\mathcal{H}\subseteq C(\mathcal{X})$. 
Let $\bold{x}$ denote the $NT$-vector $(x_{11},\ldots,x_{1T},\ldots,x_{N1},\ldots,x_{NT})$.

\begin{lemma}\label{lemma:Vfg:0}
Suppose $1\notin\mathcal{H}$,
and $p(\bold{x}|\f1t)>0$ for $\bold{x}\in U^{NT}$, where $p(\bold{x}|\f1t)$ is the joint conditional density of $X_{11},X_{12},...,X_{NT}$ given $\f1t$. If $V(f,g)=0$ for all $f\in \mathcal{H}$, then $g=0$.  
\end{lemma}
\begin{proof}[Proof of Lemma \ref{lemma:Vfg:0}]
For simplicity, we assume that $f_{1t}, X_{it}$ are both univariate.  
By assumption, 
$0=V(g,g)=\sum_{i=1}^NE\left\{(\tau_i g)'P(\tau_ig)\big|\mathcal{F}_1^T\right\}/(NT)$. Hence it follows that 
\begin{equation}\label{eq:lemma:Vfg:0:eq1}
0=\int (g(x_{i1}),g(x_{i2}),...,g(x_{iT}))P(g(x_{i1}),g(x_{i2}),...,g(x_{iT}))'p(\bold{x}|\f1t)d\bold{x}, \textrm{ for all } i\in [N].
\end{equation}
Since the integrand in (\ref{eq:lemma:Vfg:0:eq1}) is continuous and nonnegative, it holds that,
for all $i\in[N]$ and $\bold{x}$ with $p(\bold{x}|\f1t)>0$,
\begin{equation}\label{eq:lemma:Vfg:0:eq2}
(g(x_{i1}),g(x_{i2}),...,g(x_{iT}))P(g(x_{i1}),g(x_{i2}),...,g(x_{iT}))'=0.
\end{equation}
By definition, $P$ is a projection matrix whose image is the orthogonal space of 
the linear space spanned by $F_1$ and $\bar{X}$. Therefore, it yields that
\begin{equation}
(g(x_{i1}),g(x_{i2}),...,g(x_{iT}))=\alpha_i(f_{11},f_{12},...,f_{1T})+\beta_i(\bar{x}_1,\bar{x}_2,...,\bar{x}_T),
\label{comparing:the:two:equations:1}
\end{equation}
for some $\alpha_i, \beta_i \in \mathbb{R}$. Consider  
$\bold{x}=(x_{11},x_{12},...,x_{N,T-1},{x}_{NT})$ and 
$\widetilde{\bold{x}}=(x_{11},x_{12},...,x_{N,T-1},\widetilde{x}_{NT})\in U^{NT}$ with 
${x}_{NT}\neq \widetilde{x}_{NT}$ and $p(\bold{x}|\f1t)>0, p(\widetilde{\bold{x}}|\f1t)>0$, 
i.e., the two points differ only on the last element. Applying (\ref{eq:lemma:Vfg:0:eq2}) to 
point  $\widetilde{\bold{x}}$, we have 
\begin{equation}
(g(x_{i1}),g(x_{i2}),...,g(\widetilde{x}_{iT}))=\widetilde{\alpha}_i(f_{11},f_{12},...,f_{1T})+\widetilde{\beta}_i(\bar{x}_1,\bar{x}_2,...,\widetilde{\bar{x}}_T),\label{comparing:the:two:equations:2}
\end{equation}
for some $\widetilde{\alpha}_i, \widetilde{\beta}_i \in \mathbb{R}$.
Comparing (\ref{comparing:the:two:equations:1}) and (\ref{comparing:the:two:equations:2}), 
and by the fact $T>q_1+d=2$, it holds that $\alpha_i=\widetilde{\alpha}_i, \beta_i=\widetilde{\beta}_i=0$. Hence $(g(x_{i1}),g(x_{i2}),...,g({x}_{iT})) \in span((f_{11},f_{12},...,f_{1T}))$ for all $p(\bold{x}|\f1t)>0,  i \in [N]$,
and it happens if and only if $g=0$. 
\end{proof}

\begin{lemma} \label{lemma: V & g}
Suppose $\mathcal{X}$ is compact. Furthermore if $V(f,g)=0$ for all $f\in \mathcal{H}$ implies $g=0$, then Assumption \ref{A5} is valid.
\end{lemma}
\begin{Remark}
The compactness of $\mathcal{X}$ can be relaxed by Mercer's theorem; see \citet{sun2005}. 
\end{Remark}

\begin{proof}[Proof of Lemma \ref{lemma: V & g}]
Define bilinear functionals $W(g,\widetilde{g})=\sum_{i=1}^NE\left\{(\tau_i g)'(\tau_i\widetilde{g})\big|\mathcal{F}_1^T\right\}/(NT)$, and $J(g,\widetilde{g})=\langle g,\widetilde{g}\rangle_{\mathcal{H}}$. Clearly, $V(g,g)\leq W(g,g)$. Let $\mu$ be a measure  such that 
$$\int gd\mu=\frac{1}{NT}\sum_{i=1}^N\sum_{t=1}^TE(g(X_{it})|\f1t). $$
Hence, $\int g^2d\mu=W(g,g)$. By Mercer's theorem, the kernel $\bar{K}$ of $\mathcal{H}$ follows the expansion:
$$\bar{K}(x,y)=\sum_{i=1}^\infty\lambda_ie_i(x)e_i(y).$$
where $\lambda_i$ is a non-increasing positive sequence converging to zero and $\{e_i\}_{i=1}^\infty$  forms an orthonormal basis of $L_2(\mu)$, so that $W(e_i, e_j)=\delta_{ij}$. Moreover, $\{\sqrt{\lambda_i}e_i\}_{i=1}^\infty$ is also an orthonormal basis of $\mathcal{H}$, which is proved in \citet{cuckerSmale2001}. As a consequence, any $g\in \mathcal{H}$ simultaneously admits the following expansions:
$$g=\sum_{i=1}^\infty W(g,e_i)e_i,\,\,\,\,
g=\sum_{i=1}^\infty J(g, \sqrt{\lambda_i}e_i)\sqrt{\lambda_i}e_i
$$
with $\sum_{i=1}^\infty W^2(g,e_i)<\infty$ and 
$\sum_{i=1}^\infty J^2(g, \sqrt{\lambda_i}e_i)< \infty$. This implies $W(g,e_i)=\lambda_iJ(g,e_i)$.
 For any $\epsilon>0$, choose integer $k$ large enough so that $\lambda_i< \epsilon$ for $i>k$. Define functionals $l_i(g)=W(g,e_i) ,i=1,2,...,k$. By direct direct examinations, if $l_i(g)=0$ for $i=1,2,...,k$, then
  $$W(g,g)=\sum_{i=k+1}^\infty W^2(g,e_i)=\sum_{i=k+1}^\infty \lambda_i^2J^2(g,e_i)\leq \epsilon \sum_{i=k+1}^\infty \lambda_iJ^2(g,e_i)=\epsilon J(g,g).$$
Since $V(g,g)\leq W(g,g)\leq \epsilon J(g,g)$, $V$ is completely continuous w.r.t $J$. By \citet[Theorem 3.1, page 52]{weinberger1974}, there are positive eigenvalues $\{\alpha_i\}_{i=1}^{\infty}$ converging to zero and eigenfunctions $\{\widetilde{\varphi}_i\}_{i=1}^{\infty} \in \mathcal{H}$ such that $V(\widetilde{\varphi}_i,\widetilde{\varphi}_j)=\alpha_i \delta_{ij}$, $J(\widetilde{\varphi}_i, \widetilde{\varphi}_j)=\delta_{ij}$ and 
$$g=\sum_{i=1}^\infty J(g, \widetilde{\varphi}_i)\widetilde{\varphi}_i, \textrm{ for all } g \in \mathcal{H}.$$
The above implies $V(g, \widetilde{\varphi}_i)=\alpha_iJ(g, \widetilde{\varphi}_i)$. Take $\varphi_i=\widetilde{\varphi}_i/\sqrt{\alpha_i}$ and $\rho_i=1/\alpha_i$, then $\{\varphi_i\}_{i=1}^\infty$ and $\{\rho_i\}_{i=1}^{\infty}$ will satisfy Assumption \ref{A5}.
\end{proof}

\begin{proof}[Proof of Lemma \ref{lemma:P:minus:Pstar}]

Throughout we let $\|A\|_F=\sqrt{\textrm{Tr}(AA')}$ be the 
Frobenius norm.
Clearly, 
\[
\widetilde{\Sigma}=\left(\begin{matrix}
0&0&\cdots&0\cr
\bar{v}_1&\bar{v}_2&\cdots&\bar{v}_T
\end{matrix}\right),\,\,
\Sigma_\star\widetilde\Sigma'=\left(0_{(q_1+d)\times q_1},\sum_{t=1}^T Z_t^\star\bar{v}_t'\right).
\]
By direct examinations we have
\begin{eqnarray}\label{P:Pstar:eqn:-1}
\Sigma\Sigma'-\Sigma_\star\Sigma_\star'&=&\Sigma_\star\widetilde\Sigma'+\widetilde\Sigma'\Sigma_\star
+\widetilde\Sigma\widetilde\Sigma'\equiv R.
\end{eqnarray}
By independence of $v_{it}$ and $Z_t^\star$, it can be shown that
\begin{eqnarray}\label{P:Pstar:eqn:0}
&&E\left(\|\sum_{t=1}^T Z_t^\star\bar{v}_t'\|_F^2\right)=\sum_{t,l=1}^T\textrm{Tr}\left(E\left(
\bar{v}_t'\bar{v}_l\right)E\left((Z_t^\star)'Z_l^\star\right)\right)=O(T/N),\nonumber\\
&&E\left(\|\widetilde{\Sigma}\widetilde{\Sigma}'\|_F^2\right)\le 
E\left(\textrm{Tr}\left(\widetilde{\Sigma}\widetilde{\Sigma}'\widetilde{\Sigma}\widetilde{\Sigma}'\right)\right)=O(T^2/N^2).
\end{eqnarray}
Hence, 
\begin{eqnarray}\label{P:Pstar:eqn:1}
E\left(\|\Sigma_\star\widetilde{\Sigma}'\|_F^2
\right)&=&E\left(\|\sum_{t=1}^T Z_t^\star\bar{v}_t'\|_F^2\right)=O(T/N),\nonumber\\
E(\|R\|_F^2)&\le& 8E(\|\Sigma_\star\widetilde{\Sigma}'\|_F^2)+2
E(\|\widetilde{\Sigma}\widetilde{\Sigma}'\|_F^2)=O(T/N+(T/N)^2).
\end{eqnarray}
Since
\begin{eqnarray*}
\|(\Sigma\Sigma')^{-1}-(\Sigma_\star\Sigma_\star')^{-1}\|_{\textrm{op}}
&=&\|(\Sigma_\star\Sigma_\star')^{-1}
R(\Sigma\Sigma')^{-1}\|_{\textrm{op}}\\
&\le&\|(\Sigma_\star\Sigma_\star')^{-1}\|_{\textrm{op}}\|R\|_{\textrm{op}}
\|(\Sigma\Sigma')^{-1}\|_{\textrm{op}},
\end{eqnarray*}
it follows by Assumption \ref{A7} and (\ref{P:Pstar:eqn:-1}) and H\"{o}lder inequality that
\begin{equation}\label{P:Pstar:eqn:2}
E\left(\|(\Sigma\Sigma')^{-1}-(\Sigma_\star\Sigma_\star')^{-1}\|_{\textrm{op}}^{1+\omega}\right)=
O((T^3N)^{-(1+\omega)/2}+(TN)^{-(1+\omega)}),
\end{equation} 
where $\omega=(\zeta-4)/(\zeta+4)$.
Note that $E(\widetilde{\Sigma}'\widetilde{\Sigma})=
\sigma_v^2 I_T$ and $E\textrm{Tr}(\widetilde{\Sigma}'\widetilde{\Sigma})=O(T/N)$,
where $\sigma_v^2=E(v_{it}'v_{it})$ is a constant.
By direct examinations
\begin{eqnarray*}
&&P_\star-P\\
&=&\Sigma'\left((\Sigma\Sigma')^{-1}-(\Sigma_\star\Sigma_\star')^{-1}\right)\Sigma
+\Sigma_\star'(\Sigma_\star\Sigma_\star')^{-1}\widetilde{\Sigma}+
\widetilde{\Sigma}'(\Sigma_\star\Sigma_\star')^{-1}\Sigma_\star+
\widetilde{\Sigma}'(\Sigma_\star\Sigma_\star')^{-1}\widetilde{\Sigma}.
\end{eqnarray*} 
It follows by (\ref{P:Pstar:eqn:0}),
(\ref{P:Pstar:eqn:1}) and (\ref{P:Pstar:eqn:2}) and H\"{o}lder inequality that
\begin{eqnarray}
&&E\left(\|P-P_\star\|_{\textrm{op}}\right) \nonumber\\
&\le&E\left(\|\Sigma\Sigma'\|_{\textrm{op}}\|(\Sigma\Sigma')^{-1}-
(\Sigma_\star\Sigma_\star')^{-1}\|_{\textrm{op}}\right)
+2E\left(\|\widetilde{\Sigma}'(\Sigma_\star\Sigma_\star')^{-1}\Sigma_\star\|_{\textrm{op}}\right)
+E\left(\|\widetilde{\Sigma}'(\Sigma_\star\Sigma_\star')^{-1}\widetilde{\Sigma}\|_{\textrm{op}}\right)\nonumber\\
&\le&E\left(\|(\Sigma\Sigma')^{-1}-(\Sigma_\star\Sigma_\star')^{-1}\|_{\textrm{op}}^{1+\omega}\right)
^{1/(1+\omega)} E\left(\|\Sigma\Sigma'\|_{\textrm{op}}^{(1+\omega)/\omega}\right)^{\omega/(1+\omega)}\nonumber\\
&&+2E\left(\|(\Sigma_\star\Sigma_\star')^{-1}\|_{\textrm{op}}\right)^{1/2}
E\left(\textrm{Tr}\left(\widetilde{\Sigma}\widetilde{\Sigma}'\right)\right)^{1/2}+E\left(\|(\Sigma_\star\Sigma_\star')^{-1}\|_{\textrm{op}}\right)
E\left(\textrm{Tr}\left(\widetilde{\Sigma}\widetilde{\Sigma}'\right)\right)\nonumber\\
&=&O((T^3N)^{-1/2}+(TN)^{-1})T+O(T^{-1/2}\sqrt{T/N})+O(T^{-1}(T/N))=O(N^{-1/2}).\nonumber\\
\label{eq:lemma:P:minus:Pstar:eq1}
\end{eqnarray}
This proves (\ref{P:and:Pstar}). Next we show (\ref{F2:P:and:Pstar}).
For any $i\in[N]$,
\begin{eqnarray*}
&&E\{\gamma_{2i}'F_2'(P-P_\star)K_{\mathbb{X}_i}|\mathcal{F}_1^T\}\\
&=&E\{\gamma_{2i}'F_2'\Sigma'[(\Sigma\Sigma')^{-1}-(\Sigma_\star\Sigma_\star')^{-1}]\Sigma K_{\mathcal{X}_i}|\mathcal{F}_1^T\}
+E\{\gamma_{2i}'F_2'\Sigma_\star'(\Sigma_\star\Sigma_\star')^{-1}
\widetilde{\Sigma}K_{\mathcal{X}_i}|\mathcal{F}_1^T\}\\
&&+E\{\gamma_{2i}'F_2'\widetilde{\Sigma}'(\Sigma_\star\Sigma_\star')^{-1}\Sigma_\star K_{\mathbb{X}_i}|\mathcal{F}_1^T\}+
E\{\gamma_{2i}'F_2'\widetilde{\Sigma}'(\Sigma_\star\Sigma_\star')^{-1}\widetilde{\Sigma} K_{\mathbb{X}_i}|\mathcal{F}_1^T\}.
\end{eqnarray*}
 
By direct calculations it can be examined that 
\begin{eqnarray*}
&&E\left(\|\Sigma'[(\Sigma\Sigma')^{-1}-(\Sigma_\star\Sigma_\star')^{-1}]\Sigma F_2\|_{\textrm{op}}\right)
\\
&\le&E\left(\|(\Sigma\Sigma')^{-1}-(\Sigma_\star\Sigma_\star')^{-1}
\|_{\textrm{op}}\times \|\Sigma\Sigma'\|_{\textrm{op}}\times \|F_2\|_{\textrm{op}}\right)\\
&\le& E\left(\|(\Sigma\Sigma')^{-1}-(\Sigma_\star\Sigma_\star')^{-1}
\|_{\textrm{op}}^{1+\omega}\right)^{1/(1+\omega)}\\
&&\times E\left(\|\Sigma\Sigma'\|_{\textrm{op}}^{2(1+\omega)/\omega}\right)^{\omega/(2(1+\omega))}E\left(\|F_2\|_{\textrm{op}}^{2(1+\omega)/\omega}\right)^{\omega/(2(1+\omega))}\\
&=&O((T^3N)^{-1/2}+(TN)^{-1})T^{3/2}=O(N^{-1/2}+T^{1/2}/N),
\end{eqnarray*}
and
\begin{eqnarray*}
&&E\left(\|\Sigma_\star'(\Sigma_\star\Sigma_\star')^{-1}\widetilde{\Sigma}F_2\|_{\textrm{op}}\right)\\
&\le&E\left(\|(\Sigma_\star\Sigma_\star')^{-1}\|_{\textrm{op}}^{1/2}\textrm{Tr}(F_2'\widetilde{\Sigma}'\widetilde{\Sigma}F_2)^{1/2}\right)\\
&\le&E\left(\|(\Sigma_\star\Sigma_\star')^{-1}\|_{\textrm{op}}\right)^{1/2}
E\left(\textrm{Tr}(F_2'\widetilde{\Sigma}'\widetilde{\Sigma}F_2)\right)^{1/2}\\
&=&E\left(\|(\Sigma_\star\Sigma_\star')^{-1}\|_{\textrm{op}}\right)^{1/2}E\left(\textrm{Tr}(F_2'F_2)\right)^{1/2}O(N^{-1/2})\\
&=&E\left(\|(\Sigma_\star\Sigma_\star')^{-1}\|_{\textrm{op}}\right)^{1/2}E\left(\|F_2'F_2\|_{\textrm{op}}\right)^{1/2}O(N^{-1/2})=O(N^{-1/2}).
\end{eqnarray*}
For any $g\in\mathcal{H}$ with $\|g\|=1$ (implying $|g(x)|\le c_\varphi h^{-1/2}$ for any $x$),
we have
\begin{eqnarray*}
\|E\{F_2'\Sigma_\star'(\Sigma_\star\Sigma_\star')^{-1}
\widetilde{\Sigma}\tau_i g|\mathcal{F}_1^T\}\|_2\le\|F_2'\Sigma_\star'(\Sigma_\star\Sigma_\star')^{-1}\|_{\textrm{op}}\|E\{\widetilde{\Sigma}\tau_i g|\mathcal{F}_1^T\}\|_2=
O_P(\|E\{\widetilde{\Sigma}\tau_i g|\mathcal{F}_1^T\}\|_2).
\end{eqnarray*}
On the other hand, by direct examinations we have 
\begin{eqnarray*}
\|\widetilde{\Sigma}\tau_i g\|_2^2&=&\sum_{t,l=1}^T\bar{v}_t'\bar{v}_lg(x_{it})g(x_{il}).
\end{eqnarray*}
Meanwhile, for any $t\neq l$, $\bar{v}_t'g(x_{it})$ and $\bar{v}_lg(x_{il})$ are independent 
conditional on $\mathcal{F}_1^T$, and 
\[
E\{\bar{v}_lg(x_{il})|\mathcal{F}_1^T\}=\frac{1}{N}E\{v_{il}g(x_{il})|\mathcal{F}_1^T\}+\frac{1}{N}\sum_{k\neq i}E\{v_{kl}g(x_{il})|\mathcal{F}_1^T\}=\frac{1}{N}E\{v_{il}g(x_{il})|\mathcal{F}_1^T\}.
\]
The last equality holds because $v_{kl}$ and $g(x_{il})$ are conditional independent (on $\mathcal{F}_1^T$)
for $k\neq i$ and the former has mean zero.
This leads us to that
\begin{eqnarray*}
E\{\|\widetilde{\Sigma}\tau_i g\|_2^2|\mathcal{F}_1^T\}
&=&\sum_{t=1}^T E\{\bar{v}_t'\bar{v}_t g(x_{it})^2|\mathcal{F}_1^T\}+\sum_{t\neq l}
E\{\bar{v}_t'\bar{v}_lg(x_{it})g(x_{il})|\mathcal{F}_1^T\}\\
&=&\sum_{t=1}^T E\{\bar{v}_t'\bar{v}_t g(x_{it})^2|\mathcal{F}_1^T\}+\sum_{t\neq l}
E\{\bar{v}_t'g(x_{it})|\mathcal{F}_1^T\}E\{\bar{v}_lg(x_{il})|\mathcal{F}_1^T\}\\
&=&\sum_{t=1}^T E\{\bar{v}_t'\bar{v}_t g(x_{it})^2|\mathcal{F}_1^T\}+\frac{1}{N^2}\sum_{t\neq l}
E\{v_{it}'g(x_{it})|\mathcal{F}_1^T\}E\{v_{il}g(x_{il})|\mathcal{F}_1^T\}\\
&=&O_P\left(\frac{T}{Nh}+\frac{T^2}{N^2h}\right).
\end{eqnarray*}
Therefore,
\[
\|E\{F_2'\Sigma_\star'(\Sigma_\star\Sigma_\star')^{-1}
\widetilde{\Sigma}\tau_i g|\mathcal{F}_1^T\}\|_2=O_P\left(\sqrt{\frac{T}{Nh}}+\frac{T}{N\sqrt{h}}\right),
\]
where the $O_P$ term is free of $g$.

Similarly, we can show that
\begin{eqnarray}
&&E\left(\|\widetilde{\Sigma}'(\Sigma_\star\Sigma_\star')^{-1}\widetilde{\Sigma}F_2\|_{\textrm{op}}\right)\nonumber\\
&\le&E\left(\|\widetilde{\Sigma}\widetilde{\Sigma}'\|_{\textrm{op}}^2\right)^{1/4}
E\left(\|(\Sigma_\star\Sigma_\star')^{-1}\|_{\textrm{op}}^4\right)^{1/4}
E\left(F_2'\widetilde{\Sigma}'\widetilde{\Sigma}F_2\right)^{1/2}\nonumber\\
&=&O(\sqrt{T/N})O(1/T)O(\sqrt{T/N})=O(1/N). \label{eq:lemma:P:minus:Pstar:eq2}
\end{eqnarray}

Combining the above, we get that
\begin{eqnarray*}
\|E\{\gamma_{2i}'F_2'(P-P_\star)K_{\mathbb{X}_i}|\mathcal{F}_1^T\}\|=O_P\left(\sqrt{\frac{T}{Nh}}
+\frac{T}{N\sqrt{h}}\right),
\end{eqnarray*}
where the $O_P$ is free of $i\in[N]$.
Proof completed.
\end{proof}

\begin{lemma}\label{lemma:P:Pstar}
Suppose that Assumptions \ref{A1}, \ref{A5} and \ref{A7} hold.
Let $\psi$ satisfy the conditions in Lemma \ref{concentration:homo}. Then 
\[
\sup\limits_{\|g\|_{\sup}\le1} \frac{1}{\sqrt{N}}\|\sum_{i=1}^N
\psi(\mathbb{X}_i,g)'(P-P_\star)K_{\mathbb{X}_i}\|=O_P\left(1\right),
\]
and
\[
\sup\limits_{\|g\|_{\sup}\le1} \frac{1}{\sqrt{N}}\|\sum_{i=1}^N
E\left(\psi(\mathbb{X}_i,g)'(P-P_\star)K_{\mathbb{X}_i}\big|\mathcal{F}_1^T\right)\|=O_P\left(1\right).
\]
\end{lemma}

\begin{proof}[Proof of Lemma \ref{lemma:P:Pstar}]
For any $g,\widetilde{g}$ satisfying $\|g\|_{\sup}\le1$ and $\|\widetilde{g}\|\le1$,
the former implies that $\|\psi(\mathbb{X}_i,g)\|_2\le L\sqrt{h/T}$ for each $i\in[N]$, and
the latter implies that $\|\widetilde{g}\|_{\sup}\le c_\varphi h^{-1/2}$, by (\ref{P:and:Pstar}) we have
\[
\frac{1}{\sqrt{N}}\bigg|\sum_{i=1}^N\psi(\mathbb{X}_i,g)'(P-P_\star)\tau_i\widetilde{g}\bigg|
\le \frac{1}{\sqrt{N}}\sum_{i=1}^N\|\psi(\mathbb{X}_i,g)\|_2\|\tau_i\widetilde{g}\|_2\|P-P_\star\|_{\textrm{op}}=O_P\left(1\right),
\]
and
\[
\frac{1}{\sqrt{N}}\bigg|\sum_{i=1}^NE\left(\psi(\mathbb{X}_i,g)'
(P-P_\star)\tau_i\widetilde{g}\big|\mathcal{F}_1^T\right)\bigg|
\le Lc_\varphi \sqrt{N}E\left(\|P-P_\star\|_{\textrm{op}}\big|\mathcal{F}_1^T\right)=O_P(1).
\]
Proof is completed.
\end{proof}

\begin{proof}[Proof of Lemma \ref{concentration:homo}]
It follows by Lemma \ref{lemma:P:Pstar} that we only need to consider the 
process $Z_M^\star(g)=\frac{1}{\sqrt{N}}\sum_{i=1}^N[\psi(\mathbb{X}_i,g)'P_\star K_{\mathbb{X}_i}-
E\{\psi(\mathbb{X}_i,g)'P_\star K_{\mathbb{X}_i}|\mathcal{F}_1^T\}]$ for $g\in\mathcal{H}$
where the items in summation are independent conditional on $\mathcal{F}_1^T$.
Let $\textbf{K}_i=[K(X_{it},X_{il})]_{1\le t,l\le T}$, a $T\times T$ matrix.
By Assumption \ref{A5} it follows that $\textbf{K}_i\le c_\varphi^2 h^{-1}T I_T$.
For any $g_1,g_2\in\mathcal{H}$, 
\begin{eqnarray*}
&&\|(\psi(\mathbb{X}_i,g_1)-\psi(\mathbb{X}_i,g_2))'P_\star K_{\mathbb{X}_i}\|^2\\
&=&(\psi(\mathbb{X}_i,g_1)-\psi(\mathbb{X}_i,g_2))'P_\star\textbf{K}_i P_\star
(\psi(\mathbb{X}_i,g_1)-\psi(\mathbb{X}_i,g_2))\\
&\le&(Lc_\varphi\|P_\star\|_{\textrm{op}}\|g_1-g_2\|_{\sup})^2=(Lc_\varphi\|g_1-g_2\|_{\sup})^2.
\end{eqnarray*}
The last equation follows by $\|P_\star\|_{\textrm{op}}=1$ since $P_\star$ is idempotent.
Notice that $\{X_{it}: i\in[N],t\in[T]\}$ are conditional independent given $\mathcal{F}_1^T$.
It follows by \citet[Theorem 3.5]{pinelis1994} that for any $r\ge0$,
\[
P\left(\|Z_M^\star(g_1)-Z_M^\star(g_2)\|\ge r\bigg
|\mathcal{F}_1^T\right)\le 2\exp\left(-\frac{r^2}{8L^2
c_\varphi^2\|g_1-g_2\|_{\sup}^2}\right).
\]
It follows by \citet[Lemma 8.1]{kosorok2008} that 
\[
\bigg\|\|Z_M^\star(g_1)-Z_M^\star(g_2)\|\bigg\|_{\mathcal{F}_1^T,\psi_2}\le 
5Lc_\varphi\|g_1-g_2\|_{\sup},
\]
where $\|\cdot\|_{\mathcal{F}_1^T,\psi_2}$ denotes the Orlicz-norm conditional on $\mathcal{F}_1^T$
 with respect to $\psi_2(s)=\exp(s^2)-1$. This in turn leads to, by \citet[Theorem 8.4]{kosorok2008}, that
for any $\delta>0$,
\begin{eqnarray*}
&&\bigg\|\sup_{\substack{g_1,g_2\in\mathcal{G}(p)\\
\|g_1-g_2\|_{\sup}\le\delta}}\|Z_M^\star(g_1)-Z_M^\star(g_2)\|
\bigg\|_{\mathcal{F}_1^T,\psi_2}\\
&\le& C\left[\int_0^\delta\psi_2^{-1}\left(D(\varepsilon,\mathcal{G}(p),\|\cdot\|_{\sup})\right)d\varepsilon+\delta\psi_2^{-1}\left(D(\delta,\mathcal{G}(p),\|\cdot\|_{\sup})^2\right)\right]\\
&=&CJ(p,\delta),
\end{eqnarray*}
where $C>0$ is a constant depending on $L,c_\varphi$ only.
Then we have
\[
\bigg\|\sup_{\substack{g\in\mathcal{G}(p)\\
\|g\|_{\sup}\le\delta}}\|Z_M^\star(g)\|\bigg\|_{\mathcal{F}_1^T,\psi_2}
\le CJ(p,\delta).
\]
It follows again from \citet[Lemma 8.1]{kosorok2008} that for all $\delta>0,r>0$,
\begin{equation}\label{Zstar:exp:inequality}
P\left(\sup_{\substack{g\in\mathcal{G}(p)\\
\|g\|_{\sup}\le\delta}}\|Z_M^\star(g)\|\ge r\bigg|\mathcal{F}_1^T\right)
\le 2\exp\left(-\frac{r^2}{C^2 J(p,\delta)^2}\right).
\end{equation}
Let $Q_N=\log(N^{1/2}J(p,1))-1$. It follows from the proof of (\ref{eqn:chaining:arg}) that
\begin{eqnarray}\label{result:of:chaining}
&&P\left(\sup_{g\in\mathcal{G}(p)}\frac{\sqrt{N}\|Z_M^\star(g)\|}{\sqrt{N}J(p,\|g\|_{\sup})+1}
\ge C\sqrt{18\log(Q_N)}\bigg|\mathcal{F}_1^T\right)\nonumber\\
&\le& 2(Q_N+2)\exp\left(-\frac{18C^2\log(Q_N)}{C^2\exp(2)}\right)\le \frac{2(Q_N+2)}{Q_N^2}.
\end{eqnarray}
Taking expectation on both sides of (\ref{result:of:chaining}), we get that
\[
P\left(\sup_{g\in\mathcal{G}(p)}\frac{\sqrt{N}\|Z_M^\star(g)\|}{\sqrt{N}J(p,\|g\|_{\sup})+1}
\ge C\sqrt{18\log(Q_N)}\right)=o(1),\,\,\textrm{as $N\rightarrow\infty$.}
\]
This shows that, with probability approaching one,
\[
\sup_{g\in\mathcal{G}(p)}\frac{\sqrt{N}\|Z_M^\star(g)\|}{\sqrt{N}J(p,\|g\|_{\sup})+1}
\le C\sqrt{18\log(Q_N)}.
\] 
Since $\|g\|_{\sup}\le1$ for any $g\in\mathcal{G}$ and $J(p,\delta)$
is increasing in $\delta$, the above inequality implies that, with probability approaching one,
\[
\sup_{g\in\mathcal{G}(p)}\|Z_M^\star(g)\|\le C\sqrt{18\log(Q_N)}(J(p,1)+N^{-1/2}).
\]
Combining with Lemma \ref{lemma:P:Pstar}, we get that
\begin{eqnarray*}
\sup_{g\in\mathcal{G}(p)}\|Z_M(g)\|
&\le&\sup_{g\in\mathcal{G}(p)}\|Z_M(g)-Z_M^\star(g)\|+\sup_{g\in\mathcal{G}(p)}\|Z_M^\star(g)\|\\
&=&O_P\left(1+\sqrt{\log\log\left(NJ(p,1)\right)}(J(p,1)+N^{-1/2})\right).
\end{eqnarray*}
Proof completed.
\end{proof}

\begin{proof}[Proof of Lemma \ref{lemma:S:M:eta:g:eta}]
By (\ref{basic:model 2}), we have
$e_i'=\epsilon_i'-\Delta_i, \bar{v}=\epsilon_i'-\Delta_i(\bar{X}-\bar{\Gamma}_1'F_1'-\bar{\Gamma}_2'F_2') $ and 
\begin{align}
	(Y_i-\tau_ig_{\eta})'P\kxi&=[\tau_i(g_0-g_{\eta})+\Sigma'\beta_i+e_i]'P\kxi \nonumber\\
	&=[\tau_i(g_0-g_{\eta})+e_i]'P\kxi \nonumber \\
	&=[\tau_i(g_0-g_{\eta})+\epsilon_i+F_2\bar{\Gamma}_2\Delta_i']'P\kxi. \label{eq:proof:lemma:S:M:eta:g:eta:eq1} 
\end{align}
By the definition of $g_\eta$ in the proof Theorem \ref{theorem:homo convergence:rate:result:1} and (\ref{eq:proof:lemma:S:M:eta:g:eta:eq1}), we get that
\begin{equation}\label{eq:proof:lemma:S:M:eta:g:eta:eq2} 
S_{M,\eta}(g_{\eta})=S_{M,\eta}(g_{\eta})-S^{\star}_{M,\eta}(g_{\eta})=T_1+T_2-T_3+W_{\eta}g_{\eta}-E(W_{\eta}g_{\eta}|\f1t)=T_1+T_2-T_3,
\end{equation}
where
\begin{align*}
	T_1&=\frac{1}{NT}\sum_{i=1}^N[\epsilon_i'P\kxi-E(\epsilon_i'P\kxi | \f1t)],\\
	T_2&=\frac{1}{NT}\sum_{i=1}^N[\Delta_i F_2'P\kxi-E(\Delta_i F_2'P\kxi | \f1t)],\\
	T_3&=\kappa(g_{\eta}-g_0).
\end{align*}
Recall that $\kappa$ is defined in the proof of  Theorem \ref{theorem:homo convergence:rate:result:1}. It is worthwhile to mention that the terms 
$W_{\eta}g_{\eta}$ and $E(W_{\eta}g_{\eta}|\f1t)$ cancel each other 
in (\ref{eq:proof:lemma:S:M:eta:g:eta:eq2}) thanks to 
$W_\eta g_\eta \in \f1t$. Next, we will bound $T_1, T_2, T_3$ respectively.

First of all, by (\ref{ghat:minus:g0:rate:geta:minus:g0}) and (\ref{ghat:minus:g0:rate:kappa}), it yields that 
\begin{equation}\label{eq:lemma:S:M:eta:g:eta:Rate:T3}
\|T_3\|=o_P(1) \| g_{\eta}-g_0\|=o_P(\frac{1}{\sqrt{NTh}}+\frac{1}{N\sqrt{h}}+\sqrt{\eta}).
\end{equation}

Secondly, the independence of $\epsilon_i$ and $\mathbb{X}_i, F_1, F_2$ tells us that
\begin{eqnarray*}
T_1=\frac{1}{NT}\sum_{i=1}^N\epsilon_i'P\kxi.
\end{eqnarray*}
Again by the independence assumption and direct calculations, we have
\begin{eqnarray*}
 E(\|T_1\|^2 |\f1t)&=&\frac{1}{N^2T^2}\sum_{i=1}^N E(\epsilon_i'P<\kxi,\kxi>P'\epsilon_i |\f1t)\\
&=& \frac{1}{N^2T^2}\sum_{i=1}^N E(\epsilon_i'P\textbf{K}_i P'\epsilon_i |\f1t)\\
&=& \frac{1}{N^2T^2}\sum_{i=1}^N Tr(E(P \textbf{K}_i P'\epsilon_i\epsilon_i') |\f1t)\\
&=&\frac{\sigma^2_{\epsilon}}{NT^2}E\{Tr(P\textbf{K}_iP')|\f1t\}\\
&\leq& \frac{\sigma^2_{\epsilon}}{NT^2}E\{Tr(\textbf{K}_i)|\f1t\}\\
&=&O_P(\frac{1}{NTh}),
\end{eqnarray*}
where we are using the facts that $\textbf{K}_i=[K(X_{it}, X_{il})]_{1\leq t,l \leq T}$ and $Tr(\textbf{K}_i)\leq Tc_\varphi^2h^{-1}$ derived from (\ref{eq:bound:for:norm:K:sup:norm}).
So it follows
\begin{equation}\label{eq:lemma:S:M:eta:g:eta:Rate:T1}
\|T_1\|=O_P(\frac{1}{\sqrt{NTh}})
\end{equation}

Lastly, we will handle $T_2$ as follows. Since $F_2'P_\star=0$ (see Section \ref{proof:subsec:rate:homo}), it follows that
$$T_2=\frac{1}{NT}\sum_{i=1}^N[\Delta_i F_2'(P-P_{\star})\kxi-E(\Delta_i F_2'P\kxi | \f1t)].$$
By the proof and notation in Lemma \ref{lemma:P:minus:Pstar},
it can be shown that
\begin{eqnarray*}
&&P_\star-P\\
&=&\Sigma'\left((\Sigma\Sigma')^{-1}-(\Sigma_\star\Sigma_\star')^{-1}\right)\Sigma
+\Sigma_\star'(\Sigma_\star\Sigma_\star')^{-1}\widetilde{\Sigma}+
\widetilde{\Sigma}'(\Sigma_\star\Sigma_\star')^{-1}\Sigma_\star+
\widetilde{\Sigma}'(\Sigma_\star\Sigma_\star')^{-1}\widetilde{\Sigma}.
\end{eqnarray*} 
Consequently, $T_2$ has following decomposition:
\begin{eqnarray*}
	&&T_2\\
	&=&\frac{1}{NT}\sum_{i=1}^N[\Delta_i F_2'(\Sigma'\left((\Sigma\Sigma')^{-1}-(\Sigma_\star\Sigma_\star')^{-1}\right)\Sigma)\kxi-E(\Delta_i F_2'(\Sigma'\left((\Sigma\Sigma')^{-1}-(\Sigma_\star\Sigma_\star')^{-1}\right)\Sigma)\kxi | \f1t)]\\
	&&+\frac{1}{NT}\sum_{i=1}^N[\Delta_i F_2'\Sigma_\star'(\Sigma_\star\Sigma_\star')^{-1}\widetilde{\Sigma}\kxi-E(\Delta_i F_2'\Sigma_\star'(\Sigma_\star\Sigma_\star')^{-1}\widetilde{\Sigma}\kxi | \f1t)]\\
	&&+\frac{1}{NT}\sum_{i=1}^N[\Delta_i F_2'\widetilde{\Sigma}'(\Sigma_\star\Sigma_\star')^{-1}\Sigma_\star\kxi-E(\Delta_i F_2'\widetilde{\Sigma}'(\Sigma_\star\Sigma_\star')^{-1}\Sigma_\star\kxi | \f1t)]\\
	&&+\frac{1}{NT}\sum_{i=1}^N[\Delta_i F_2'\widetilde{\Sigma}'(\Sigma_\star\Sigma_\star')^{-1}\widetilde{\Sigma}\kxi-E(\Delta_i F_2'\widetilde{\Sigma}'(\Sigma_\star\Sigma_\star')^{-1}\widetilde{\Sigma}\kxi | \f1t)]\\
	&\equiv&T_{21}+T_{22}+T_{23}+T_{24}.
\end{eqnarray*}
The rest of the proof proceeds to bound the terms $T_{2i}, i=1,2,3,4$.
By (\ref{eq:lemma:P:minus:Pstar:eq1}) in the proof of Lemma \ref{lemma:P:minus:Pstar}, we obtain the following:
\begin{eqnarray*}
E(\|F_2'\Sigma'\left((\Sigma\Sigma')^{-1}-(\Sigma_\star\Sigma_\star')^{-1}\right)\Sigma\|_{\textrm{op}})&=&O(N^{-1/2}+T^{1/2}/N),\\
E(\|F_2'\widetilde{\Sigma}'(\Sigma_\star\Sigma_\star')^{-1}\Sigma_\star\|_{\textrm{op}})&=&O(N^{-1/2}),\\
E(\|F_2'\widetilde{\Sigma}'(\Sigma_\star\Sigma_\star')^{-1}\widetilde{\Sigma}\|_{\textrm{op}})&=&O(1/N).
\end{eqnarray*}
Therefore, it follows that
\begin{eqnarray*}
	&&\|E(\frac{1}{NT}\sum_{i=1}^N \Delta_iF_2'\Sigma'\left((\Sigma\Sigma')^{-1}-(\Sigma_\star\Sigma_\star')^{-1})\right)\Sigma\kxi|\f1t)\|\\
	&\leq&\frac{1}{NT}\sum_{i=1}^NE(\|\Delta_iF_2'\Sigma'\left((\Sigma\Sigma')^{-1}-(\Sigma_\star\Sigma_\star')^{-1})\right)\Sigma \kxi\| |\f1t)\\
	&\leq&\frac{1}{NT}\sum_{i=1}^N\|\Delta_i\|_2E\left(\|F_2'\Sigma'\left((\Sigma\Sigma')^{-1}-(\Sigma_\star\Sigma_\star')^{-1})\right)\Sigma\|_{\textrm{op}} \sqrt{\sum_{t=1}^T\|K_{X_{it}}\|^2} |\f1t \right)\\	
	&\leq&\frac{1}{NT}\sum_{i=1}^N\|\Delta_i\|_2 E\left(\|F_2'\Sigma'\left((\Sigma\Sigma')^{-1}-(\Sigma_\star\Sigma_\star')^{-1})\right)\Sigma\|_{\textrm{op}} \sqrt{Tc_\varphi^2h^{-1}}  |\f1t \right)\\
	&\leq& \frac{1}{\sqrt{T}}\sup_{1\leq i \leq N}\|\Delta_i\|_2\sqrt{c_\varphi^2h^{-1}} E(\|F'\Sigma'\left((\Sigma\Sigma')^{-1}-(\Sigma_\star\Sigma_\star')^{-1}\right)\Sigma\|_{\textrm{op}}|\f1t)\\
	&=&O_P(\frac{1}{\sqrt{NTh}}+\frac{1}{N\sqrt{h}}).
\end{eqnarray*}
As a consequence, $\|T_{21}\|=O_P((NTh)^{-1/2}+N^{-1}h^{-1/2})$. Similarly,
\begin{eqnarray*}
	&&E\{\|E(\frac{1}{NT}\sum_{i=1}^N \Delta_iF_2'\widetilde{\Sigma}'(\Sigma_\star\Sigma_\star')^{-1}\Sigma_\star\kxi|\f1t)\|\}=O_P(\frac{1}{\sqrt{NTh}}),\\
	&&E\{\|E(\frac{1}{NT}\sum_{i=1}^N \Delta_iF_2'\widetilde{\Sigma}'(\Sigma_\star\Sigma_\star')^{-1}\widetilde{\Sigma}\kxi|\f1t)\|\}=O_P(\frac{1}{N\sqrt{Th}}).\\
\end{eqnarray*}
So it follows that $\|T_{23}\|=O_P((NTh)^{-1/2})$ and $\|T_{24}\|=O_P(N^{-1}(Th)^{-1/2})$.
Finally, we will handle $T_{22}$. Let $W=F_2'\Sigma_\star'(\Sigma_\star\Sigma_\star')^{-1}$. It can be easily seen from (\ref{eq:lemma:P:minus:Pstar:eq2}) that $W \in \f1t$ and $\|W\|_{\textrm{op}}=O_P(1)$.
To bound $T_{22}$, notice
$$\widetilde{\Sigma}\kxi=\left(\begin{matrix}
0_{q_1\times T}\cr
\sum_{t=1}^T\bar{v}_tK_{X_{it}}
\end{matrix}\right)=\left(\begin{matrix}
0_{q_1\times T}\cr
\sum_{t=1}^T\bar{v}_{t1}K_{X_{it}}\cr
\sum_{t=1}^T\bar{v}_{t2}K_{X_{it}}\cr
\cdots\cr
\sum_{t=1}^T\bar{v}_{td}K_{X_{it}}
\end{matrix}\right),$$
where $\bar{v}_{ti}$ is the $i$th element of vector $\bar{v}_t$. By direct calculations, it follows that
\begin{eqnarray}
\|T_{22}\|&=&\|\frac{1}{NT}\sum_{i=1}^N\{\Delta_iW\widetilde{\Sigma}\kxi-E(\Delta_iW\widetilde{\Sigma}\kxi|\f1t)\}\|\nonumber\\
&\leq& \frac{1}{NT}\sum_{i=1}^N\|\Delta_iW\widetilde{\Sigma}\kxi-E(\Delta_iW\widetilde{\Sigma}\kxi|\f1t)\|\nonumber \\
&\leq& \frac{1}{NT}\sum_{i=1}^N \|\Delta_i\|_2\|W\|_{\textrm{op}} \sqrt{\sum_{l=1}^d \|\sum_{t=1}^T\left(\bar{v}_{tl}K_{X_{it}}-E(\bar{v}_{tl}K_{X_{it}}|\f1t)\right)\|^2}. \label{eq:lemma:S:M:eta:g:eta:eq1}
\end{eqnarray}
By (\ref{eq:lemma:S:M:eta:g:eta:eq1}), it suffices to find the rate of 
\begin{equation}\label{eq:lemma:S:M:eta:g:eta:eq2}
\frac{1}{NT}\sum_{i=1}^N\sqrt{\sum_{l=1}^d \|\sum_{t=1}^T\left(\bar{v}_{tl}K_{X_{it}}-E(\bar{v}_{tl}K_{X_{it}}|\f1t)\right)\|^2}.
\end{equation}
Because $d$ is finite and fixed, to simplify
our technical arguments, assume $d=1$ without loss of generality.
Direct examinations give the following decomposition:
\begin{eqnarray*}
&&\|\sum_{t=1}^T\left(\bar{v}_{tl}K_{X_{it}}-E(\bar{v}_{tl}K_{X_{it}}|\f1t)\right)\|^2\\
&=&\frac{1}{N^2}\|\sum_{t=1}^T\sum_{j=1}^N\left(v_{jtl}K_{X_{it}}-E(v_{jtl}K_{X_{it}}|\f1t)\right)\|^2\\
&=&\frac{2}{N^2}\|\sum_{t=1}^T\sum_{j\neq i}^N\left(v_{jtl}K_{X_{it}}-E(v_{jtl}K_{X_{it}}|\f1t)\right)\|^2+\frac{2}{N^2}\|\sum_{t=1}^T\left(v_{itl}K_{X_{it}}-E(v_{itl}K_{X_{it}}|\f1t)\right)\|^2\\
&\equiv& T_{221}+T_{222}.
\end{eqnarray*}
When $i\neq j$, $v_{jtl}$ is independent of $X_{it}, \f1t$, so it follows that
\begin{eqnarray*}
&&E\{\|\sum_{t=1}^T\sum_{j\neq i}\left(v_{jtl}K_{X_{it}}-E(v_{jtl}K_{X_{it}}|\f1t)\right)\|^2 |\f1t \}\\
&=&E\{\|\sum_{t=1}^T\sum_{j\neq i}v_{jtl}K_{X_{it}}\|^2 |\f1t \}\\
&=&E\{\sum_{t,t'=1}^T\sum_{j,j'\neq i}v_{jtl}v_{j't'l}K(X_{it},X_{it'}) |\f1t \}\\
&=&\sum_{t=1}^T\sum_{j\neq i}E(v_{jtl}^2)E(K(X_{it},X_{it})\}|\f1t)\\
&\leq &NTE(v_{11l}^2)c_\varphi^2 h^{-1}, \\
\end{eqnarray*}
As a consequence, $T_{221}=O_P(T(Nh)^{-1})$. To deal with $T_{222}$, by Cauchy inequality, it yields that
\begin{eqnarray*}
	E\{\|\sum_{t=1}^T(v_{itl}K_{X_{it}})\|^2 |\f1t \}&\leq & E\{{\sum_{t=1}^T v_{itl}^2} {\sum_{t=1}^T \|K_{X_{it}}\|^2}|\f1t \}\\ 
	&\leq&  E({\sum_{t=1}^T v_{itl}^2}) Tc_\varphi^2 h^{-1}\\
	&=&E(v_{11l}^2)T^2c_\varphi h^{-1},
\end{eqnarray*}
which further implies $T_{222}=O_P(T^2(N^2h)^{-1})$.
By Jensen's inequality and $d=1$, it follows that
\begin{eqnarray}
	(\ref{eq:lemma:S:M:eta:g:eta:eq2})&=&E\left(\frac{1}{NT}\sum_{i=1}^N\sqrt{\|\sum_{t=1}^T\left(\bar{v}_{tl}K_{X_{it}}-E(\bar{v}_{tl}K_{X_{it}}|\f1t)\right)\|^2} | \f1t \right)\nonumber\\
	&\leq &\frac{1}{NT}\sum_{i=1}^N\sqrt{E\left(  \|\sum_{t=1}^T\left(\bar{v}_{tl}K_{X_{it}}-E(\bar{v}_{tl}K_{X_{it}}|\f1t)\right)\|^2|\f1t \right)}\nonumber\\
	&\leq & \sqrt{2c_\varphi^2E(v_{11l}^2)(\frac{1}{NTh}+\frac{1}{N^2h})}\nonumber\\
	&=&O_P(\frac{1}{\sqrt{NTh}}+\frac{1}{N\sqrt{h}}).\label{eq:lemma:S:M:eta:g:eta:eq3}
\end{eqnarray}
Combining (\ref{eq:lemma:S:M:eta:g:eta:eq1}) and (\ref{eq:lemma:S:M:eta:g:eta:eq3}), it follows that $\|T_{22}\|=O_P((NTh)^{-1/2}+(Nh^{1/2})^{-1}).$ As a consequence, we have
\begin{equation} \label{eq:lemma:S:M:eta:g:eta:Rate:T2}
\|T_2\|=O_P(\frac{1}{\sqrt{NTh}}+\frac{1}{N\sqrt{h}}).
\end{equation}
Combining (\ref{eq:lemma:S:M:eta:g:eta:Rate:T3}), (\ref{eq:lemma:S:M:eta:g:eta:Rate:T1}) and (\ref{eq:lemma:S:M:eta:g:eta:Rate:T2}), it yields that
\begin{eqnarray*}
\|S_{M,\eta}(g_{\eta})\|&=&O_P(\frac{1}{\sqrt{NTh}}+\frac{1}{N\sqrt{h}})+o_P(\sqrt{\eta}).
\end{eqnarray*}
Proof completed.
\end{proof}

Next we will prove Lemmas \ref{lemma:ANTm:2:Im}, \ref{lemma:bound:ANT} and \ref{lemma:normality:epsilon:P:Kxi}.
For this purpose, let us introduce a set of notation.
Define $V_{NT\star}, A_{NT\star}, V_{NTm\star}, A_{NTm\star}, H_{NTm\star}$ as follows,
\begin{eqnarray*}
V_{NT\star}=\frac{1}{NT}\sum_{i=1}^N\kxi(x_0)'P_\star\kxi(x_0),   A_{NT\star}=V_{NT\star}^{-1/2},\\ V_{NTm\star}=\frac{1}{NT}\sum_{i=1}^N\phi_m^{\prime}\Phi_i'P_\star\Phi_i\phi_m, A_{NTm\star}=V_{NTm\star}^{-1/2},H_{NTm\star}=\frac{1}{NT}\sum_{i=1}^N\Phi_i' P_\star \Phi_i.\\
\end{eqnarray*} 
\begin{proof}[Proof of Lemma \ref{lemma:ANTm:2:Im}]
Define 
\begin{eqnarray*}
	Q_{i\star}=E(\frac{\Phi_i'P_{\star}\Phi_i}{T}|\f1t), \bar{Q}_{\star}=\frac{1}{N}\sum_{i=1}^N Q_{i\star}, Q_{i}=E(\frac{\Phi_i'P\Phi_i}{T}|\f1t), \bar{Q}=\frac{1}{N}\sum_{i=1}^N Q_{i}=I_m.
\end{eqnarray*} 

Notice that, conditioning on $\f1t$, $\Phi_i$ are independent.
Hence, by Chebyshev's inequality, it follows that
\begin{eqnarray*}
	P(\|H_{NTm\star}-\bar{Q}_{\star}\|_F>\epsilon|\f1t)&=&P(\|\frac{1}{N}\sum_{i=1}^N(\frac{\Phi_i'P_{\star}\Phi_i}{T}-Q_{i\star})\|_F>\epsilon|\f1t)\\
	&\leq& \frac{1}{\epsilon^2N^2}E\{Tr \left([\sum_{i=1}^N(\frac{\Phi_i'P_{\star}\Phi_i}{T}-Q_{i\star})]^2\right) |\f1t\}\\
	&=& \frac{1}{\epsilon^2N^2}\sum_{i=1}^NTr\{ E\left([\frac{\Phi_i'P_{\star}\Phi_i}{T}-Q_{i\star}]^2\right) |\f1t\}\\
	&=&\frac{1}{\epsilon^2N^2}\sum_{i=1}^N E\left(\|\frac{\Phi_i'P_{\star}\Phi_i}{T}-Q_{i\star}\|_F^2 |\f1t\right) \\
	&\leq &\frac{1}{\epsilon^2N^2}\sum_{i=1}^N E\left(\|\frac{\Phi_i'P_{\star}\Phi_i}{T}\|_F^2 |\f1t\right) \\
	&\leq &\frac{1}{\epsilon^2N^2T^2}\sum_{i=1}^N E\left(\|{\Phi_i}\|_F^4 |\f1t\right)\\
	&\leq & \frac{m^2(c_\varphi+1)^4}{\epsilon^2N}.
\end{eqnarray*}
As a consequence, it follows that,
\begin{eqnarray*}
P(\|H_{NTm\star}-\bar{Q}_{\star}\|_F>\frac{\epsilon m (c_\varphi+1)^2}{\sqrt{N}} |\f1t)\leq \frac{1}{\epsilon^2}. 
\end{eqnarray*}
By taking expectation on both sides, we have
\begin{eqnarray*}
P(\|H_{NTm\star}-\bar{Q}_{\star}\|_F>\frac{\epsilon m (c_\varphi+1)^2}{\sqrt{N}})&\leq& \frac{1}{\epsilon^2}.
\end{eqnarray*}
Since $c_\varphi=O_P(1)$, we obtain 
\begin{equation}\label{eq:proof:lemma:ANTm:2:Im:eq1}
\|H_{NTm\star}-\bar{Q}_{\star}\|_F=O_P(mN^{-1/2}).
\end{equation}
By Lemma \ref{lemma:P:minus:Pstar}, we have
\begin{eqnarray}
	E(\|H_{NTm}-H_{NTm\star}\|_F|\f1t)&\leq&\frac{1}{NT}\sum_{i=1}^N E(\|\Phi_i'(P-P_{\star})\Phi_i\|_F|\f1t) \nonumber\\
	&\leq&\frac{1}{NT}\sum_{i=1}^N E(\|(P-P_{\star})\|_{\textrm{op}}\|\Phi_i\|_F^2 |\f1t)\nonumber\\
	&\leq& \frac{1}{NT}\sum_{i=1}^N  mTc_\varphi^2  E(\|(P-P_{\star})\|_{\textrm{op}}|\f1t)\nonumber\\
	&=&O_P(\frac{m}{\sqrt{N}}).\label{eq:proof:lemma:ANTm:2:Im:eq2}
\end{eqnarray}
Again by Lemma \ref{lemma:P:minus:Pstar} and similar calculations, it follows that
\begin{eqnarray}
\|\bar{Q}-\bar{Q}_{\star}\|_F&=&\|E(\frac{1}{NT}\sum_{i=1}^N\Phi_i'(P-P_{\star})\Phi_i |\f1t)\|_F\nonumber\\
&\leq& \frac{1}{NT}\sum_{i=1}^N E(\|\Phi_i'(P-P_{\star})\Phi_i\|_F|\f1t)\nonumber\\
&=&O_P(\frac{m}{\sqrt{N}}).\label{eq:proof:lemma:ANTm:2:Im:eq3}
\end{eqnarray}
Combining (\ref{eq:proof:lemma:ANTm:2:Im:eq1}), (\ref{eq:proof:lemma:ANTm:2:Im:eq2}) and (\ref{eq:proof:lemma:ANTm:2:Im:eq3}), it yields that $$\|H_{NTm}-I_m\|_F=\|H_{NTm}-\bar{Q}\|_F=O_P(\frac{m}{\sqrt{N}})=o_P(1).$$

To the end of the proof, we quantify the minimal and maximal eigenvalues of $H_{NTm}$ as follows.
\begin{eqnarray*}
	\lambda_{\min}(H_{NTm})&=&\min_{\|u\|_2=1}u'H_{NTm}u\\
	&\geq& \min_{\|u\|_2=1}u'I_mu-\min_{\|u\|_2=1}|u'(H_{NTm}-I_m)u|\\
	&=&1-\|H_{NTm}-I_m\|_{\textrm{op}}\\
    &=&1+o_P(1),
\end{eqnarray*}
and
\begin{eqnarray*}
	\lambda_{\max}(H_{NTm})&=&\max_{\|u\|_2=1}u'H_{NTm}u\\
	&\leq& \max_{\|u\|_2=1}u'I_mu+\max_{\|u\|_2=1}|u'(H_{NTm}-I_m)u|\\
	&=&1+\|H_{NTm}-I_m\|_{\textrm{op}}\\
    &=&1+o_P(1),
\end{eqnarray*}
where have used the trivial
inequality $\|H_{NTm}-I_m\|_{\textrm{op}}\leq \|H_{NTm}-I_m\|_{F}=o_P(1)$.
Proof completed.
\end{proof}

\begin{proof}[Proof of Lemma \ref{lemma:bound:ANT}]
By Lemma \ref{lemma:ANTm:2:Im}, we find a lower bound for $V_{NTm}$ and an upper bound for
$A_{NTm}$ as follows:
\begin{eqnarray}\label{eq:proof:lemma:bound:ANT:eq2}
	V_{NTm}&=&\phi_m'H_{NTm}\phi_m\geq \lambda_{\min}(H_{NTm})\|\phi_m\|_2^2\geq \lambda_{\min}(H_{NTm})C,\nonumber\\
	A_{NTm}&=&V_{NTm}^{-1/2}\leq \lambda_{\min}^{-1/2}(H_{NTm})\|\phi_m'\|_2^{-1}=O_P(1)(\sum_{\nu=1}^m\frac{\varphi_\nu^2(x_0)}{(1+\eta\rho_\nu)^2})^{-1/2}=O_P(1).
\end{eqnarray}
Define $L_i(x_0)=\kxi(x_0)-\Phi_i \phi_m$. Then it follows that
$$E(\|L_i\|_2^2|\f1t)\leq Tc_\varphi^4(\sum_{\nu=m+1}^\infty \frac{1}{1+\eta\rho_\nu})^2 \equiv Tc_\varphi^4D_m^2.$$
Directly calculation shows that
\begin{eqnarray}\label{eq:proof:lemma:bound:ANT:eq1}
	|V_{NT}-V_{NTm}|\leq |\frac{2}{NT}\sum_{i=1}^NL_i'P\kxi|+|\frac{1}{NT}\sum_{i=1}^NL_i'PL_i|\equiv
	2|T1|+|T2|.
\end{eqnarray}

Let $R_{x_0}(\cdot)=\sum_{\nu=m+1}^\infty\frac{\varphi_\nu(x_0)\varphi_\nu(\cdot)}{1+\eta \rho_\nu}$. Notice  $L_i=\tau_iR_{x_0}$ and  $E(T_1| \f1t)=V(K_{x_0}, R_{x_0})$. Similar to the proof of Lemma \ref{lemma:ANTm:2:Im}, 
we can show that
\begin{eqnarray*}
	E(|T_1-V(K_{x_0}, R_{x_0})||\f1t)&=&O_P(\frac{D_m}{\sqrt{N}h}).
\end{eqnarray*}
Meanwhile we have the following 
$$V(K_{x_0}, R_{x_0})=\sum_{\nu=m+1}^\infty \frac{\varphi_\nu^2(x_0)}{(1+\eta \rho_\nu)^2}\leq c_\varphi^2D_m.$$
As a consequence, it follows that $|T_1|=O_P(D_m)$. 

A bound for $T_2$ is given by the following inequality,
\begin{eqnarray*}
	E(|T_2||\f1t)\leq \frac{1}{NT}\sum_{i=1}^NE(\|L_i\|_2^2|\f1t)=O_P(D_m^2).
\end{eqnarray*}

So (\ref{eq:proof:lemma:bound:ANT:eq1}) becomes $V_{NT}-V_{NTm}=O_P(D_m)=o_P(1)$. Hence $A_{NT}=A_{NTm}+o_P(1)=O_P(1)$, where last equality is from (\ref{eq:proof:lemma:bound:ANT:eq2}). Proof completed.
\end{proof}

\begin{proof}[Proof of Lemma \ref{lemma:normality:epsilon:P:Kxi}] The proof of this lemma is based on Lyapunov C.L.T.
	Let $c_i=A_{NTm}P\Phi_i \phi_m/(NT)$. We have $$ \sqrt{NT} A_{NTm}(\frac{1}{NT}\sum_{i=1}^N\phi_m' \Phi_i' P \epsilon_i)=\sum_{i=1}^N \sqrt{NT} c_i'\epsilon_i.$$ 
	Since $c_i \in \d1t$ and $\epsilon_i$ is independent of $\d1t$, it follows that
	\begin{eqnarray*}
		E[(\sum_{i=1}^N\sqrt{NT} c_i'\epsilon_i)^2| |\d1t]&=&NT\sigma_{\epsilon}^2\sum_{i=1}^N c_i'c_i\\
		&=& NT\sigma_{\epsilon}^2 A_{NTm}^2\frac{1}{N^2T^2}\sum_{i=1}^N\phi_m' \Phi_i' P \Phi_i \phi_m\\
		&=&\sigma_{\epsilon}^2.
\end{eqnarray*}	 

Let $c_{it}$ be the $t$th element of $c_i$. By direct examinations, it follows that 
\begin{eqnarray}
	\sum_{i=1}^NE[(\sqrt{NT} c_i'\epsilon_i)^4| |\d1t]&=&N^2T^2\sum_{i=1}^N\sum_{t=1}^T c_{it}^4E(\epsilon_{it}^4)\nonumber\\&&+3N^2T^2\sum_{i=1}^N\sum_{t=1}^T\sum_{t'\neq t}c_{it}^2c_{it'}^2E(\epsilon_{it}^2\epsilon_{it'}^2).\label{eq:proof:lemma:normality:epsilon:P:Kxi:eq3}
\end{eqnarray}
Next we are going to find a bound for $c_{it}$. By direct calculation, we have
\begin{eqnarray*}
	|c_{it}|&=& |A_{NTm}\frac{1}{NT}p_{t\cdot}\Phi_i \phi_m|\\
	&\leq& \|A_{NTm} \phi_m\|_2 \|\frac{1}{NT}\sum_{s=1}^Tp_{ts}\Phi_{i,s\cdot}\|_2\\
	&\leq& \lambda_{\min}^{-1/2}(H_{NTm})\frac{1}{NT}\|\sum_{s=1}^Tp_{ts}\Phi_{i,s\cdot}\|_2,
\end{eqnarray*}
where $p_{t\cdot}$ is the $t$th row of $P$, $p_{ts}$ is the $(t, s)$th element of $P$ and 
$\Phi_{i,s\cdot}$ is the $s$th row of $\Phi_i$.
Meanwhile, $p_{ts}=\delta_{ts}-Z_s'(\Sigma \Sigma')^{-1}Z_t$, hence
\begin{eqnarray*}
	\|\sum_{s=1}^Tp_{ts}\Phi_{i,s\cdot}\|_2&=&\|\Phi_{i, t\cdot}-\frac{1}{T}Z_t' (\frac{\Sigma\Sigma'}{T})^{-1}\sum_{s=1}^TZ_s \Phi_{i,s\cdot}\|_2\\
	&\leq& \|\Phi_{i, t\cdot}\|_2+\|Z_t\|_2\|(\frac{\Sigma\Sigma'}{T})^{-1}\|_{\textrm{op}}\sqrt{\frac{1}{T}\sum_{s=1}^T\|Z_s\|_2^2}\sqrt{\frac{1}{T}\sum_{s=1}^T\|\Phi_{i,s\cdot}\|_2^2}\\
	&\leq&\sqrt{mc_\varphi^2}+\|Z_t\|_2\|(\frac{\Sigma\Sigma'}{T})^{-1}\|_{\textrm{op}}\sqrt{\frac{1}{T}\sum_{s=1}^T\|Z_s\|_2^2}\sqrt{mc_\varphi^2}\\
	&\leq &\sqrt{mc_\varphi^2}(1+b\|Z_t\|_2),
\end{eqnarray*}
where $b=\|(\frac{\Sigma\Sigma'}{T})^{-1}\|_{\textrm{op}}\sqrt{\frac{1}{T}\sum_{s=1}^T\|Z_s\|_2^2}=O_P(1)$ by Assumption \ref{A7}. So $|c_{it}|\leq a(1+b\|Z_t\|_2)$, where $a=\lambda_{\min}^{-1/2}(H_{NTm})\frac{1}{NT}\sqrt{mc_\varphi^2}$. By Lemma \ref{lemma:ANTm:2:Im}, we have
\begin{eqnarray}\label{eq:proof:lemma:normality:epsilon:P:Kxi:eq1}
	\sum_{i=1}^N\sum_{t=1}^T c_{it}^4\leq \sum_{i=1}^N\sum_{t=1}^T 8a^4(1+b^4\|Z_t\|_2^4)=O_P(\frac{m^2}{N^3T^3}),
\end{eqnarray}
and 
\begin{eqnarray}
	\sum_{i=1}^N\sum_{t=1}^T\sum_{t'\neq t}c_{it}^2c_{it'}^2&\leq&\sum_{i=1}^N \sqrt{\sum_{t=1}^T\sum_{t'\neq t}c_{it}^4}\sqrt{\sum_{t=1}^T\sum_{t'\neq t}c_{it'}^4}\nonumber\\
		&\leq& T\sum_{i=1}^N\sum_{t=1}^T c_{it}^4\nonumber\\
		&=&O_P(\frac{m^2}{N^3T^2}).\label{eq:proof:lemma:normality:epsilon:P:Kxi:eq2}
\end{eqnarray}
Combining (\ref{eq:proof:lemma:normality:epsilon:P:Kxi:eq3}), (\ref{eq:proof:lemma:normality:epsilon:P:Kxi:eq1}) and (\ref{eq:proof:lemma:normality:epsilon:P:Kxi:eq2}), we have
$\sum_{i=1}^NE[(\sqrt{NT} c_i'\epsilon_i)^4| |\d1t]=O_P(m^2/N)$. 
And by Lyapunov C.L.T, the result follows. Proof completed.
\end{proof}
\end{document}